\documentclass[a4paper,10pt]{article}
\usepackage[utf8]{inputenc}
\usepackage[english]{babel}
\usepackage{graphicx}
\usepackage{fancyhdr}
\usepackage{geometry}
\usepackage{amsfonts}
\usepackage{amssymb}
\usepackage{amsthm}
\usepackage{amsmath}
\usepackage{amsfonts}
\usepackage{amssymb}
\usepackage{mathtools}
\usepackage{enumitem}
\usepackage[percent]{overpic}
\usepackage{tikz}
\usepackage{mathrsfs}
\usepackage{ifthen,tabularx,graphicx,multirow}
\usepackage{xargs}
\usepackage[colorinlistoftodos,prependcaption,textsize=tiny]{todonotes}
\newcommandx{\at}[2][1=]{\todo[linecolor=red,backgroundcolor=red!25,bordercolor=red,#1]{#2}}
\numberwithin{equation}{section}
\newtheorem{theorem}{Theorem}

\newtheorem{example}{Example}
\newtheorem{lemma}{Lemma}
\newtheorem{corollary}{Corollary}
\newtheorem{proposition}{Proposition}
\newtheorem{definition}{Definition}
\numberwithin{theorem}{section}
\numberwithin{lemma}{section}
\numberwithin{corollary}{section}
\numberwithin{proposition}{section}
\numberwithin{definition}{section}
\numberwithin{example}{section}
\usepackage{algorithm}
\usepackage{hyperref}
\usepackage{algpseudocode}
\makeatletter 
\newcounter{algorithmicH}
\let\oldalgorithmic\algorithmic
\renewcommand{\algorithmic}{
  \stepcounter{algorithmicH}
  \oldalgorithmic}
\renewcommand{\theHALG@line}{ALG@line.\thealgorithmicH.\arabic{ALG@line}}
\makeatother

\usepackage{mathtools,booktabs}
\DeclarePairedDelimiter\ceil{\lceil}{\rceil}

\usepackage{varwidth}

\usepackage[capitalise]{cleveref}
\usepackage{cite}
\definecolor{rev1}{HTML}{cb270f}
\definecolor{rev2}{HTML}{1c8235}

\title{Rigorous data-driven computation of spectral properties of Koopman operators for dynamical systems}
\author{Matthew J. Colbrook\thanks{Department of Applied Mathematics and Theoretical Physics, University of Cambridge, Cambridge, CB3 0WA. ({m.colbrook@damtp.cam.ac.uk})} \and Alex Townsend\thanks{Department of Mathematics, Cornell University, Ithaca, NY  14853. ({townsend@cornell.edu})}}
\date{}
\begin{document}
\maketitle
 
\begin{abstract}
Koopman operators are infinite-dimensional operators that globally linearize nonlinear dynamical systems, making their spectral information valuable for understanding dynamics. However, Koopman operators can have continuous spectra and infinite-dimensional invariant subspaces, making computing their spectral information a considerable challenge. This paper describes data-driven algorithms with rigorous convergence guarantees for computing spectral information of Koopman operators from trajectory data. We introduce residual dynamic mode decomposition (ResDMD), which provides the first scheme for computing the spectra and pseudospectra of general Koopman operators from snapshot data without spectral pollution. Using the resolvent operator and ResDMD, we compute smoothed approximations of spectral measures associated with general measure-preserving dynamical systems. We prove explicit convergence theorems for our algorithms, which can achieve high-order convergence even for chaotic systems when computing the density of the continuous spectrum and the discrete spectrum. Since our algorithms come with error control, ResDMD allows aposteri verification of spectral quantities, Koopman mode decompositions, and learned dictionaries. We demonstrate our algorithms on the tent map, circle rotations, Gauss iterated map, nonlinear pendulum, double pendulum, and Lorenz system. Finally, we provide kernelized variants of our algorithms for dynamical systems with a high-dimensional state space. This allows us to compute the spectral measure associated with the dynamics of a protein molecule with a 20,046-dimensional state space and compute nonlinear Koopman modes with error bounds for turbulent flow past aerofoils with Reynolds number $>10^5$ that has a 295,122-dimensional state space.
\end{abstract}

\begin{keywords}
Dynamical systems, Koopman operator, Data-driven discovery, Dynamic mode decomposition, Spectral theory
\end{keywords}

\begin{AMS}
37M10, 65P99, 65F99, 65T99, 37A30, 47A10, 47B33, 37N10, 37N25 
\end{AMS}

\section{Introduction}
Dynamical systems are a mathematical description of time-dependent states or quantities that characterize evolving processes in classical mechanics, electrical circuits, fluid flows, climatology, finance, neuroscience, epidemiology, and many other fields. Throughout the paper, we consider autonomous dynamical systems whose state evolves over a state-space $\Omega\subseteq\mathbb{R}^d$ in discrete time-steps according to a function $F:\Omega \rightarrow \Omega$. In other words, we consider dynamical systems of the form 
\begin{equation} 
\pmb{x}_{n+1} = F(\pmb{x}_n), \qquad n\geq 0, 
\label{eq:DynamicalSystem} 
\end{equation} 
where $\pmb{x}_0$ is a given initial condition. Such a dynamical system forms a trajectory of iterates $\pmb{x}_0,\pmb{x}_1,\pmb{x}_2,\ldots$ in $\Omega$. We want to analyze such trajectories to answer questions about the system's behavior. The interaction between numerical analysis and dynamical systems theory has stimulated remarkable growth in the subject since the 1960s~\cite{kalman1963mathematical,lorenz1963deterministic,epstein1969stochastic,stuart1998dynamical}. With the arrival of big data \cite{hey2009fourth}, modern statistical learning \cite{friedman2017elements}, and machine learning \cite{mohri2018foundations}, data-driven algorithms are now becoming increasingly important in understanding dynamical systems~\cite{schmidt2009distilling,brunton2019data}.

A classical viewpoint to analyze dynamical systems that originates in the seminal work of Poincar\'{e}~\cite{poincare1899methodes} is to study fixed points and periodic orbits, as well as stable and unstable manifolds. Two fundamental challenges with Poincar\'{e}'s geometric state-space viewpoint are: 
\begin{itemize}[leftmargin=*]

\item \textbf{Nonlinear dynamics:} To understand the stability of fixed points of nonlinear dynamical systems, one typically forms local models centered at these fixed points. Such models allow one to predict long-time dynamics in small neighborhoods of fixed points and attracting manifolds. However, they do not provide reasonable predictions for all initial conditions. A global understanding of nonlinear dynamics in state-space remains largely qualitative~\cite{budivsic2012applied}. For example, Brunton and Kutz describe the global understanding of nonlinear dynamics in state-space as a ``\textit{mathematical grand challenge of the 21st century}"~\cite{brunton2019data}. This paper focuses on the Koopman operator viewpoint of dynamical systems. Koopman operators provide a global linearization of~\eqref{eq:DynamicalSystem} by studying the evolution of observables of the system. 
	
\item \textbf{Unknown dynamics:} For many applications, a system's dynamics may be too complicated to describe analytically, or we may have incomplete knowledge of its evolution. Typically, we can only acquire several sequences of iterates of~\eqref{eq:DynamicalSystem} starting at different values of $\pmb{x}_0$. This constraint means that constructing local models can be impossible. In this paper, we focus on data-driven approaches to learning and analyzing the dynamical system with trajectories of iterates from~\eqref{eq:DynamicalSystem}. 

\end{itemize}

Koopman operator theory, which dates back to Koopman and von Neumann~\cite{koopman1932dynamical,koopman1931hamiltonian}, is an alternative viewpoint to analyze a dynamical system that uses the space of scalar observable functions~\cite{mezicAMS}. Its increasing popularity has led to the term ``Koopmanism''~\cite{budivsic2012applied} and thousands of articles over the last decade. Sparked by \cite{mezic2005spectral,mezic2004comparison}, a reason for the recent attention is its use in data-driven methods for studying dynamical systems (see~\cite{brunton2021modern} for a review and the history). Some popular applications include fluid dynamics~\cite{schmid2010dynamic,rowley2009spectral,mezic2013analysis,giannakis2018koopman}, epidemiology~\cite{proctor2015discovering}, neuroscience~\cite{brunton2016extracting}, finance~\cite{mann2016dynamic}, robotics~\cite{berger2015estimation,bruder2019modeling}, power grids~\cite{susuki2011coherent,susuki2011nonlinear}, and molecular dynamics\cite{nuske2014variational,klus2018data,schwantes2015modeling,schwantes2013improvements}. 

Let $g:\Omega\rightarrow\mathbb{C}$ be a function that one can use to indirectly measure the dynamical system's state in~\eqref{eq:DynamicalSystem}. Such a function $g$ is often called an observable. Therefore, $g(\pmb{x}_k)$ indirectly measures $\pmb{x}_k$ and $g(\pmb{x}_{k+1}) = g(F(\pmb{x}_k))$ measures the state one time-step later. Since we are interested in how the state evolves, we can define an observable $\mathcal{K}_g: \Omega \rightarrow \mathbb{C}$ such that $\mathcal{K}_g(\pmb{x}) = g(F(\pmb{x})) = (g\,\circ\, F)(\pmb{x})$, which tells us how our measurement evolves from one time-step to the next. In this way, $\mathcal{K}_g$ indirectly monitors the evolution of~\eqref{eq:DynamicalSystem}. After exhaustively using $g$ to measure the dynamical system at various time steps, we might want to probe the dynamical system using a different observable. As one tries different observables to probe the dynamics, one wonders if there is a principled way to select ``good" observables. This motivates studying the Koopman operator given by $\mathcal{K}: g \mapsto \mathcal{K}_g$. One typically works in the Hilbert space $L^2(\Omega,\omega)$ of observables for a positive measure $\omega$ on $\Omega$.\footnote{We do not assume that this measure is invariant, and the most common choice of $\omega$ is the standard Lebesgue measure. This choice is natural for Hamiltonian systems for which the Koopman operator is unitary on $L^2(\Omega,\omega)$. For other systems, we can select $\omega$ according to the region where we wish to study the dynamics, such as a Gaussian measure.} To be precise, we consider $\mathcal{K} : \mathcal{D}(\mathcal{K})\rightarrow L^2(\Omega,\omega)$, where $\mathcal{D}(\mathcal{K}) \subseteq L^2(\Omega,\omega)$ is a suitable domain of observables. That is,
\begin{equation} 
[\mathcal{K}g](\pmb{x}) = (g\circ F)(\pmb{x}), \qquad \pmb{x}\in\Omega, \qquad g\in \mathcal{D}(\mathcal{K}), 
\label{eq:KoopmanOperator} 
\end{equation}
where the equality is understood in the $L^2(\Omega,\omega)$ sense. For a fixed $g$, $[\mathcal{K}g](\pmb{x})$ measures the state of the dynamical system after one time-step if it is currently at $\pmb{x}$. On the other hand, for a fixed $\pmb{x}$,~\eqref{eq:KoopmanOperator} is a linear composition operator with $F$. Since $\mathcal{K}$ is a linear operator, regardless of whether the dynamics are linear or nonlinear, the spectral information of $\mathcal{K}$ determines the behavior of the dynamical system \eqref{eq:DynamicalSystem} (e.g., see~\eqref{koopman_mode_decomp_re}). However, since $\mathcal{K}$ is an infinite-dimensional operator, its spectral information can be far more complicated than that of a finite matrix, and far more difficult to compute~\cite{colbrookthesis,colbrook2022computation,colbrook2022foundations,SCI_ref,colbrook2019infinite}. For example, $\mathcal{K}$ can have both discrete\footnote{Throughout this paper, we use the term ``discrete spectra'' to mean the eigenvalues of $\mathcal{K}$, also known as the point spectrum. This also includes embedded eigenvalues, in contrast to the usual definition of discrete spectrum.} and continuous spectra~\cite{mezic2005spectral}.

Computing the spectral properties of $\mathcal{K}$ is an active area of research, and several popular algorithms exist.\footnote{Computing spectral properties of Perron–Frobenius or transfer operators is another area of active research~\cite{dellnitz1999approximation,froyland1997computer,froyland2007ulam,junge2009discretization}. On appropriate spaces, which are often not Hilbert spaces, the Perron–Frobenius operator is the adjoint of the Koopman operator. One can extend some of our notions to Banach spaces~\cite{seidel2012n}; however, a difficulty is that the approximation of finite matrix norms can be NP-hard~\cite{hendrickx2010matrix}.} Generalized Laplace analysis, which is related to the power method, uses prior information about the eigenvalues of $\mathcal{K}$ to compute corresponding Koopman modes.\footnote{These are vectors to reconstruct the system's state as a linear combination of the candidate Koopman eigenfunctions.} Dynamic mode decomposition (DMD), which is related to the proper orthogonal decomposition, uses a linear model to fit trajectory data~\cite{schmid2010dynamic,tu2014dynamic,rowley2009spectral,kutz2016dynamic}, and extended DMD (EDMD)~\cite{williams2015data,williams2015kernel,klus2020data} is a Galerkin approximation of $\mathcal{K}$ that uses a rich dictionary of observables (see \cref{sec:basic_EDMD}). Other data-driven methods include deep learning~\cite{li2017extended,yeung2019learning,lusch2018deep,mardt2018vampnets,otto2019linearly}, reduced-order modeling~\cite{guo2019data,benner2015survey}, sparse identification of nonlinear dynamics~\cite{brunton2016discovering,rudy2017data}, and kernel analog forecasting~\cite{burov2021kernel,giannakis2021learning,zhao2016analog}. However, remaining challenges include the following:
\begin{itemize}[leftmargin=*]

\item[C1.] \textbf{Continuous spectra:} Since the Koopman operator is infinite-dimensional, it can have a continuous spectrum. Continuous spectra must be treated with considerable care as discretizing $\mathcal{K}$ destroys its presence. Several data-driven approaches are proposed to handle continuous spectra, such as HAVOK analysis~\cite{brunton2017chaos}, which applies DMD to a vector of time-delayed measurements, and feed-forward neural networks with an auxiliary network to parametrize the continuous spectrum~\cite{lusch2018deep}. While these methods are promising, they currently lack convergence guarantees. Most existing nonparametric approaches for computing continuous spectra of $\mathcal{K}$ are restricted to ergodic systems~\cite{giannakis2019data}, as this allows relevant integrals to be computed using long-time averages. For example,~\cite{arbabi2017study} approximates discrete spectra using harmonic averaging, applies an ergodic cleaning process, and then estimates the continuous spectrum using Welch's method~\cite{welch1967use}. Though these methods are related to well-studied techniques in signal processing~\cite{ghil2002advanced}, they can be challenging to apply in the presence of noise or if a pair of eigenvalues are close together. Moreover, they often rely on heuristic parameter choices and cleanup procedures. Other approaches for ergodic systems include approximating the functional calculus using integral operator compact regularizations \cite{das2021reproducing}, and approximating the moments of the spectral measure using ergodicity followed by separating the atomic and continuous parts of the spectrum using the Christoffel--Darboux kernel~\cite{korda2020data}. In this paper, we do not assume ergodicity, and we develop a computational framework that deals jointly with continuous and discrete spectra via computing smoothed approximations of spectral measures with explicit high-order convergence rates.

\item[C2.] \textbf{Invariant subspaces:} A finite-dimensional invariant subspace of $\mathcal{K}$ is a space of observables $\mathcal{G} = {\rm span}\!\left\{g_1,\ldots,g_k\right\}$ such that $\mathcal{K}g\in \mathcal{G}$ for all $g\in\mathcal{G}$. A common assumption in the literature is that $\mathcal{K}$ has a finite-dimensional nontrivial\footnote{If $\omega(\Omega)<\infty$, the constant function $1$ generates an invariant subspace, but this is not interesting for studying dynamics.} invariant space, which may not be the case (e.g., when the system is mixing). Even if a (nontrivial) finite-dimensional invariant subspace exists, it can be challenging to compute or may not capture all of the dynamics of interest. Often, one must settle for approximate invariant subspaces, and methods such as EDMD generally result in closure issues~\cite{brunton2016koopman}. In this paper, we develop methods that directly compute spectral properties of $\mathcal{K}$ instead of restrictions of $\mathcal{K}$ to finite-dimensional subspaces.

\item[C3.] \textbf{Spectral pollution:} A well-known difficulty of computing spectra of infinite-dimensional operators is spectral pollution, where discretizations cause spurious eigenvalues that have nothing to do with the operator~\cite{lewin2010spectral,colbrook2019compute,colbrook2023avoiding}. Methods such as EDMD suffer from spectral pollution~\cite{williams2015data}, and heuristics are common to reduce a user's concern. One can compare eigenvalues computed using different discretization sizes or verify that a candidate eigenfunction behaves linearly on the trajectories of \eqref{eq:DynamicalSystem} in the way that the corresponding eigenvalue predicts~\cite{kaiser2017data}. In some cases, it is possible to approximate the spectral information of a Koopman operator without spectral pollution. For example, when performing a Krylov subspace method on a finite-dimensional invariant subspace of known dimension, the Hankel-DMD algorithm computes the corresponding eigenvalues for ergodic systems~\cite{arbabi2017ergodic}. Instead, it is highly desirable to have a principled way of detecting spectral pollution with as few assumptions as possible. In this paper, we develop methods to compute residuals associated with the spectrum with error control, allowing us to compute spectra of general Koopman operators without spectral pollution.

\item[C4.] \textbf{Chaotic behavior:} Many dynamical systems exhibit rich, chaotic behavior. If a system is chaotic, small perturbations to the initial state $\pmb{x}_0$ can lead to large perturbations of the state later. Koopman operators with continuous spectra are a generic feature when the underlying dynamics are chaotic~\cite{basley2011experimental,arbabi2017study}. Chaotic behavior can make data-driven approaches challenging when using long-time trajectories.\footnote{Many dynamical systems, such as hyperbolic systems with chaotic dynamics, satisfy shadowing theorems, meaning that numerical trajectories remain close to true trajectories. Under suitable conditions, we can still apply quadrature rules such as~\eqref{quad_comp222}.} In this paper, we use one-step trajectory data, called snapshot data, that can take the form of long or short trajectories depending on the system and available data.

\item{C5.} 	\textbf{Nonlinearity and high-dimensional state-space:} For many dynamical systems, the corresponding dynamical system is strongly nonlinear and has a very large state-space dimension. This can make choosing a good set of observables challenging. Often this choice takes the form of a learned dictionary. In this paper, our algorithms come with verification (i.e., error bounds), which allows us to perform a-posteriori verification that the learned dictionary was reasonable.

\end{itemize} 
In response to these challenges, we provide rigorously convergent and data-driven algorithms to compute the spectral properties of Koopman operators. Throughout the paper, we only assume that $\mathcal{K}$ is a closed and densely defined operator, which is an essential assumption before talking about spectral information of $\mathcal{K}$; otherwise, the spectrum is the whole of $\mathbb{C}$. In particular, we do not assume that $\mathcal{K}$ has a nontrivial finite-dimensional invariant subspace or that it only has a discrete spectrum.

\subsection{Novel contributions}

To tackle $\textbf{C1}$, we give a computational framework for computing the spectral measures of Koopman operators associated with measure-preserving dynamical systems. Our algorithms can achieve arbitrarily high-order convergence. We deal with discrete and continuous spectra of a Koopman operator $\mathcal{K}$ by calculating them together, and we prove rigorous convergence results:
\begin{itemize}[noitemsep]\setlength{\itemindent}{-1em}
	\item Pointwise recovery of spectral densities of the continuous part of the spectrum of $\mathcal{K}$ (see~\cref{thm:unitary_pointwise_convergence}).
	\item Weak convergence, involving integration against test functions (see~\cref{thm:unitary_weak_convergence}).
	\item The recovery of the eigenvalues of $\mathcal{K}$ (see~\cref{atom_theorem}).
\end{itemize}
\cref{alg:spec_meas_poly} is based on estimating autocorrelations and filtering, which requires long-time trajectory data from~\eqref{eq:DynamicalSystem}. Using a connection with the resolvent operator, we develop an alternative rational-based approach that can achieve high-order convergence on short-time trajectory data (see~\cref{alg:spec_meas_rat}) while being robust to noise (see~\cref{sec:stability_argument}).

To tackle $\textbf{C2}$ and $\textbf{C3}$, we derive a new DMD-type algorithm called residual DMD (ResDMD). ResDMD constructs Galerkin approximations for not only $\mathcal{K}$ but also $\mathcal{K}^*\mathcal{K}$. 
This key difference (which requires no extra data) allows us to rigorously compute spectra of general Koopman operators (without the measure-preserving assumption) together with residuals. ResDMD goes a long way to overcoming challenges \textbf{C2} and \textbf{C3} by allowing us to:
\begin{itemize}[noitemsep]\setlength{\itemindent}{-1em}
	\item Remove the spectral pollution of extended DMD (see~\cref{alg:mod_EDMD,triv_prop}).
	\item Compute spectra and pseudospectra with convergence guarantees (see~\cref{alg:res_EDMD,sec:computing_spectra_limits}).
\end{itemize}
It is important to stress that ResDMD computes residuals associated with the underlying infinite-dimensional operator $\mathcal{K}$, in contrast to earlier work that computes residuals of observables with respect to a finite DMD discretization matrix~\cite{drmac2018data}. However, in contrast to ResDMD, residuals with respect to a finite DMD discretization cannot give error bounds on the spectral information of $\mathcal{K}$ and can suffer from issues such as spectral pollution.

All of our algorithms can be used for chaotic systems, and we do not run into the difficulties of $\textbf{C4}$. We also describe a kernelized variant of ResDMD (see~\cref{alg:kern_algs}). We first use the kernelized EDMD~\cite{williams2015kernel} to learn a suitable dictionary of observables from a subset of the given data, and then we employ this dictionary in \cref{alg:mod_EDMD,alg:res_EDMD,alg:spec_meas_rat}. Compared with traditional kernelized approaches, our key advantage is that we have convergence theory and perform a posterior verification that the learned dictionary is suitable, allowing us to tackle $\textbf{C5}$.

Code for ResDMD and the examples of this paper is provided at \textcolor[rgb]{0,0,1}{\url{https://github.com/MColbrook/Residual-Dynamic-Mode-Decomposition}}.

\subsection{Paper structure} 
In~\cref{sec:background}, we provide background material to introduce spectral measures and pseudospectra of operators. In~\cref{sec:section3_filter_fin}, we describe data-driven algorithms for computing the spectral measures of Koopman operators associated with measure-preserving dynamics when one is given long-time trajectory data. In~\cref{sec:RES_DMD}, we present ResDMD for computing spectral properties of general Koopman operators. In~\cref{res_kern_fin}, we compute spectral measures of $\mathcal{K}$ associated with measure-preserving dynamics using snapshot data. Finally, in~\cref{sec:LARGEDIM}, we develop a kernelized approach and apply our algorithms to dynamical systems with high-dimensional state-spaces. Throughout, we use $\langle\cdot,\cdot\rangle$ and $\|\cdot\|$ to denote the inner product and norm corresponding to $L^2(\Omega,\omega)$, respectively. We use $\sigma(T)$ to denote the spectrum of a linear operator $T$ and $\|T\|$ to denote its operator norm if it is bounded.  

\section{Spectral measures, spectra and pseudospectra}\label{sec:background} 
We are interested in the spectral measures, spectra, and pseudospectra of Koopman operators. This section introduces the relevant background for these spectral properties.

\subsection{Spectral measures for measure-preserving dynamical systems}\label{sec:spec_meas_def_jkhkhjk}
Suppose that the associated dynamics is measure-preserving so that $\omega(E) = \omega\left( \left\{\pmb{x}: F(\pmb{x})\in E\right\} \right)$ for any Borel measurable subset $E\subset\Omega$. Equivalently, this means that the Koopman operator $\mathcal{K}$ associated with the dynamical system in~\eqref{eq:DynamicalSystem} is an isometry, i.e., $\|\mathcal{K}g\|=\|g\|$ for all observables $g\in \mathcal{D}(\mathcal{K})=L^2(\Omega,\omega)$. Dynamical systems such as Hamiltonian flows~\cite{arnold1989mathematical}, geodesic flows on Riemannian manifolds~\cite[Chapter 5]{dubrovin2012modern}, Bernoulli schemes in probability theory~\cite{shields1973theory}, and ergodic systems~\cite{walters2000introduction} are all measuring-preserving. Moreover, many dynamical systems become measure-preserving in the long run~\cite{mezic2005spectral}. 

Spectral measures provide a way of diagonalizing normal operators, including self-adjoint and unitary operators, even in the presence of continuous spectra. Unfortunately, a Koopman operator that is an isometry does not necessarily commute with its adjoint (see~\cref{sec:tent_map}). Therefore, we must consider a unitary extension (see~\cref{sec:unitaryExtension}) before defining a spectral measure (see~\cref{sec:spec_meas_cric_def}) and Koopman mode decomposition (see~\cref{sec:cts_disc_KMD}).

\subsubsection{Unitary extensions of isometries}\label{sec:unitaryExtension} 
Given a Koopman operator $\mathcal{K}$ of a measure-preserving dynamical system, we use the concept of unitary extension to formally construct a related normal operator $\mathcal{K}'$. That is, suppose that $\mathcal{K}:L^2(\Omega,\omega)\rightarrow L^2(\Omega,\omega)$ is an isometry, then there exists a unitary extension $\mathcal{K}'$ defined on an extended Hilbert space $\mathcal{H}'$ with $L^2(\Omega,\omega)\subset\mathcal{H}'$~\cite[Proposition I.2.3]{nagy2010harmonic}.\footnote{To see how to extend $\mathcal{K}$ to a unitary operator $\mathcal{K}'$, consider the Wold--von Neumann decomposition~\cite[Theorem I.1.1]{nagy2010harmonic}. This decomposition states that $\mathcal{K}$ can be written as $\mathcal{K}=(\oplus _{\alpha \in I}S_\alpha)\oplus U$ for some index set $I$, where $S_\alpha$ is the unilateral shift on a Hilbert space $\mathcal{H}_\alpha$ and $U$ is a unitary operator. Since one can extend any unilateral shift to a unitary bilateral shift, one can extend $\mathcal{K}$ to a unitary operator $\mathcal{K}'$.} Even though such an extension is not unique, it allows us to understand the spectral information of $\mathcal{K}$ by considering $\mathcal{K}'$, which is a normal operator. If $F$ is invertible and measure-preserving, $\mathcal{K}$ is unitary and we can simply take $\mathcal{K}'=\mathcal{K}$ and $\mathcal{H}'=L^2(\Omega,\omega)$.

\subsubsection{Spectral measures of normal operators}\label{sec:unitarySpectralMeasure}
To explain the notion of spectral measures for a normal operator, we begin with the more familiar finite-dimensional case. Any linear operator acting on a finite-dimensional Hilbert space has a purely discrete spectrum consisting of eigenvalues. In particular, the spectral theorem for a normal matrix $A\in \mathbb{C}^{n\times n}$, i.e., $A^*A = AA^*$, states that there exists an orthonormal basis of eigenvectors $v_1,\dots,v_n$ for $\mathbb{C}^n$ such that
\begin{equation}\label{eqn:disc_decomp}
v = \left(\sum_{k=1}^n v_kv_k^*\right)v, \quad v\in\mathbb{C}^n \qquad\text{and}\qquad Av = \left(\sum_{k=1}^n\lambda_k v_kv_k^*\right)v, \quad v\in\mathbb{C}^n,
\end{equation}
where $\lambda_1,\ldots,\lambda_n$ are eigenvalues of $A$, i.e., $Av_k = \lambda_kv_k$ for $1\leq k\leq n$. In other words, the projections $v_kv_k^*$ simultaneously decompose the space $\mathbb{C}^n$ and diagonalize the operator $A$. This intuition carries over to the infinite-dimensional setting by replacing $v\in\mathbb{C}^n$ by $f\in\mathcal{H}'$, and $A$ by a normal operator $\mathcal{K}'$. However, the eigenvectors of $\mathcal{K}'$ need not form a basis for $\mathcal{H}'$ or diagonalize $\mathcal{K}'$. Instead, the spectral theorem for normal operators states that the projections $v_kv_k^*$ in~\eqref{eqn:disc_decomp} can be replaced by a projection-valued measure $\mathcal{E}$ supported on the spectrum of $\mathcal{K}'$~\cite[Thm.~VIII.6]{reed1972methods}. In our setting, $\mathcal{K}'$ is unitary; hence, its spectrum is contained inside the unit circle $\mathbb{T}$. The measure $\mathcal{E}$ assigns an orthogonal projector to each Borel measurable subset of $\mathbb{T}$ such that
\[
f=\left(\int_\mathbb{T} d\mathcal{E}(y)\right)f \qquad\text{and}\qquad \mathcal{K}'f=\left(\int_\mathbb{T} y\,d\mathcal{E}(y)\right)f, \qquad f\in\mathcal{H}'.
\]
Analogous to~\eqref{eqn:disc_decomp}, $\mathcal{E}$ decomposes $\mathcal{H}'$ and diagonalizes the operator $\mathcal{K}'$. For example, if $U\subset\mathbb{T}$ contains only discrete eigenvalues of $\mathcal{K}'$ and no other types of spectra, then $\mathcal{E}(U)$ is simply the spectral projector onto the invariant subspace spanned by the corresponding eigenfunctions. More generally, $\mathcal{E}$ decomposes elements of $\mathcal{H}'$ along the discrete and continuous spectrum of $\mathcal{K}'$ (see \cref{sec:cts_disc_KMD}).

\subsubsection{Spectral measures of Koopman operators}\label{sec:spec_meas_cric_def}
Given an observable $g\in L^2(\Omega,\omega)\subset\mathcal{H}'$ of interest such that $\|g\| = 1$, the spectral measure of $\mathcal{K}'$ with respect to $g$ is a scalar measure defined as $\mu_{g}(U):=\langle\mathcal{E}(U)g,g\rangle$, where $U\subset\mathbb{T}$ is a Borel measurable set~\cite{reed1972methods}. For plotting and visualization, it is more convenient to equivalently consider the corresponding probability measures $\nu_g$ defined on the periodic interval $[-\pi,\pi]_{\mathrm{per}}$ after a change of variables $\lambda=\exp(i\theta)$ so that $d\mu_g(\lambda)=d\nu_g(\theta)$. Therefore, throughout this paper, we compute and visualize $\nu_g$. We use the notation \smash{$\int_{[-\pi,\pi]_{\mathrm{per}}}$} to denote integration along the periodic interval $[-\pi,\pi]_{\mathrm{per}}$ as \smash{$\int_{-\pi}^\pi$} is ambiguous since spectral measures can have atoms at $\pm\pi$. The particular choice of $g$ is up to the practitioner: smooth $g$ makes $\nu_g$ easier to compute but tends to blur out the spectral information of $\mathcal{K}$, whereas $\nu_g$ is more challenging to compute for nonsmooth $g$ but can give better resolution of the underlying dynamics. In other situations, the application dictates that a particular $g$ is interesting (see~\cref{examMDMD}).

To compute $\nu_g$, we start by noting that the Fourier coefficients of $\nu_g$ are given by
\begin{equation}
\widehat{\nu_g}(n):=\frac{1}{2\pi}\int_{[-\pi,\pi]_{\mathrm{per}}}e^{-in\theta}\,d\nu_g(\theta)=\frac{1}{2\pi}\int_{\mathbb{T}}\lambda^{-n}\,d\mu_{g}(\lambda)=\frac{1}{2\pi}\langle \mathcal{K}'^{-n}g,g\rangle, \qquad n\in\mathbb{Z}.
\label{eq:autoCorrelations} 
\end{equation}
Since $\mathcal{K}'$ is unitary, its inverse is its adjoint and thus, the Fourier coefficients of $\nu_g$ can be expressed in terms of correlations $\langle \mathcal{K}^n g,g\rangle$ and $\langle g,\mathcal{K}^n g\rangle$. That is, for $g\in L^2(\Omega,\omega)$,
\begin{equation}
\label{F_to_prod_final}
\widehat{\nu_g}(n)=\frac{1}{2\pi}\langle \mathcal{K}^{-n}g,g\rangle,\quad n<0, \qquad \widehat{\nu_g}(n)=\frac{1}{2\pi}\langle g,\mathcal{K}^ng\rangle,\quad n\geq 0.
\end{equation}
Since $g\in L^2(\Omega,\omega)$ and~\eqref{F_to_prod_final} only depend on correlations with $\mathcal{K}$, and $\nu_g$ is determined by its Fourier coefficients, we find that $\nu_{g}$ is independent of the choice of unitary extension $\mathcal{K}'$. Henceforth, we can safely dispense with the extension $\mathcal{K}'$, and call $\nu_{g}$ the spectral measure of $\mathcal{K}$ with respect to $g$.

From~\eqref{F_to_prod_final}, we find that $\widehat{\nu_{g}}(-n)=\overline{\widehat{\nu_{g}}(n)}$ for $n\in\mathbb{Z}$, which tells us that $\nu_g$ is completely determined by the forward-time dynamical \textit{autocorrelations} $\langle  g,\mathcal{K}^ng\rangle$ with $n\geq 0$. Equivalently, the spectral measure of $\mathcal{K}$ with respect to $g\in L^2(\Omega,\omega)$ is a signature for the forward-time dynamics of~\eqref{eq:DynamicalSystem} if the closure of $\mathrm{span}\{g,\mathcal{K}g,\mathcal{K}^2g,\ldots\}$ is $L^2(\Omega,\omega)$. If the closure of $\mathrm{span}\{g,\mathcal{K}g,\mathcal{K}^2g,\ldots\}$ is $L^2(\Omega,\omega)$, then $g$ is called cyclic. If $g$ is not cyclic, then $\nu_g$ only determines the action of $\mathcal{K}$ on the closure of $\mathrm{span}\{g,\mathcal{K}g,\mathcal{K}^2g,\dots\}$, which can still be useful in some applications.

\subsubsection{Continuous and discrete parts of spectra, and Koopman mode decompositions} \label{sec:cts_disc_KMD}
Of particular importance to dynamical systems is Lebesgue's decomposition of $\nu_g$~\cite{stein2009real}:
\begin{equation}\label{eqn:spec_meas}
d\nu_g(y)= \underbrace{\sum_{\lambda=\exp(i\theta)\in\mathrm{\sigma}_{{\rm p}}(\mathcal{K})}\langle\mathcal{P}_\lambda g,g\rangle\,\delta({y-\theta})dy}_{\text{discrete part}}+\underbrace{\rho_g(y)\,dy +d\nu_g^{(\mathrm{sc})}(y)}_{\text{continuous part}}.
\end{equation}
The discrete (or atomic) part of $\nu_g$ is a sum of Dirac delta distributions, supported on $\mathrm{\sigma}_{{\rm p}}(\mathcal{K})$, the set of eigenvalues of $\mathcal{K}$.\footnote{After mapping to the periodic interval, the discrete part of $\nu_g$ is supported on the closure of $\sigma_{\mathrm{p}}(\mathcal{K}')$. However, we can always choose the extension $\mathcal{K}'$ so that $\mathrm{\sigma}_{{\rm p}}(\mathcal{K}')$=$\mathrm{\sigma}_{{\rm p}}(\mathcal{K})$ with the same eigenspaces.} The coefficient of each $\delta$ in the sum is $\langle\mathcal{P}_\lambda g,g\rangle=\|\mathcal{P}_\lambda g\|^2$, where $\mathcal{P}_\lambda$ is the orthogonal spectral projector associated with the eigenvalue $\lambda$. The continuous part of $\nu_g$ consists of a part that is absolutely continuous with respect to the Lebesgue measure, with Radon--Nikodym derivative $\rho_g\in L^1([-\pi,\pi]_{\rm per})$, and a singular continuous component $\smash{\nu_g^{(\mathrm{sc})}}$. The decomposition in~\eqref{eqn:spec_meas} provides important information on the evolution of dynamical systems. For example, suppose that there is no singular continuous spectrum. Then any $g\in L^2(\Omega,\omega)$ can be written as
$$
g=\sum_{\lambda\in\mathrm{\sigma}_{{\rm p}}(\mathcal{K})}c_\lambda\varphi_\lambda+\int_{[-\pi,\pi]_{\mathrm{per}}}\phi_{\theta,g}\, d\theta,
$$
where the $\varphi_\lambda$ are the eigenfunctions of $\mathcal{K}$, $c_\lambda$ are expansion coefficients and $\phi_{\theta,g}$ is a ``continuously parametrized'' collection of eigenfunctions.\footnote{To be precise, $\phi_{\theta,g}\, d\theta$ is the absolutely continuous component of $d\mathcal{E}(\theta)g$ and $\rho_g(\theta)=\langle \phi_{\theta,g},g\rangle$.} Then, one obtains the Koopman mode decomposition~\cite{mezic2005spectral}
\begin{equation}
\label{koopman_mode_decomp_re}
g(\pmb{x}_n)=[\mathcal{K}^ng](\pmb{x}_0)=\sum_{\lambda\in\mathrm{\sigma}_{{\rm p}}(\mathcal{K})}c_{\lambda}\lambda^n\varphi_\lambda(\pmb{x}_0)+\int_{[-\pi,\pi]_{\mathrm{per}}}e^{in\theta}\phi_{\theta,g}(\pmb{x}_0)\, d\theta.
\end{equation}

One can often characterize a dynamical system in terms of these decompositions. For example, suppose $F$ is measure-preserving and bijective, and $\omega$ is a probability measure. In that case, the dynamical system is: (1) ergodic if and only if $\lambda=1$ is a simple eigenvalue of $\mathcal{K}$, (2) weakly mixing if and only if $\lambda=1$ is a simple eigenvalue of $\mathcal{K}$ and there are no other eigenvalues, and (3) mixing if $\lambda=1$ is a simple eigenvalue of $\mathcal{K}$ and $\mathcal{K}$ has absolutely continuous spectrum on $\mathrm{span}\{1\}^\perp$~\cite{halmos2017lectures}. Different spectral types also have interpretations in the context of fluid mechanics~\cite{mezic2013analysis}, and weakly anomalous transport where the Koopman operator has singular continuous spectra~\cite{zaslavsky2002chaos}. 

\subsection{Spectra and pseudospectra}\label{pseudospec_intro_NLP}
Since the Koopman operators of interest in this paper are not always normal, the spectrum of $\mathcal{K}$ can be sensitive to small perturbations~\cite{trefethen2005spectra}. We care about pseudospectra as they are a way to determine which regions of the computed spectra are accurate and trustworthy. Moreover, if the Koopman operator is nonnormal, the system's transient behavior can differ greatly from the behavior at large times. Pseudospectra can be used to detect and quantify transients that are not captured by the spectrum~\cite[Section IV]{trefethen2005spectra}~\cite{trefethen1993hydrodynamic}. For a finite matrix $A$ and $\epsilon>0$, the $\epsilon$-pseudospectrum of $A$ is defined as $\mathrm{\sigma}_\epsilon(A) = \left\{\lambda\in\mathbb{C}: \|(A-\lambda I)^{-1}\|\geq1/\epsilon \right\}= \cup_{\|B\|\leq\epsilon} \mathrm{\sigma}(A + B)$, where $\mathrm{\sigma}(A+B)$ is the set of eigenvalues of $A + B$.\footnote{Some authors use a strict inequality in the definition of $\epsilon$-pseudospectra.} Therefore, the $\epsilon$-pseudospectra of $A$ are regions in the complex plane enclosing eigenvalues of $A$ that tell us how far an $\epsilon$ sized perturbation can perturb an eigenvalue. The $\epsilon$-pseudospectra of a Koopman operator must be defined with some care because $\mathcal{K}$ may be an unbounded operator~\cite{shargorodsky2008level}. We define the $\epsilon$-pseudospectra of $\mathcal{K}$ to be the following~\cite[Prop.~4.15]{roch1996c}:\footnote{While~\cite[Prop.~4.15]{roch1996c} considers bounded operators, it can be adjusted to cover unbounded operators~\cite[Thm.~4.3]{trefethen2005spectra}.}
$$
\mathrm{\sigma}_{\epsilon}(\mathcal{K}):=\mathrm{cl}\left(\{\lambda\in\mathbb{C}:\|(\mathcal{K}-\lambda)^{-1}\| > 1/\epsilon\}\right)=\mathrm{cl}\left(\cup_{\|\mathcal{B}\|< \epsilon}\mathrm{\sigma}(\mathcal{K}+\mathcal{B})\right),
$$
where ${\rm cl}$ is the closure of the set.

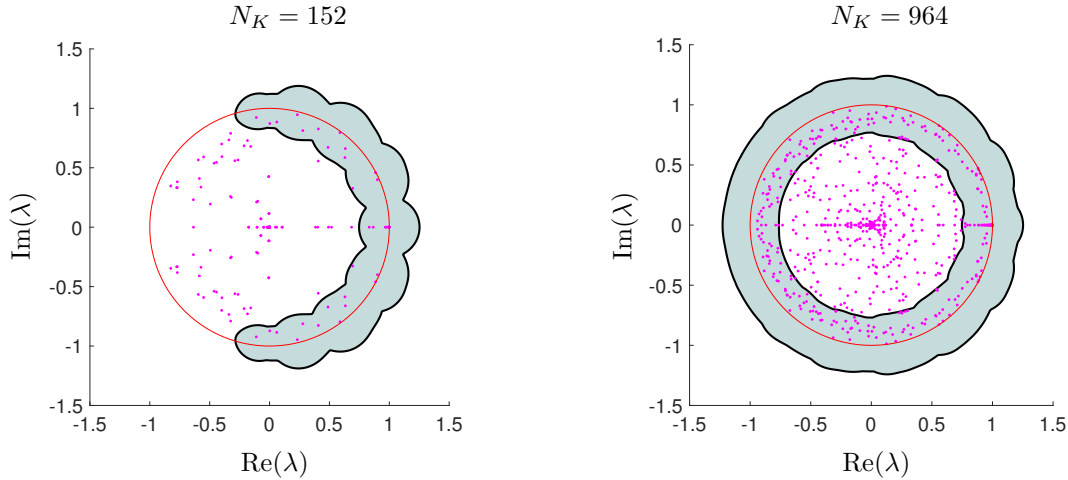
\begin{figure}[!tbp]
 \centering
 \begin{minipage}[b]{0.49\textwidth}
  \begin{overpic}[width=\textwidth,trim={0mm 0mm 0mm 0mm},clip]{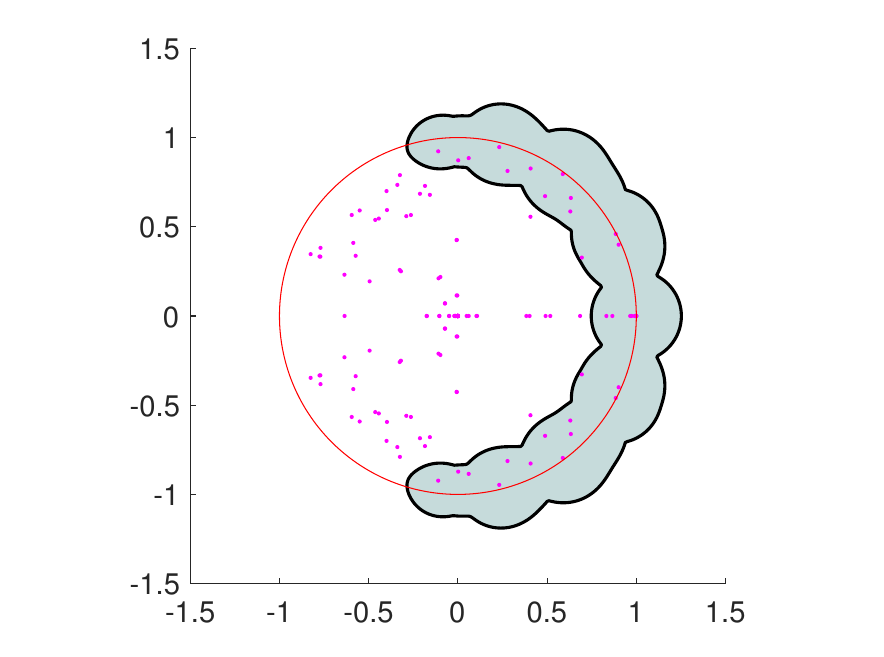}
		\put (45,73) {$N_K=152$}
   \put (47,-3) {$\mathrm{Re}(\lambda)$}
		\put (8,33) {\rotatebox{90}{$\mathrm{Im}(\lambda)$}}
   \end{overpic}
 \end{minipage}
	\begin{minipage}[b]{0.49\textwidth}
  \begin{overpic}[width=\textwidth,trim={0mm 0mm 0mm 0mm},clip]{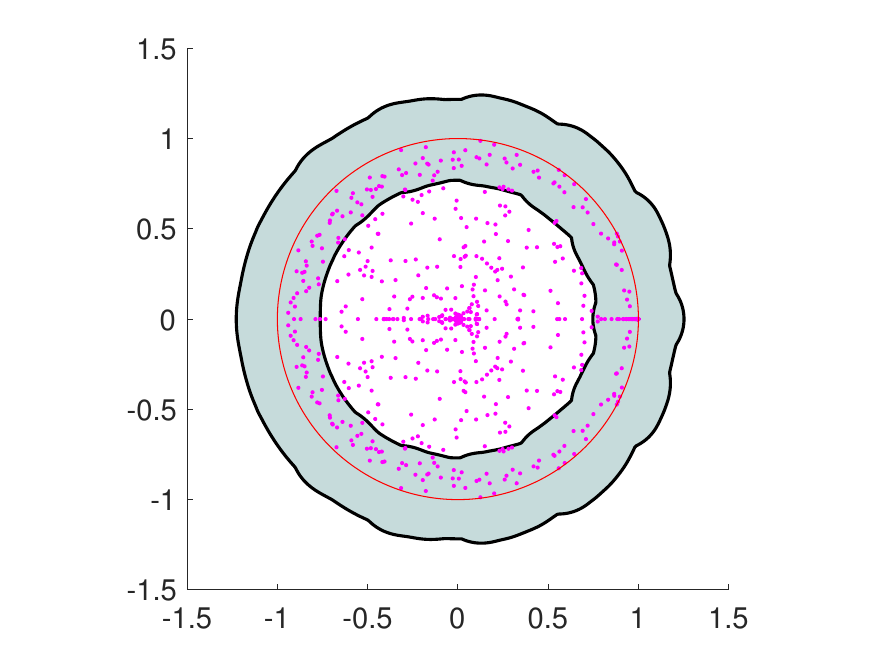}
		\put (45,73) {$N_K=964$}
   \put (47,-3) {$\mathrm{Re}(\lambda)$}
		\put (8,33) {\rotatebox{90}{$\mathrm{Im}(\lambda)$}}
   \end{overpic}
 \end{minipage}
	  \caption{The $\epsilon$-pseudospectra for the nonlinear pendulum and $\epsilon=0.25$ (shaded region) computed using \cref{alg:res_EDMD} with discretization sizes $N_K$. Discretization sizes correspond to a hyperbolic cross approximation (see~\cref{non_lin_pend_pse_exam}). The computed $\epsilon$-pseudospectra converge as $N_K\rightarrow\infty$. The unit circle (red line) is shown with the EDMD eigenvalues (magenta dots), many of which are spurious. Our algorithm called ResDMD removes spurious eigenvalues by computing pseudospectra (see~\cref{sec:RES_DMD}).}
\label{fig:pseudo1}
\end{figure}

Pseudospectra allow us to detect so-called \textit{spectral pollution}, which are spurious eigenvalues that are caused by discretization and have no relation with the underlying Koopman operator. Spectral pollution can cluster at points not in the spectrum of $\mathcal{K}$, even when $\mathcal{K}$ is a normal operator, as well as persist as the discretization size increases~\cite{lewin2010spectral}. For example, consider the dynamical system of the nonlinear pendulum. Let $\pmb{x}=(x_1,x_2)=(\theta,\dot{\theta})$ be the state variables governed by the following equations of motion: 
\begin{equation}
\label{eq:non_lin_pend_ham}
\dot{x_1}=x_2,\quad \dot{x_2}=-\sin(x_1),\quad\text{ with}\quad \Omega=[-\pi,\pi]_{\mathrm{per}}\times \mathbb{R},
\end{equation}
where $\omega$ is the standard Lebesgue measure. We consider the corresponding discrete-time dynamical system by sampling with a time-step $\Delta_t=0.5$. The system is Hamiltonian, and hence the Koopman operator is unitary. It follows that $\sigma_\epsilon(\mathcal{K})=\{z\in\mathbb{C}:\mathrm{dist}(z,\mathbb{T})\leq \epsilon\}$. \cref{fig:pseudo1} shows the pseudospectrum computed using \cref{alg:res_EDMD} for $\epsilon=0.25$ (see~\cref{non_lin_pend_pse_exam}). The algorithm uses a discretization size of $N_K$ to compute a set guaranteed to be inside the $\epsilon$-pseudospectrum (i.e., no spectral pollution) that also converges as $N_K\rightarrow\infty$. We also show the corresponding EDMD eigenvalues. Some of these EDMD eigenvalues are reliable, but most are not, demonstrating severe spectral pollution. Note that this spectral pollution has nothing to do with any stability issues but instead is due to the discretization of the infinite-dimensional operator $\mathcal{K}$ by a finite matrix. Using the $\epsilon$-pseudospectrum for different $\epsilon$, we can detect which eigenvalues are reliable (see~\cref{alg:mod_EDMD}). 

\section{Computing spectral measures from trajectory data}\label{sec:section3_filter_fin}
This section assumes that the Koopman operator $\mathcal{K}$ is associated with a measure-preserving dynamical system and hence is an isometry. We aim to compute smoothed approximations of $\nu_g$ from trajectory data of the dynamical system, even in the presence of continuous spectra. More precisely, we want to select a smoothing parameter $\epsilon>0$ and then compute a smooth periodic function $\nu^\epsilon_g$ such that for any continuous periodic function $\phi$ we have~\cite[Ch.~1]{billingsley2013convergence} 
\begin{equation}
\label{def:wk_conv}
\int_{[-\pi,\pi]_{\mathrm{per}}} \phi(\theta)\nu^\epsilon_g(\theta)\,d\theta \approx \int_{[-\pi,\pi]_{\mathrm{per}}}\phi(\theta)\,d\nu_g(\theta).
\end{equation}
Moreover, we want our smoothed approximation $\nu^\epsilon_g$ to converge weakly to $\nu_g$ $\epsilon\rightarrow 0$, meaning that the error in~\eqref{def:wk_conv} goes to zero.\footnote{Since $\int_{[-\pi,\pi]_{\mathrm{per}}}\phi(\theta)\,d\nu_g(\theta)=\langle \phi(-i\log(\mathcal{K}))g,g\rangle$, this convergence also approximates the functional calculus of $\mathcal{K}$. For example, if the dynamical system \eqref{eq:DynamicalSystem} corresponds to sampling a continuous-time dynamical system at discrete time steps, one could use \eqref{def:wk_conv} to recover spectral properties of Koopman operators that generate continuous-time dynamics.} We describe an algorithm that achieves this from trajectory data (see~\cref{alg:spec_meas_poly}). The algorithm deals with general spectral measures, including those with a singular continuous component. We also evaluate approximations of the density $\rho_g$ pointwise and can evaluate the atomic part of $\nu_g$, corresponding to eigenvalues and eigenfunctions of $\mathcal{K}$.  

\subsection{Computing autocorrelations from trajectory data}\label{sec:ComputingAutoCorrelations} 
We want to collect trajectory data from~\eqref{eq:DynamicalSystem} and use this data to recover as much information as possible about the spectral properties of $\mathcal{K}$. We assume that~\eqref{eq:DynamicalSystem} is observed for $M_2$ time-steps, starting at $M_1$ initial conditions. It is helpful to view the trajectory data as an $M_1\times M_2$ matrix 
\begin{equation} 
B_{\text{data}} = \begin{bmatrix} 
\pmb{x}_0^{(1)} & \cdots & \pmb{x}_{M_2-1}^{(1)}\\ 
\vdots & \ddots & \vdots \\ 
\pmb{x}_0^{(M_1)} & \cdots & \pmb{x}_{M_2-1}^{(M_1)}\\ 
\end{bmatrix}.
\label{eq:TrajectoryData} 
\end{equation} 
Each row of $B_{\text{data}}$ corresponds to an observation of $M_2$ time steps of a single trajectory of the dynamical system. In particular, \smash{$\pmb{x}_{i+1}^{(j)} = F(\pmb{x}_i^{(j)})$} for $0\leq i\leq M_2-2$ and $1\leq j\leq M_1$. 

There are two main ways that trajectory data might be collected: (1) Experimentally-determined initial states of the dynamical system, where one must do the best one can with predetermined initial states, and (2) Algorithmically-determined initial states, where the algorithm can select the initial states and then record the dynamics. When recovering properties of $\mathcal{K}$, often it is best to have lots of initial states that explore the whole state-space $\Omega$ with a preference of having $M_1$ large and $M_2$ small. If $M_2$ is too large, then each trajectory could quickly get trapped in attracting states, making it difficult to recover the global properties of the dynamical system.

The type of trajectory data determines how we calculate autocorrelations. There are three main ways to compute autocorrelations corresponding to different types of initial conditions $\{\pmb{x}_0^{(j)}\}_{j=1}^{M_1}$:

\begin{enumerate}[leftmargin=*]

\item \textbf{Initial conditions at quadrature nodes:} 
Suppose that we are free to choose $\{\pmb{x}_0^{(j)}\}_{j=1}^{M_1}$ so that they are an $M_1$-point quadrature rule with weights $\{w_j\}_{j=1}^{M_1}$. Integrals and inner products are then approximated with numerical integration by evaluating functions at the data points. This means that autocorrelations that determine $\nu_g$ (see~\cref{sec:spec_meas_cric_def}) can be approximated as 
$$
\langle g,\mathcal{K}^ng\rangle = \int_{\Omega}g(\pmb{x})\overline{[\mathcal{K}^ng](\pmb{x})}\,d\omega(\pmb{x}) \approx\sum_{j=1}^{M_1}w_jg(\pmb{x}_0^{(j)})\overline{[\mathcal{K}^ng](\pmb{x}_0^{(j)})}=\sum_{j=1}^{M_1}w_jg(\pmb{x}_{0}^{(j)})\overline{g(\pmb{x}_n^{(j)})}, \qquad n\geq 0.
$$
High-order quadrature rules can lead to fast rates of convergence. If $\mathcal{K}^ng$ is analytic in a neighborhood of $\Omega$, then we can often select a quadrature rule that even converges exponentially as $M_1\rightarrow\infty$~\cite{trefethen2014exponentially}. 

\item \textbf{Random initial conditions:} Suppose that $\{\pmb{x}_0^{(j)}\}_{j=1}^{M_1}$ are drawn independently at random from a fixed probability distribution over $\Omega$ that is absolutely continuous with respect to $\omega$ with a sufficiently regular Radon--Nikodym derivative $\kappa$. Then, the autocorrelations of $\mathcal{K}$ can be approximated using Monte Carlo integration, i.e.,
$$
\langle g,\mathcal{K}^ng\rangle\approx\frac{1}{M_1}\sum_{j=1}^{M_1}\kappa(\pmb{x}_0^{(j)})g(\pmb{x}_0^{(j)})\overline{[\mathcal{K}^ng](\pmb{x}_0^{(j)})}=\frac{1}{M_1}\sum_{j=1}^{M_1}\kappa(\pmb{x}_0^{(j)})g(\pmb{x}_{0}^{(j)})\overline{g(\pmb{x}_n^{(j)})}.
$$
This typically converges at a rate of $\mathcal{O}(M_1^{-1/2})$~\cite{caflisch1998monte}, but is a practical approach if the state-space dimension is large. One could also consider quasi-Monte Carlo integration, which can achieve a faster rate of $\mathcal{O}(M_1^{-1})$ (up to logarithmic factors) under suitable conditions~\cite{caflisch1998monte}.

\item \textbf{A single fixed initial condition:} Intuitively, a system is ergodic if any trajectory eventually visits all parts of the state space. Formally, this means that $F$ in~\eqref{eq:DynamicalSystem} is measure-preserving, $\omega$ is a probability measure, 
and if $A$ is a Borel measurable subset of $\Omega$ with $\left\{\pmb{x}: F(\pmb{x})\in A\right\}\subset A$, 
then $\omega(A)=0$ or $\omega(A)=1$. If a dynamical system is ergodic, then Birkhoff's ergodic theorem~\cite{birkhoff1931proof} implies that
\begin{equation}
\label{quad_comp222}
\langle g,\mathcal{K}^ng\rangle=\lim_{M_2\rightarrow\infty} \frac{1}{M_2-n}\sum_{j=0}^{M_2-n-1}g(\pmb{x}_{j})\overline{[\mathcal{K}^ng](\pmb{x}_{j})}=\lim_{M_2\rightarrow\infty} \frac{1}{M_2-n}\sum_{j=0}^{M_2-n-1}g(\pmb{x}_{j})\overline{g(\pmb{x}_{j+n})}, 
\end{equation}
for almost any initial condition $\pmb{x}_0$.\footnote{Here we mean $\omega$-almost any initial condition. In practice, it often holds in the Lebesgue sense for almost all initial conditions, and then one can randomly pick $\pmb{x}_0$ from a suitable distribution over $\Omega$.} However, the sampling scheme in \eqref{quad_comp222} is restricted to ergodic dynamical systems and often requires very long trajectories (see~\cite{kachurovskii1996rate} for convergence rates).
\end{enumerate}

If one is entirely free to select the initial conditions of the trajectory data, and $d$ is not too large, then we recommend picking them based on a high-order quadrature rule (see~\cref{sec:matrix_conv_galerkin,examp:mouse}). We can use sparse grids when the state-space dimension $d$ is moderately large. When $d$ is large, we can use a kernelized approach (see~\cref{sec:LARGEDIM}). One may also want to study the dynamics near attractors, where setting up an explicit quadrature nodes for initial conditions can be tricky. In this case, small $M_1$ and large $M_2$ can be ideal with ergodic sampling \cite[Section 7.2]{korda2020data}. We have no control over the initial conditions for experimental data, and we may need to combine the above methods of computing autocorrelations. Often it is assumed that the initial conditions correspond to random initial conditions or that we have access to long trajectories. We deal with two examples of real-world data sets in~\cref{sec:LARGEDIM}.
 
\subsection{Recovering the spectral measure from autocorrelations} 
We now suppose that one has already computed the autocorrelations $\langle g,\mathcal{K}^ng\rangle$ for $0\leq n\leq N$ and would like to recover a smoothed approximation of $\nu_g$. Since the Fourier coefficients of $\nu_g$ are given by autocorrelations (see~\eqref{eq:autoCorrelations}), the task is similar to Fourier recovery~\cite{gottlieb1997gibbs,adcock2012stable}. We are particularly interested in approaches with good convergence properties as $N\rightarrow\infty$, as this reduces the number of computed autocorrelations and the sample size $M=M_1M_2$ required for good recovery of the spectral measure. 

Motivated by the classical task of recovering a continuous function by its partial Fourier series~\cite{fejer1900fonctions}, we start by considering the ``windowing trick" from sampling theory. That is, we define a smoothed approximation to $\nu_g$ as a measure with density
\begin{equation}
\label{fejer_sum}
\nu_{g,N}(\theta)=\sum_{n=-N}^N \varphi\!\left(\frac{n}{N}\right)\widehat{\nu_g}(n)e^{in\theta} = \frac{1}{2\pi} \sum_{n=-N}^{-1}\varphi\!\left(\frac{n}{N}\right) \overline{\langle g,\mathcal{K}^{-n}g\rangle} e^{in\theta} + \frac{1}{2\pi} \sum_{n=0}^{N}\varphi\!\left(\frac{n}{N}\right)\langle g,\mathcal{K}^{n}g\rangle e^{in\theta}. 
\end{equation}
The function $\varphi : [-1,1]\rightarrow \mathbb{R}$ is often called a filter function \cite{tadmor2007gibbs,hesthaven2017numerical}. The idea of $\varphi$ is that $\varphi(x)$ is close to $1$ when $x$ is close to $0$, and $\varphi$ tapers to $0$ near $x=\pm 1$. By carefully tapering $\varphi$, the partial sum in~\eqref{fejer_sum} converges to $\nu_g$ as $N\rightarrow\infty$ (see~\cref{sec:convergence}). For fast pointwise or weak convergence of $\nu_{g,N}$ to $\nu_g$, it is desirable for $\varphi$ to be an even function that smoothly tapers from $1$ to $0$.

Using the definition of $\widehat{\nu_g}(n)$, we have
\begin{align}
\nu_{g,N}(\theta_0)=\sum_{n=-N}^N \varphi\!\left(\frac{n}{N}\right)\frac{1}{2\pi}\int_{[-\pi,\pi]_{\mathrm{per}}}e^{-in\theta_0}\, d\nu_g(\theta) e^{in\theta_0}=\int_{[-\pi,\pi]_{\mathrm{per}}}\frac{1}{2\pi}\sum_{n=-N}^N \varphi\!\left(\frac{n}{N}\right)e^{in(\theta_0-\theta)} \, d\nu_g(\theta),\label{rev_added_conv_kernel}
\end{align}
which represents our smoothed approximation as a convolution with a smoothing kernel. One of the simplest filters is the hat function $\varphi_{\rm hat}(x) = 1-|x|$, for which~\eqref{fejer_sum} corresponds to the classical Ces\`{a}ro summation of Fourier series and~\eqref{rev_added_conv_kernel} is the convolution of $\nu_g$ with the Fej\`{e}r kernel, i.e., \smash{$F_{N}(\theta)=\sum_{n=-N}^N(1-{|n|}/{N})e^{in\theta}$}. From this viewpoint, the fact that $\nu_{g,N}$ weakly converges to $\nu_g$ is immediate from Parseval's formula~\cite[Section 1.7]{katznelson2004introduction}. In particular, for any Lipschitz continuous test function $\phi$ with Lipschitz constant $L$, 
$$
\left|\phi(\theta)-\frac{1}{2\pi}[\phi*F_N](\theta)\right|\leq \frac{1}{2\pi}\int_{[-\pi,\pi]_{\mathrm{per}}}\underbrace{|\phi(\theta-s)-\phi(\theta)|}_{\leq L|s|}F_N(s)\,ds\leq \frac{L}{N\pi}\int_0^\pi s \frac{1-\cos(Ns)}{1-\cos(s)}\,ds,
$$
where we use the fact that $F_N$ is nonnegative and integrates to one. Using Fubini's theorem, we can write
$$
\int_{[-\pi,\pi]_{\mathrm{per}}} \phi(\theta)\nu_{g,N}(\theta)\,d\theta= \frac{1}{2\pi}\int_{[-\pi,\pi]_{\mathrm{per}}} [\phi*F_N](\theta) \, d\nu_g(\theta),
$$
where $*$ denotes convolution.  Since we can bound $s^2/(1-\cos(s))$ from above by $\pi^2/2$ for $s\in[0,\pi]$, we have
\begin{align}
\left| \int_{[-\pi,\pi]_{\mathrm{per}}}\phi(\theta)\,d\nu_g(\theta)-\int_{[-\pi,\pi]_{\mathrm{per}}} \phi(\theta)\nu_{g,N}(\theta)\,d\theta\right|&\leq \frac{\pi L}{2N}\int_0^\pi \frac{1-\cos(Ns)}{s}\,ds\\
&= \frac{\pi L}{2N}\int_0^{\pi N}\frac{1-\cos(x)}{x}\,dx=\mathcal{O}(N^{-1}\log(N)),\label{fejer_limit}
\end{align}
where the last equality follows from a change of variables. Therefore, there is slow weak convergence of $\nu_{g,N}$ to $\nu_g$ as $N\rightarrow \infty$, providing us with an algorithm for ensuring~\eqref{def:wk_conv} with $\epsilon=1/N$. \cref{alg:spec_meas_poly} summarizes our computational framework for recovering a smoothed version of $\nu_g$ from autocorrelations of the trajectory data.

\begin{algorithm}[t]
\textbf{Input:} Trajectory data, a filter $\varphi$, and an observable $g\in L^2(\Omega,\omega)$. \\
\vspace{-4mm}
\begin{algorithmic}[1]
\State Approximate the autocorrelations $a_n = \frac{1}{2\pi}\langle g,\mathcal{K}^ng\rangle$ for $0\leq n\leq N$. 
(The precise value of $N$ and the approach depends on the trajectory data (see~\cref{sec:ComputingAutoCorrelations}).)
\State Set $a_{-n} = \overline{a_n}$ for $1\leq n\leq N$.
\end{algorithmic} \textbf{Output:} The function $\nu_{g,N}(\theta) = \sum_{n=-N}^N \varphi\left(\frac{n}{N}\right)a_n e^{in\theta}$ that can be evaluated for any $\theta\in[-\pi,\pi]_{\rm per}$.
\caption{A computational framework for recovering an approximation of the spectral measure $\nu_g$ associated with a Koopman operator that is an isometry.}
\label{alg:spec_meas_poly}
\end{algorithm}

Other filter functions can provide a faster rate of convergence than $\varphi_{\rm hat}(x) = 1-|x|$, including the cosine and fourth-order filters~\cite{vandeven1991family,gottlieb1997gibbs}:
\[
\varphi_{{\rm cos}}(x) =\frac{1}{2}(1+\cos(\pi x)), \qquad \varphi_{{\rm four}}(x) = 1- x^4(-20|x|^3 + 70x^2 - 84|x| + 35).
\] 
For the recovery of measures, we find that a particularly good choice is
\begin{equation}
\label{cpt_kernel}
\varphi_{{\rm bump}}(x)=
\exp\left(-\frac{2}{1-|x|}\exp\left(-\frac{c}{x^4}\right)\right), \qquad c\approx 0.109550455106347,
\end{equation}
where the value of $c$ is selected so that $\varphi_{{\rm bump}}(1/2)=1/2$. This filter can lead to arbitrary high orders of convergence with errors between $\nu_{g,N}$ and $\nu_g$ that go to zero faster than any polynomial in $N^{-1}$(see~\cref{sec:convergence}). A further useful property is that $\nu_{g,N}$ localizes any singular behavior of $\nu_g$ (see~\cref{sec:high_order_needed3}).\footnote{This because the kernel associated with $\varphi_{{\rm bump}}$ (see \cref{polynomial_magic}) is highly localized due to the smoothness of $\varphi_{{\rm bump}}$.} 

To demonstrate the application of various filters for the recovery of spectral measures from autocorrelations, we consider the pedagogical dynamical system given by the shift operator. The observed orders of convergence of~\cref{alg:spec_meas_poly} are predicted by our theory (see~\cref{sec:high_order_kernel}). 

\begin{example}[Shift operator]\label{example_shift_operator}
Consider the shift operator with state-space $\Omega=\mathbb{Z}$ (and counting measure $\omega$) given by 
$$
x_{n+1} = F(x_{n}), \qquad F(x) = x + 1.  
$$
We seek to compute the spectral measure $\nu_g$ with respect to $g\in L^2(\mathbb{Z},\omega)=\ell^2(\mathbb{Z})$, where $\ell^2(\mathbb{Z})$ is the space of square summable doubly infinite vectors. This example is a building block of many dynamical systems, such as Bernoulli shifts, with so-called Lebesgue spectrum~\cite[Chapter 2]{arnold1968ergodic}. We consider the observable $g(k)=C\sin(k)/k$, where $C\approx 0.564189583547756$ is a normalization constant so that $\|g\|=1$. For this example, $\nu_g$ is absolutely continuous, but $\rho_g$ has discontinuities at $\theta=\pm1$. \cref{fig:shift_example} shows the weak convergence (left) and pointwise convergence (right) for various filters.

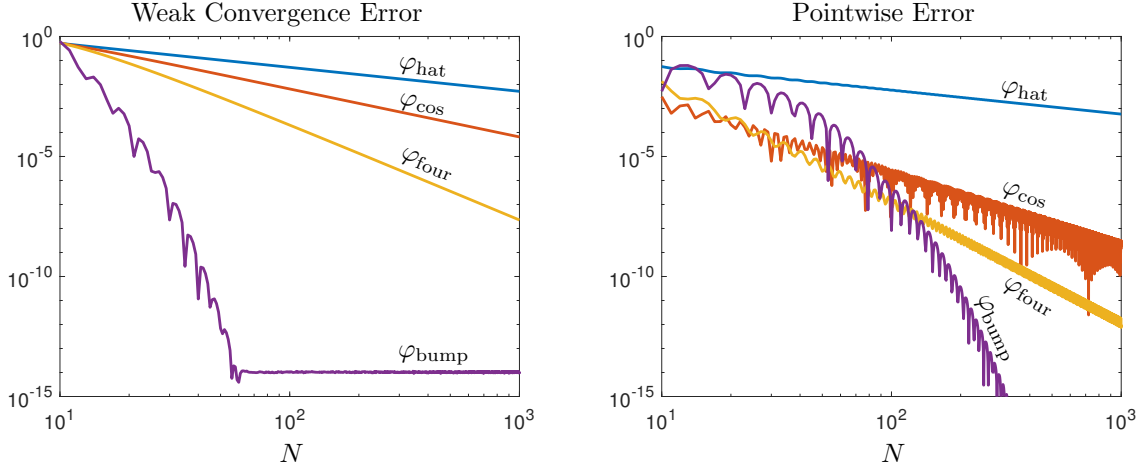
\begin{figure}[t]
 \centering
  \begin{minipage}[b]{0.49\textwidth}
  \begin{overpic}[width=\textwidth]{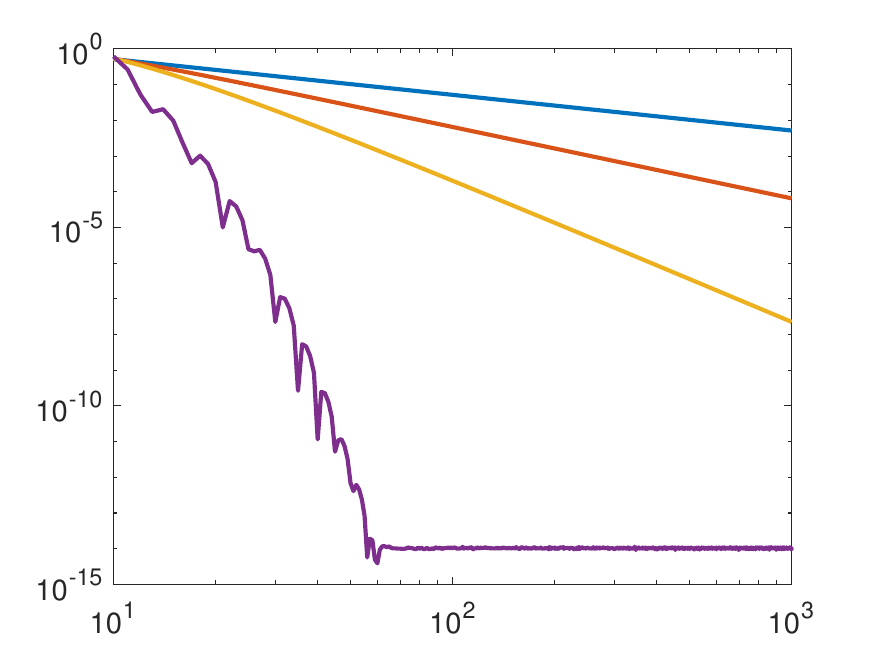}
  \put (25,72) {Weak Convergence Error}
  \put (50,-3) {$\displaystyle N$}
  \put(70,64) {\rotatebox{-6}{$\varphi_{\rm hat}$}}
  \put(70,58) {\rotatebox{-10}{$\varphi_{{\rm cos}}$}}
  \put(70,49) {\rotatebox{-23}{$\varphi_{{\rm four}}$}}
  \put(70,15) {\rotatebox{0}{$\varphi_{{\rm bump}}$}}
  \end{overpic}
 \end{minipage}
	\begin{minipage}[b]{0.49\textwidth}
  \begin{overpic}[width=\textwidth]{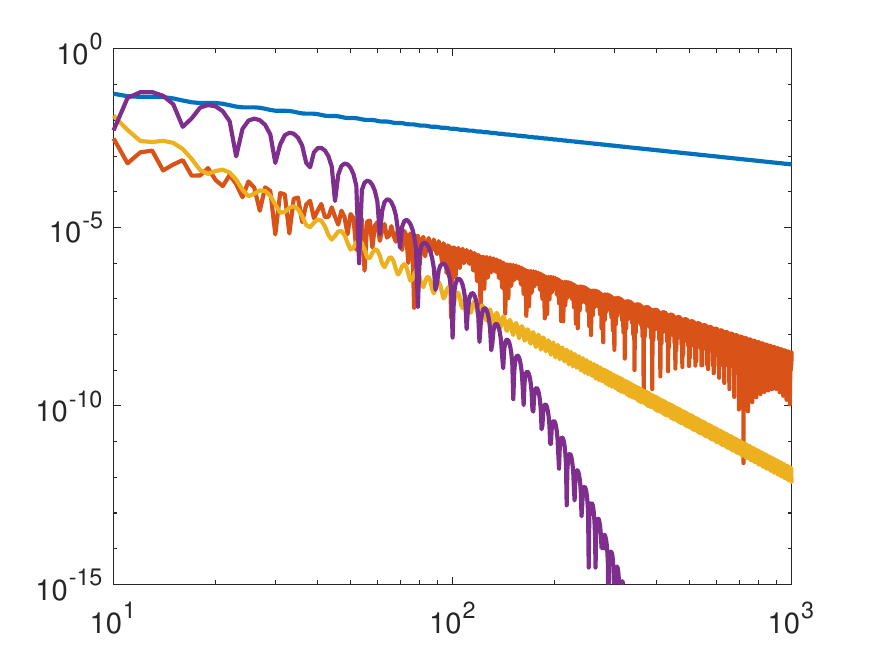}
   \put (35,72) {Pointwise Error}
   \put (50,-3) {$\displaystyle N$}
  \put(70,60) {\rotatebox{-6}{$\varphi_{\rm hat}$}}
  \put(70,43) {\rotatebox{-16}{$\varphi_{{\rm cos}}$}}
  \put(70,27) {\rotatebox{-30}{$\varphi_{{\rm four}}$}}
  \put(65,25) {\rotatebox{-65}{$\varphi_{{\rm bump}}$}}
   \end{overpic}
 \end{minipage}
 \hfill
 \caption{Relative errors between $\nu_{g,N}$ and $\nu_g$ for the shift operator computed with filters $\varphi_{\rm hat}$ (blue), $\varphi_{{\rm cos}}$ (red), $\varphi_{{\rm four}}$ (yellow), and $\varphi_{{\rm bump}}$ (purple). Left: Relative error between $\nu_{g,N}$ to $\nu_g$ in the sense of~\eqref{fejer_limit}, illustrating weak convergence in \eqref{def:wk_conv} for the test function $\phi(\theta)=\cos(5\theta)/(2+\cos(\theta))$. Right: Relative error between $\nu_{g,N}$ to $\rho_g$ at $\theta=0$, illustrating pointwise convergence. The errors are calculated using an exact analytical expression for $\nu_g$.}
	\label{fig:shift_example}
\end{figure}
\end{example} 

\subsection{High-order kernels for spectral measure recovery}\label{sec:high_order_kernel}
To develop convergence theory for~\cref{alg:spec_meas_poly}, we view filtering in the broader context of convolution with kernels. This provides a unified framework with~\cref{res_kern_fin} and simplifies the arguments when analyzing the recovery of measures (that may be highly singular), as opposed to the recovery of functions.

Instead of viewing $\nu_{g,N}$ as constructed by ``windowing" the Fourier series of $\nu_g$ in~\eqref{fejer_sum} by a filter, we form an approximation to $\nu_g$ by convolution. That is, we define 
$$
\nu_g^\epsilon(\theta_0) = \int_{[-\pi,\pi]_{\rm per}} K_\epsilon(\theta_0 - \theta) d\nu_g(\theta),
$$
where $K_\epsilon$ are a family of integrable functions $\{K_{\epsilon}:0<\epsilon\leq 1\}$ satisfying certain properties (see~\cref{unit_mth_kern_def}) so that $\nu_g^\epsilon$ converges to $\nu_g$ in some sense. We saw in~\eqref{rev_added_conv_kernel} that $\nu_g^\epsilon = \nu_{g,N}$ when $K_\epsilon(\theta) = \smash{\frac{1}{2\pi}\sum_{n=-N}^N \varphi\left(\frac{n}{N}\right)e^{in\theta}}$ and $N = \lfloor \epsilon^{-1}\rfloor$, say. The most famous example of $K_\epsilon$ is the Poisson kernel for the unit disc given by~\cite[p.16]{katznelson2004introduction}
\begin{equation}
\label{pois_def_kern}
K_{\epsilon}(\theta)=\frac{1}{2\pi}\frac{(1+\epsilon)^2-1}{1+(1+\epsilon)^2-2(1+\epsilon)\cos(\theta)},
\end{equation}
in polar coordinates with $r=(1+\epsilon)^{-1}$. The Poisson kernel is a first-order kernel because, up to a logarithmic factor, it leads to a first-order algebraic convergence rate of $\nu_g^\epsilon$ to $\nu_g$. We now give the following general definition of an $m$th order kernel and justify their name by showing that they lead to an $m$th order rate of convergence of $\nu_g^\epsilon$ to $\nu_g$ in a weak and pointwise sense (see~\cref{sec:convergence}).

\begin{definition}[$m$th order periodic kernel]\label{unit_mth_kern_def} Let $\{K_{\epsilon}:0<\epsilon\leq 1\}$ be a family of integrable functions on the periodic interval $[-\pi,\pi]_{\mathrm{per}}$. We say that $\{K_{\epsilon}\}$ is an $m$th order kernel for $[-\pi,\pi]_{\mathrm{per}}$ if
\begin{itemize}
	\item[(i)] (Normalized) $\int_{[-\pi,\pi]_{\mathrm{per}}} K_{\epsilon}(\theta)\,d\theta=1$.
	\item[(ii)] (Approximately vanishing moments) There exists a constant $C_K$ such that\footnote{One may wonder if insisting on exactly vanishing moments or removing the logarithmic term improves convergence rates as $\epsilon\downarrow 0$. This is not the case (see the discussion in \cref{sec:convergence}).}
	\begin{equation}\label{unit_first_lem2}
	\left|\int_{[-\pi,\pi]_{\mathrm{per}}}\theta^nK_{\epsilon}(\theta)\,d\theta\right|\leq C_{K} \epsilon^m \log(\epsilon^{-1}), \quad \text{for any integer $1\leq n\leq m-1$.} 
	\end{equation}
	\item[(iii)] (Decay away from $0$) For any $\theta\in[-\pi,\pi]$ and $0<\epsilon\leq 1$,
	\begin{equation}\label{unitary_decay}
	\left|K_{\epsilon}(\theta)\right|\leq \frac{C_K\epsilon^m}{(\epsilon+|\theta|)^{m+1}}.
	\end{equation}
\end{itemize}
\label{def:mthOrderKernel}
\end{definition}
The conditions in~\cref{def:mthOrderKernel} are primarily technical assumptions that allow one to prove appropriate convergence rates of $\nu_g^\epsilon$ to $\nu_g$. For pointwise convergence, property (iii) is required to apply a local cut-off argument away from singular parts of the measure. Properties (i) and (ii) are used to show that terms vanish in a local Taylor series expansion of the Radon--Nikodym derivative, and the remainder is bounded by (iii). For weak convergence, we apply similar arguments to the test function by Fubini's theorem. 

The properties of an $m$th order kernel can be translated to properties of a filter.
\begin{proposition}\label{polynomial_magic}
Let $m\in\mathbb{N}$ and suppose that $\varphi$ is an even continuous function that is compactly supported on $[-1,1]$ such that (a) $\varphi \in \mathcal{C}^{m-1}([-1,1])$, (b) $\varphi(0)=1$ and $\varphi^{(n)}(0)=0$ for any integer $1\leq n\leq m-1$, (c) $\varphi^{(n)}(1)=0$ for any integer $0\leq n\leq m-1$, and (d) $\varphi|_{[0,1]}\in \mathcal{C}^{m+1}([0,1])$. Then,
\begin{equation} 
K_{\epsilon}(\theta) = \frac{1}{2\pi}\sum_{n=-N}^N \varphi\left(\frac{n}{N}\right)e^{in\theta}, \qquad N = \lfloor \epsilon^{-1} \rfloor
\label{FilterToKernel}
\end{equation} 
is an $m$th order kernel for $[-\pi,\pi]_{\mathrm{per}}$.
\end{proposition}
\begin{proof}
We need to verify that $K_{\epsilon}$ in~\eqref{FilterToKernel} satisfies the three properties in~\cref{unit_mth_kern_def}. To see that $K_{\epsilon}$ is normalized, we integrate~\eqref{FilterToKernel} term-by-term and note that $\varphi(0) = 1$. 

For property (iii), we define $\Phi(k)=\int_{-1}^1\varphi(x)e^{ikx}\,dx$ and note that by the Poisson summation formula, 
$K_{\epsilon}(\theta)=\frac{N}{2\pi}\sum_{j=-\infty}^\infty \Phi(N(\theta+2\pi j))$. 
Since $\varphi$ is an even function, $\Phi(k)=2\int_0^1 \varphi(x)\cos(kx)\,dx$, and since $\varphi$ satisfies (b), (c) and (d), we find that for $k\neq 0$
$$
\left|\Phi(k)\right|=\frac{2}{k^{m}}\begin{dcases}
\left|\int_0^1\varphi^m(x)\cos(kx)\,dx\right|,&\quad \text{if }m\text{ is even},\\
\left|\int_0^1\varphi^m(x)\sin(kx)\,dx\right|,&\quad \text{if }m\text{ is odd}. 
\end{dcases}
$$
Finally, one can use property (d) to perform integration-by-parts (now with possibly nonvanishing endpoint conditions) to deduce that $|\Phi(k)|\lesssim {(1+|k|)^{-(m+1)}}$. Therefore, we find that
$$
\left|K_\epsilon(\theta)\right|\lesssim\sum_{j=-\infty}^\infty\frac{N}{(1+|N(\theta+2\pi j)|)^{m+1}}\lesssim \frac{N}{(1+N|\theta|)^{m+1}}.
$$
This bound implies the decay condition in \eqref{unitary_decay}.  

For property (ii), note that for any $N\geq |n|$, the following analytical expression holds: 
$$
\int_{[-\pi,\pi]_{\mathrm{per}}}K_{\epsilon}(-\theta)e^{in\theta}\,d\theta-1=\varphi\left(\frac{n}{N}\right)-\varphi(0), \qquad N = \lfloor \epsilon^{-1} \rfloor.
$$
As a consequence of Taylor's theorem and $\varphi^{(j)}(0)=0$,
$$
\left|\int_{[-\pi,\pi]_{\mathrm{per}}}K_{\epsilon}(-\theta)e^{in\theta}\,d\theta-1\right|\leq \frac{n^m\|\varphi|_{[0,1]}^{(m)}\|_{L^\infty}}{m!N^m} = \mathcal{O}\left(\epsilon^{m}\right) \quad \text{as } \epsilon \rightarrow 0.
$$
The result now follows from~\cref{unitary_rational_tool}.
\end{proof}

Therefore, it can be verified that: $\varphi_{\mathrm{hat}}$, $\varphi_{{\rm cos}}$, and $\varphi_{\rm four}$ induce first-order, second-order and fourth-order kernels in \eqref{FilterToKernel}, respectively. Similarly, $\varphi_{\rm bump}$ induces a kernel that is $m$th order for any $m\in\mathbb{N}$. For example, up to a logarithmic factor, the rate of convergence between $\nu_g^\epsilon$ (resp.~$\nu_{g,N}$) and $\nu_g$ for $\varphi_{\rm four}$ is $\mathcal{O}(\epsilon^4)$ as $\epsilon\rightarrow0$ (resp.~$\mathcal{O}(N^{-4})$ as $N\rightarrow \infty$) in a weak and pointwise sense (see~\cref{sec:convergence}).

Readers who are experts on filters may be surprised by condition (d) in~\cref{polynomial_magic}. This additional condition, which holds for all common choices of filters, is theoretically needed to obtain the optimal convergence rates between $\nu_g^\epsilon$ and $\nu_g$. It means that our convergence theory is one algebraic order better than the convergence theorems derived in the literature~\cite[Theorem 3.3]{gottlieb1997gibbs}.

\subsection{Convergence results}\label{sec:convergence}
We now provide convergence theorems for recovering the spectral measure of $\mathcal{K}$. A reader concerned with only the practical aspects of our algorithms can safely skip over this section while appreciating that the convergence guarantees are rigorous.

\subsubsection{Pointwise convergence} 
For a point $\theta_0\in [-\pi,\pi]$, the value of the approximate spectral measure $\nu_g^\epsilon(\theta_0)$ converges to the Radon--Nikodym derivative, $\rho_g(\theta_0)$ provided that $\nu_g$ is absolutely continuous in an interval containing $\theta_0$ (without this separation condition it still converges for almost every $\theta_0$). The convergence rate depends on the smoothness of $\rho_g$ in a small interval $I$ containing $\theta_0$. In particular, we write $\rho_g\in \mathcal{C}^{n,\alpha}(I)$ if $\rho_g$ is $n$-times continuously differentiable on $I$ and the $n$th derivative is H\"{o}lder continuous with parameter $0\leq \alpha<1$. For $h_1\in\mathcal{C}^{0,\alpha}(I)$ and $h_2\in\mathcal{C}^{k,\alpha}(I)$ we define the seminorm and norm, respectively, as
$$
|h_1|_{\mathcal{C}^{0,\alpha}(I)}=\sup_{x\neq y\in I}\frac{|h_1(x)-h_1(y)|}{|x-y|^{\alpha}},\quad \|h_2\|_{\mathcal{C}^{k,\alpha}(I)}=|h_2^{(k)}|_{\mathcal{C}^{0,\alpha}(I)}+\max_{0\leq j\leq k}\|h_2^{(j)}\|_{\infty,I}.
$$
We state the following pointwise convergence theorem for general complex-valued measures $\nu$ as we apply it to measures corresponding to test functions to prove~\cref{thm:unitary_weak_convergence}. The choice $\nu=\nu_g$ with $\|\nu_g\|=1$ in~\cref{thm:unitary_pointwise_convergence} gives pointwise convergence of spectral measures.

\begin{theorem}[Pointwise convergence]\label{thm:unitary_pointwise_convergence}
Let $\{K_{\epsilon}\}$ be an $m$th order kernel for $[-\pi,\pi]_{\mathrm{per}}$ and let $\nu$ be a complex-valued measure on $[-\pi,\pi]_{\mathrm{per}}$ with finite total variation $\|\nu\|$. Suppose that for some $\theta_0\in[-\pi,\pi]$ and $\eta\in(0,\pi)$, $\nu$ is absolutely continuous on $I=(\theta_0-\eta,\theta_0+\eta)$ with Radon--Nikodym derivative $\rho\in \mathcal{C}^{n,\alpha}(I)$ ($\alpha\in[0,1)$). Then the following hold for any $0\leq \epsilon <1$: 
\begin{itemize}
	\item[(i)] If $n+\alpha<m$, then
	\begin{equation}\label{unit_pt_bd1}
	\left|\rho(\theta_0)\!-\!\int_{[-\pi,\pi]_{\mathrm{per}}}\!\!\!\! K_{\epsilon}(\theta_0-\theta) \,d\nu(\theta)\right|\lesssim C_K\!\left(\|\nu\|+\|\rho\|_{\mathcal{C}^{n,\alpha}(I)}\right)\!\! \left(\epsilon^{n+\alpha}+\frac{\epsilon^m}{(\epsilon+\eta)^{m+1}}\right)\!\!\left(1+\eta^{-n-\alpha}\right)\!.
	\end{equation}
	\item[(ii)] If $n+\alpha\geq m$, then
\begin{equation}\label{unit_pt_bd2}
	\left|\rho(\theta_0)\!-\!\int_{[-\pi,\pi]_{\mathrm{per}}}\!\!\!\! K_{\epsilon}(\theta_0-\theta) \,d\nu(\theta)\right|\lesssim C_K\!\left(\|\nu\|+\|\rho\|_{\mathcal{C}^{m}(I)}\right)\!\! \left(\epsilon^{m}\log(\epsilon^{-1})+\frac{\epsilon^m}{(\epsilon+\eta)^{m+1}}\right)\!\!\left(1+\eta^{-m}\right)\!.
	\end{equation}
\end{itemize}
Here, `$\lesssim$' means that the inequality holds up to a constant that only depends on $n+\alpha$ and $m$.
\end{theorem}
\begin{proof} By periodicity, we can assume without loss of generality that $\theta_0=0$. First, we decompose $\rho$ into two parts $\rho=\rho_1+\rho_2$, where $\rho_1\in\mathcal{C}^{n,\alpha}(I)$ is compactly supported on $I$ and $\rho_2$ vanishes on $(-\eta/2,+\eta/2)$.
Using \eqref{unitary_decay}, we have
\begin{equation}\label{E_decomp_unit}
\left|\rho(0)-\int_{[-\pi,\pi]_{\mathrm{per}}}\!\!\!\! K_{\epsilon}(-\theta) d\nu(\theta)\right|\leq \left|\rho_1(0)-\int_{[-\pi,\pi]_{\mathrm{per}}}\!\!\!\! K_{\epsilon}(-\theta)\rho_1(\theta)\,d\theta\right|+\int_{|\theta|>\eta/2} \frac{C_K\epsilon^m\,d|\nu^{\mathrm{r}}|(\theta)}{(\epsilon+\eta/2)^{m+1}},
\end{equation}
where $d\nu^{\mathrm{r}}(\theta)=d\nu(\theta)-\rho_1(\theta)\,d\theta$. The second term on the right-hand side of \eqref{E_decomp_unit} is bounded by $\hat C_1 C_K(\|\nu\|+\|\rho_1\|_{L^\infty(I)})\epsilon^m(\epsilon+\eta)^{-(m+1)}$ for some constant $\hat C_1$ independent of all parameters. To bound the first term, we expand $\rho_1$ using Taylor's theorem:
\begin{equation}\label{eqn:taylor}
\rho_1(\theta)=\sum_{j=0}^{k-1}\frac{\rho_1^{(j)}(0)}{j!}\theta^j + \frac{\rho_1^{(k)}(\xi_{\theta})}{k!}\theta^k, \qquad k=\min(n,m),
\end{equation}
where $|\xi_{\theta}|\leq |\theta|$. We now consider the two cases of the theorem separately.

{\bf Case (i): $\mathbf{n+\alpha<m}$.} In this case, $k=n$ and we can select $\rho_1$ so that,
\begin{equation}\label{cut_off1}
\|\rho_1\|_{\mathcal{C}^{n,\alpha}(I)}\leq C(n,\alpha)\|\rho\|_{\mathcal{C}^{n,\alpha}(I)}\left(1+\eta^{-n-\alpha}\right),\quad \|\rho_1\|_{L^\infty(I)}\leq C(n,\alpha)\|\rho\|_{\mathcal{C}^{n,\alpha}(I)}
\end{equation}
for some universal constant $C(n,\alpha)$ that only depends on $n$ and $\alpha$. The existence of such a decomposition follows from standard arguments with cut-off functions. Using \eqref{eqn:taylor}, part (ii) of Definition \ref{unit_mth_kern_def} and the first bound of \eqref{cut_off1}, we obtain
\begin{equation}\label{eqn:approx_error_term1APPENDIX}
\begin{split}
 \left|\rho_1(0)-\int_{[-\pi,\pi]_{\mathrm{per}}}\!\!\!\! K_{\epsilon}(-\theta)\rho_1(\theta)\,d\theta\right|\leq& \hat C_2C_K\|\rho\|_{\mathcal{C}^{n,\alpha}(I)}\epsilon^m \log(\epsilon^{-1})\left(1+\eta^{-n-\alpha}\right)\\
&+\left|\int_{[-\pi,\pi]_{\mathrm{per}}}\!\!\!\! K_{\epsilon}(-\theta)\frac{\rho_1^{(n)}(\xi_{\theta})-\rho_1^{(n)}(0)}{n!}\theta^n\,d\theta\right|,
\end{split}
\end{equation}
for some constant $\hat C_2$ independent of $\epsilon$, $\eta$ and $\nu$ (or $\rho,\rho_1,\rho_2$). Note that we have added the factor of $\rho_1^{(n)}(0)$ into the integrand by a second application of part (ii) of Definition \ref{unit_mth_kern_def} and the fact that $n<m$. The H\"older continuity of $\rho_1^{(n)}$ implies that $|\rho_1^{(n)}(\xi_{\theta})-\rho_1^{(n)}(0)|\leq C(n,\alpha)\|\rho\|_{\mathcal{C}^{n,\alpha}(I)}\left(1+\eta^{-n-\alpha}\right)\theta^{\alpha}$. Using this bound in the integrand on the right-hand side of \eqref{eqn:approx_error_term1APPENDIX} and \eqref{unitary_decay}, we obtain
$$
\left|\rho_1(0)-\int_{[-\pi,\pi]_{\mathrm{per}}}\! \!\! K_{\epsilon}(-\theta)\rho_1(\theta)\,d\theta\right|\leq \hat C_3C_K\|\rho\|_{\mathcal{C}^{n,\alpha}(I)}\left(\!\epsilon^m \log(\epsilon^{-1})+ \epsilon^{n+\alpha}\int_{0}^{\pi/\epsilon}\!\!\!\frac{\tau^{n+\alpha}d\tau}{(1+\tau)^{m+1}}\!\right)\!\left(1\!+\!\eta^{-n-\alpha}\right),
$$
for some constant $\hat C_3$ independent of $\epsilon$, $\eta$ and $\nu$ (or $\rho,\rho_1,\rho_2$). Since $m> n+\alpha$, the integral in brackets converges as $\epsilon\downarrow 0$, and the bound in~\eqref{unit_pt_bd1} now follows.

{\bf Case (ii): $\mathbf{n+\alpha\geq m}$.} 
In this case, $k=m$ and we can select $\rho_1$ such that
$$
\|\rho_1\|_{\mathcal{C}^{m}(I)}\leq C(m)\|\rho\|_{\mathcal{C}^{m}(I)}\left(1+\eta^{-m}\right),
$$
for some universal constant $C(m)$ that only depends on $m$. Again, the existence of such a decomposition follows from standard arguments with cut-off functions. Using \eqref{eqn:taylor} and applying \eqref{unit_first_lem2} to the powers $\theta_j$ for $j<m$ and \eqref{unitary_decay} to the $\theta^m$ term, we obtain
\begin{equation*}
\begin{split}
 \left|\rho_1(0)-\int_{[-\pi,\pi]_{\mathrm{per}}}K_{\epsilon}(-\theta)\rho_1(\theta)\,d\theta\right|&\leq \hat C_2C_K\|\rho\|_{\mathcal{C}^{m}(I)}\left(\epsilon^m \log(\epsilon^{-1})+\epsilon^{m}\int_{0}^{\pi/\epsilon}\frac{\tau^m d\tau}{(1+\tau)^{m+1}}\right)\left(1+\eta^{-m}\right),
\end{split}
\end{equation*}
for some constant $\hat C_2$ independent of $\epsilon$, $\eta$ and $\nu$ (or $\rho,\rho_1,\rho_2$). The bound in \eqref{unit_pt_bd2} now follows.
\end{proof}

The logarithmic term in~\eqref{unit_pt_bd2} is due to the divergence of the integral $\int_{0}^{\pi/\epsilon}\frac{\tau^m d\tau}{(1+\tau)^{m+1}}$ as $\epsilon\downarrow 0$ and cannot, in general, be removed even with exactly vanishing moments replacing~\eqref{unit_first_lem2}.

\subsubsection{Weak convergence}
We now prove weak convergence in the sense of \eqref{def:wk_conv}. Using \cref{thm:unitary_pointwise_convergence} and a duality argument, we can also provide convergence rates for \eqref{def:wk_conv}. 
\begin{theorem}[Weak convergence]\label{thm:unitary_weak_convergence}
Let $\{K_{\epsilon}\}$ be an $m$th order kernel for $[-\pi,\pi]_{\mathrm{per}}$, $\phi\in\mathcal{C}^{n,\alpha}([-\pi,\pi]_{\mathrm{per}})$, and let $\nu_g$ be a spectral measure on the periodic interval $[-\pi,\pi]_{\mathrm{per}}$. Then
	\begin{equation}\label{weak_bd_m_ord}
	\left|\int_{[-\pi,\pi]_{\mathrm{per}}} \phi(\theta)\nu^\epsilon_g(\theta)\,d\theta - \int_{[-\pi,\pi]_{\mathrm{per}}}\phi(\theta)\,d\nu_g(\theta)\right|\lesssim C_K\|\phi\|_{\mathcal{C}^{n,\alpha}([-\pi,\pi]_{\mathrm{per}})}\left(\epsilon^{n+\alpha}+\epsilon^{m}\log(\epsilon^{-1})\right),
	\end{equation}
where `$\lesssim$' means that the inequality holds up to a constant that only depends on $n+\alpha$ and $m$.
\end{theorem}
\begin{proof}
Let $\tilde K_\epsilon(\theta)=K_\epsilon(-\theta)$, then it is easily seen that $\{\tilde K_{\epsilon}\}$ is an $m$th order kernel for $[-\pi,\pi]_{\mathrm{per}}$. Fubini's theorem allows us to exchange the order of integration to see that
$$
\int_{[-\pi,\pi]_{\mathrm{per}}} \phi(\theta)\nu^\epsilon_g(\theta)\,d\theta=\int_{[-\pi,\pi]_{\mathrm{per}}}\phi(\theta)[K_\epsilon*\nu_g](\theta)\,d\theta=\int_{[-\pi,\pi]_{\mathrm{per}}}[\tilde K_\epsilon*\phi](\theta)\,d\nu_g(\theta).
$$
We can now apply \cref{thm:unitary_pointwise_convergence} to the absolutely continuous measure with Radon--Nikodym derivative $\phi$ and the kernel $\tilde K_\epsilon$ (e.g., with $\eta=\pi/2$) to see that
$$
\left|[\tilde K_\epsilon*\phi](\theta)-\phi(\theta)\right|\leq C_1C_K\|\phi\|_{\mathcal{C}^{n,\alpha}([-\pi,\pi]_{\mathrm{per}})}\left(\epsilon^{n+\alpha}+\epsilon^{m}\log(\epsilon^{-1})\right),
$$
for some constant $C_1$ depending on $n$, $\alpha$ and $m$. Since $\nu_g$ is a probability measure, \eqref{weak_bd_m_ord} follows.
\end{proof}
The high-order convergence in~\cref{thm:unitary_weak_convergence} does not require regularity assumptions on $\nu_g$. Moreover, though not covered by the theorem, weak convergence still holds for any $m$th order kernel and continuous periodic function $\phi$.

\subsubsection{Recovery of the atomic parts of the spectral measure} 
Finally, we consider the recovery of the atomic parts of spectral measures or, equivalently, $\sigma_{\mathrm{p}}(\mathcal{K})$ - the set of eigenvalues of $\mathcal{K}$ (see~\eqref{eqn:spec_meas}). This convergence is achieved by \textit{rescaling} the smoothed approximation $K_\epsilon*\nu_g$. The following theorem means that~\cref{alg:spec_meas_poly} converges to both the eigenvalues of $\mathcal{K}$ and the continuous part of the spectrum of $\mathcal{K}$.
\begin{theorem}[Recovery of atoms]
\label{atom_theorem}
Let $\{K_{\epsilon}\}$ be an $m$th order kernel for $[-\pi,\pi]_{\mathrm{per}}$ that satisfies $\limsup_{\epsilon\downarrow 0}\frac{\epsilon^{-1}}{|K_\epsilon(0)|}<\infty$, and let $\nu_g$ be a spectral measure on $[-\pi,\pi]_{\mathrm{per}}$. Then, for any $\theta_0\in[-\pi,\pi]_{\mathrm{per}}$,
\begin{equation}
\label{origin_well_behaved20}
\nu_g(\{\theta_0\})=\lim_{\epsilon\downarrow 0}\frac{1}{K_\epsilon(0)}[K_\epsilon*\nu_g](\theta_0).
\end{equation}
\end{theorem}
\begin{proof}
By periodicity, we may assume without loss of generality that $\theta_0=0$. Let $\nu_g'=\nu_g-\nu_g(\{0\})\delta_0$, then
\begin{equation}
\label{extra_step_atom_pf}
\frac{1}{K_\epsilon(0)}[K_\epsilon*\nu_g](0)=\nu_g(\{0\})+\frac{1}{K_\epsilon(0)}[K_\epsilon*\nu_g'](0).
\end{equation}
Consider the function $K_{\epsilon}(-\theta) / K_\epsilon(0)$, which is uniformly bounded for sufficiently small $\epsilon$ using \eqref{unitary_decay} and the assumption $\limsup_{\epsilon\downarrow 0}\frac{\epsilon^{-1}}{|K_\epsilon(0)|}<\infty$. Since $\lim_{\epsilon\downarrow0} K_{\epsilon}(-\theta) / K_\epsilon(0)=0$ for any $\theta\neq 0$ and $\nu_g'(\{0\})=0$,
$$
\lim_{\epsilon\downarrow 0}\frac{1}{K_\epsilon(0)}[K_\epsilon*\nu_g'](0)=\lim_{\epsilon\downarrow 0}\int_{[-\pi,\pi]_{\mathrm{per}}} \frac{K_{\epsilon}(-\theta)}{K_\epsilon(0)}d\nu_g'=0,
$$
where we used the dominated convergence theorem. The theorem now follows from~\eqref{extra_step_atom_pf}. 
\end{proof}

The condition that $\limsup_{\epsilon\downarrow 0}\frac{\epsilon^{-1}}{|K_\epsilon(0)|}<\infty$ is a technical condition that is satisfied by all the kernels constructed in this paper. A condition such as this is required to recover the atomic part of $\nu_g$, as it says that $K_\epsilon$ must become localized around $0$ sufficiently quickly as $\epsilon \rightarrow 0$.

\begin{figure}[t]
 \centering
  \begin{minipage}[b]{0.49\textwidth}
  \begin{overpic}[width=\textwidth]{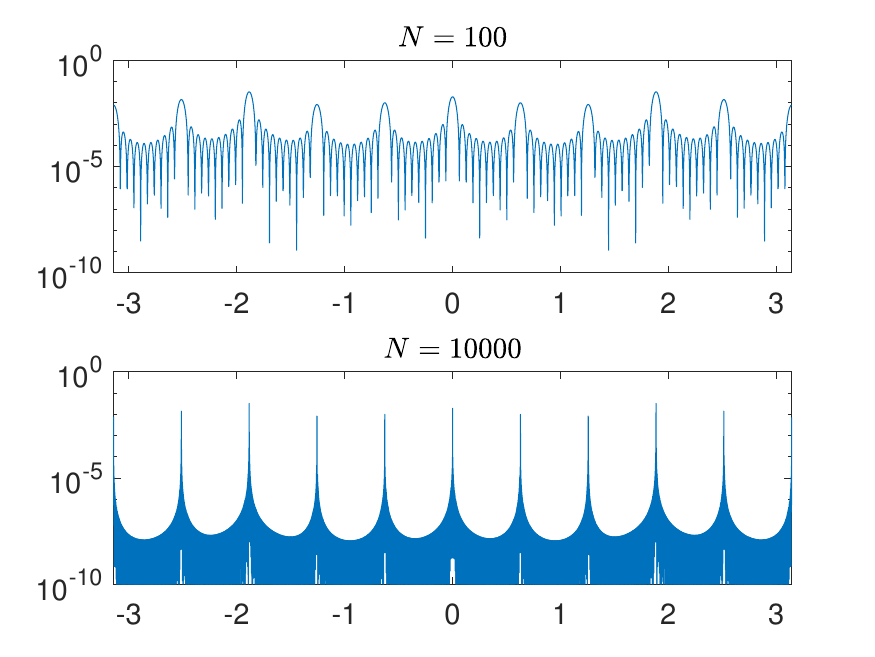}
  \end{overpic}
 \end{minipage}
	\begin{minipage}[b]{0.49\textwidth}
  \begin{overpic}[width=\textwidth]{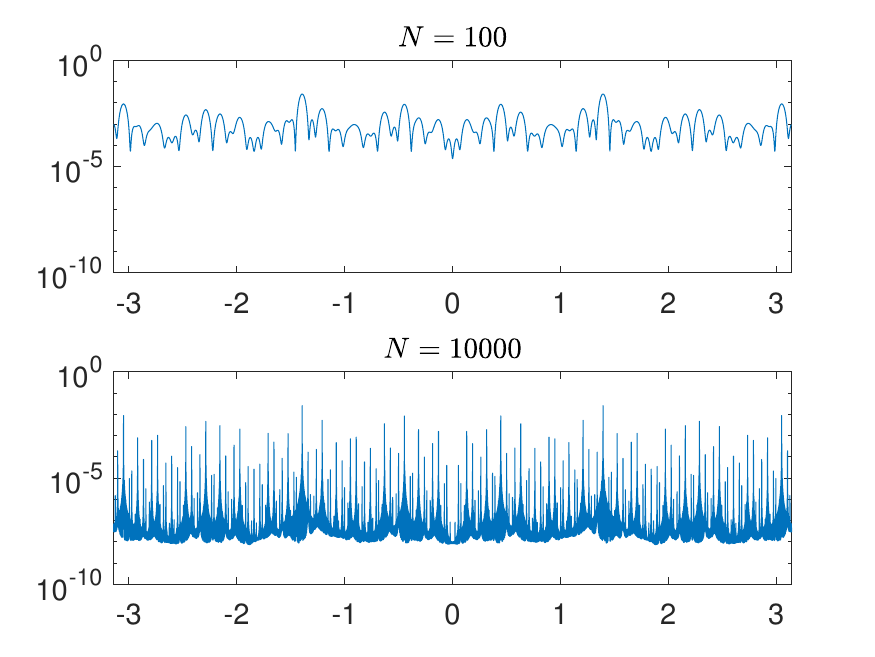}
   \end{overpic}
 \end{minipage}
 \hfill
 \caption{The approximation $[K_\epsilon*\nu_g](\theta)/K_\epsilon(0)$ of the atomic parts of the spectral measure using $\varphi_{\rm hat}$ and $g(x)=C\cos(11x)/\sqrt{1.001+\sin(x)}$, where $C\approx 0.150803789609385$ is a normalization constant so that $\|g\|=1$. As $\epsilon\downarrow 0$, the curves converge pointwise to the map $\theta\mapsto \nu_g(\{\theta\})$. Left column: The approximation for the rotation in~\eqref{rot_irr_eq} with $c=0.7$ and $N=100$ (top-left) and $N=10000$ (bottom-left) ($\epsilon=1/N$). Right column: The approximation for the rotation in~\eqref{rot_irr_eq} with $c=1/\sqrt{2}\approx0.707106781186548$ and and $N=100$ (top-left) and $N=10000$ (bottom-left) ($\epsilon=1/N$).}
\label{fig:irr_rot_example}
\end{figure}

However, we end this section with the following warning to the reader about recovering atomic parts of spectral measures. As soon as $\nu_g$ has atoms (i.e., $\mathcal{K}$ has eigenvalues and $g$ is not orthogonal to all the eigenspaces), the map $\theta\mapsto \nu_g(\{\theta\})$ is discontinuous. One can prove that, in general, separating the point spectrum from the rest of the spectrum, either in terms of spectral measures or spectral sets, is impossible for any algorithm. This holds even for simple classes of operators~\cite{colbrook2021computingCIMP,colbrook2022foundations}, unless we know apriori that the spectrum is discrete in a region of interest~\cite[Section 7.3]{colbrook2021computing}. We use smoothing with convolution kernels to regularize the map $\theta\mapsto \nu_g(\{\theta\})$. Analogues of~\cref{atom_theorem} were proven in \cite{korda2020data} using the Christoffel--Darboux kernel (which should not be confused with a convolution kernel as we have defined it) and in \cite{das2020koopman} using reproducing kernel Hilbert spaces to regularize harmonic averaging techniques. An apparent advantage of the framework of convolution kernels is that we need minimal assumptions on our dynamical system and trajectory data, and can provide explicit convergence rates in~\cref{thm:unitary_pointwise_convergence,thm:unitary_weak_convergence}.

\begin{example}[Irrational circle rotation]\label{example_irrational_circle}
Consider the rotation operator with state-space $\Omega=[-\pi,\pi]_{\mathrm{per}}$ (and standard Lebesgue measure $\omega$) given by 
\begin{equation}
\label{rot_irr_eq}
x_{n+1} = F(x_{n}), \qquad F(x) = x + 2c\pi.  
\end{equation}
If $c=p/q$ is rational, with $p$ and $q$ coprime, then the Koopman operator has pure point spectrum at $\theta = 0, 2\pi/q,\ldots, 2\pi(q-1)/q$. Otherwise, the Koopman operator has dense point spectrum. We consider the observable $g(x)=C\cos(11x)/\sqrt{1.001+\sin(x)}$, where $C\approx 0.150803789609385$ is a normalization constant so that $\|g\|=1$. This particular choice is so that $g$ has a relatively slowly decaying Fourier series.~\cref{fig:irr_rot_example} shows $[K_\epsilon*\nu_g](\theta)/K_\epsilon(0)$ for the filter $\varphi_{\rm hat}(x) = 1-|x|$ and $N=100,10000$. On the left, we plot the results for $c=0.7$ and on the right for $c=1/\sqrt{2}\approx 0.707106781186548$. We see the highly discontinuous nature of $\theta_0\mapsto \nu_g(\{\theta_0\})$, both in terms of $\theta_0$ and the map $F$ itself.
\end{example}

\subsection{Numerical examples}\label{sec:NumericalExamples}
We now consider two numerical examples of \cref{alg:spec_meas_poly} and the convergence theory in \cref{sec:convergence}.

\subsubsection{Tent map}\label{sec:tent_map}

\begin{figure}
 \centering
	\begin{minipage}[b]{0.49\textwidth}
  \begin{overpic}[width=\textwidth,trim={0mm 0mm 0mm 0mm},clip]{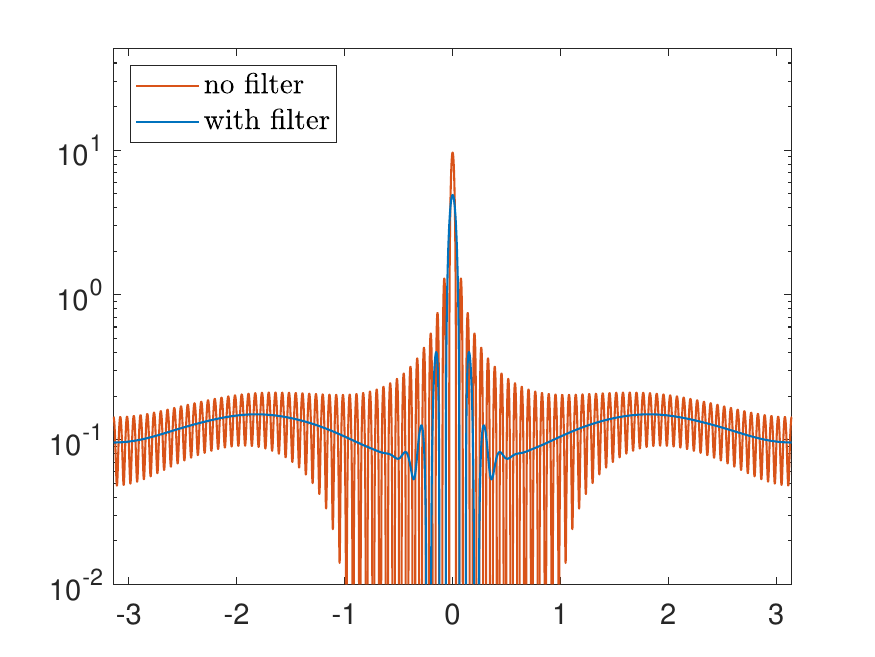}
		\put (45,73) {$\nu_g^{100}(\theta)$}
   \put (50,-3) {$\displaystyle \theta$}
		\put(30,38){\vector(0,-1){10}}
		\put(29,40){$\rho_g$}
		\put(58,58){\vector(-1,0){6}}
		\put(58,57){\small{}Eigenvalue}
   \end{overpic}
 \end{minipage}
	\hfill
	\begin{minipage}[b]{0.49\textwidth}
  \begin{overpic}[width=\textwidth,trim={0mm 0mm 0mm 0mm},clip]{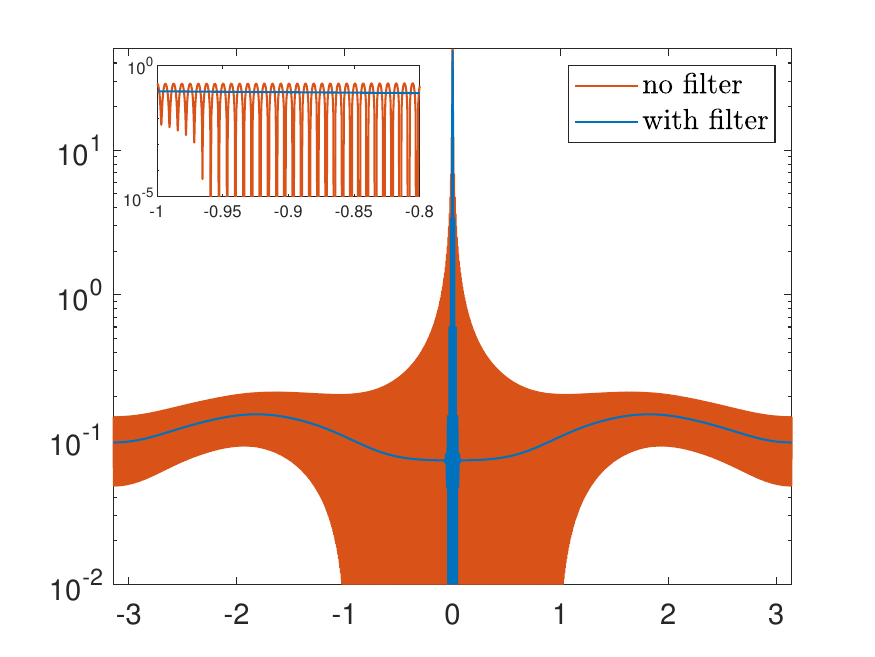}
		\put (45,73) {$\nu_g^{1000}(\theta)$}
   \put (50,-3) {$\displaystyle \theta$}
   \end{overpic}
 \end{minipage}
 \caption{Computed approximate spectral measures with respect to the function $g$ in~\eqref{f_tent2} using \cref{alg:spec_meas_poly} with the filter in \eqref{cpt_kernel} for the tent map. These are computed with (blue) and without (red) the filter in~\eqref{cpt_kernel} for discretization sizes $N = 100$ (left) and $N = 1000$ (right). The function $\nu_{g,N}$ is highly oscillatory if no filter is used (see zoomed-in subplot).} 
\label{fig:tent}
\end{figure}
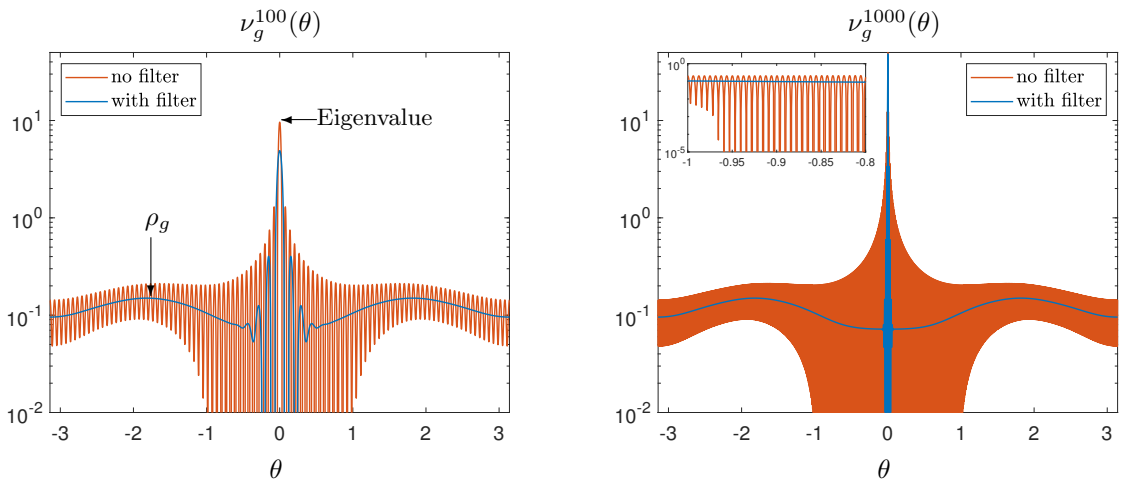

The tent map with parameter $2$ is the function $F:[0,1]\rightarrow [0,1]$ given by $F(x) = 2\min\{x,1-x\}$. It generates a chaotic system with discrete and continuous spectra. We consider $\Omega=[0,1]$ with the usual Lebesgue measure. The corresponding Koopman operator $\mathcal{K}$ is an isometry. However, it is not unitary since the function $\mathcal{K}g$ is symmetric about $x=1/2$ for any function $g$, and hence $\mathcal{K}$ is not onto. Furthermore, the decomposition in \eqref{eqn:spec_meas} reduces to an atomic part at $\theta=0$ of size $(\int_{0}^1g(x)dx)^2$ and an absolutely continuous part. The tent map thus demonstrates that our algorithm deals with mixed spectral types without a priori knowledge of the eigenvalues of $\mathcal{K}$. 

To compute the inner products $\langle g,\mathcal{K}^ng \rangle$ for \cref{alg:spec_meas_poly}, we sample the observable $g$ on an equally spaced dyadic grid with equal weights.\footnote{While this quadrature rule may seem suboptimal, it is selected because of the dyadic structure of the tent map and to avoid issues when computing the integrals $\langle g,\mathcal{K}^ng \rangle$ for large $n$ due to the chaotic nature of $F$.} 
As an example, consider the arbitrary discontinuous function
\begin{equation}
\label{f_tent2}
g(\theta)=C|\theta-1/3|+C\sin(20\theta)+\begin{dcases}C,  & \theta>0.78,\\
0, & \theta\leq 0.78,
\end{dcases}
\end{equation}
where $C\approx 1.035030525813683$ is a normalization constant. Applying \cref{alg:spec_meas_poly}, we found convergence to the Radon--Nikodym derivative away from the singular part of the measure (the atom at zero) behaved as predicted by \cref{thm:unitary_pointwise_convergence}. To see the importance of the filter, we plot the reconstruction from the Fourier coefficients both with and without the filter \eqref{cpt_kernel} in \cref{fig:tent}. As $N$ increases, we see that the filter localizes the severe oscillations near the origin, and we gain convergence to the Radon--Nikodym derivative away from the eigenvalue at $0$. This is not the case without the filter, where severe oscillations pollute the entire interval. Using the filter in~\eqref{cpt_kernel}, the atomic part approximated via \cref{atom_theorem} was correct to $4.4\times 10^{-4}$ for $N=1000$ and converges in the limit $N\rightarrow\infty$, whereas convergence via \eqref{origin_well_behaved20} does not hold without the filter.

Finally, we consider the sample complexity $M_1M_2$ needed to recover the Fourier coefficients of the measures. We consider the maximum absolute error of the computed $\widehat{\nu_g}(n)$ for $|n|\leq 10$ for our quadrature-based method and the ergodic sampling method from \eqref{quad_comp222}. For the ergodic method, $M_1=1$ and hence the sample complexity is the length of the trajectory used. However, for this example, the naive application of \eqref{quad_comp222} for the ergodic method is severely unstable. Because of the binary nature of the tent map, each application of $F$  loses a single bit of precision. Using floating-point arithmetic means that trajectories eventually stagnate at the fixed point $0$. This error is not a random numerical error but a structured one. A good solution is to perturb each evaluation of $F$ by a small random amount. The instability is absent for the tent map with parameters that avoid exact binary representations. \cref{fig:tent2} shows the results for a range of different $g$. The ergodic method converges like $\mathcal{O}(M_2^{-1/2})$ as $M_2\rightarrow \infty$, and the quadrature-based method converges with a faster rate in $M_1$. 

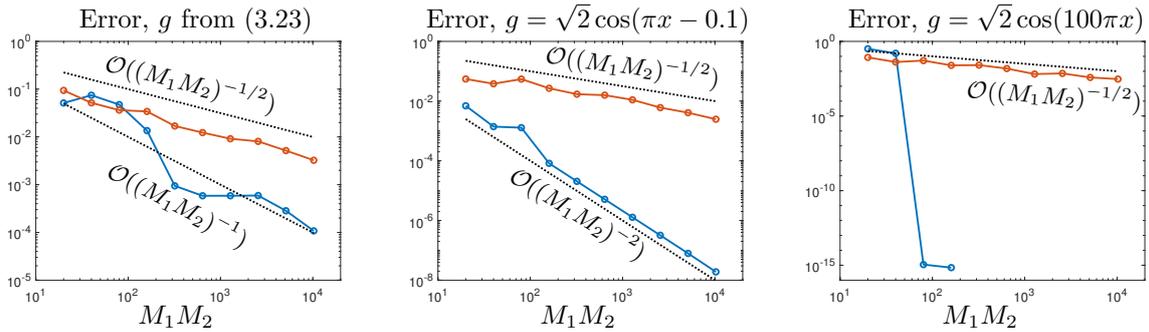
\begin{figure}
 \centering
 \begin{minipage}[b]{0.325\textwidth}
  \begin{overpic}[width=\textwidth,trim={0mm 0mm 0mm 0mm},clip]{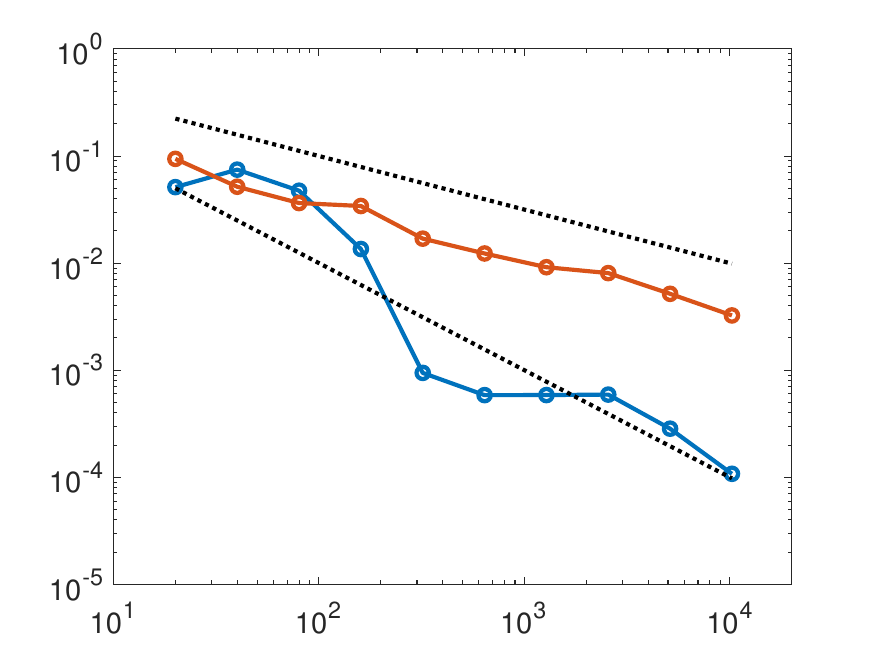}
		\put (24,73) {Error, $g$ from \eqref{f_tent2}}
		\put (30,35) {\small\rotatebox{-28} {$\mathcal{O}((M_1M_2)^{-1})$}}
		\put (30,62) {\small\rotatebox{-14} {$\mathcal{O}((M_1M_2)^{-1/2})$}}
   \put (39,-3) {$M_1M_2$}
   \end{overpic}
 \end{minipage}
	\begin{minipage}[b]{0.325\textwidth}
  \begin{overpic}[width=\textwidth,trim={0mm 0mm 0mm 0mm},clip]{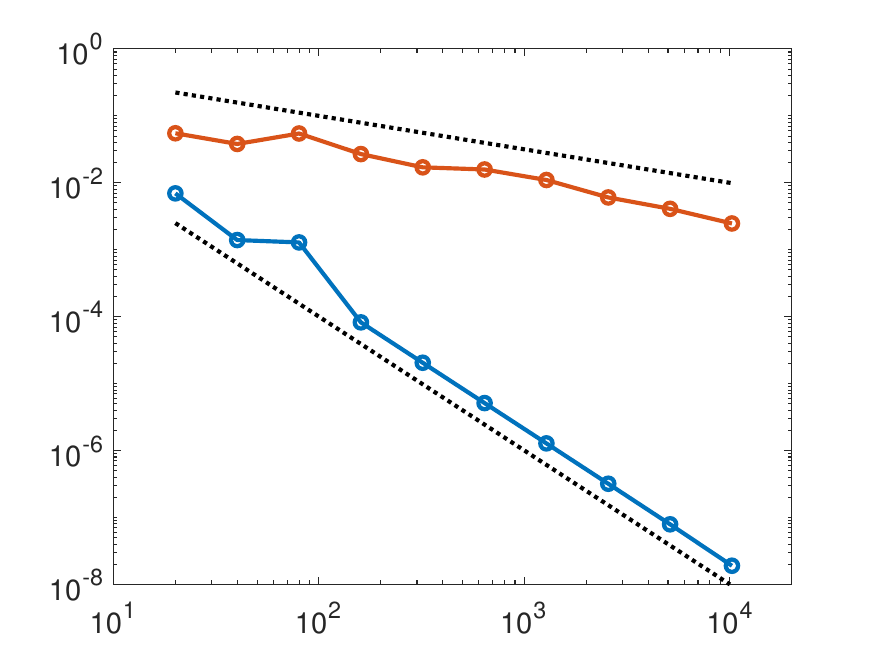}
		\put (10,73) {Error, $g=\sqrt{2} \cos(\pi x-0.1)$}
		\put (30,34) {\small\rotatebox{-32} {$\mathcal{O}((M_1M_2)^{-2})$}}
		\put (40,64) {\small\rotatebox{-8} {$\mathcal{O}((M_1M_2)^{-1/2})$}}
   \put (39,-3) {$M_1M_2$}
   \end{overpic}
 \end{minipage}
	\begin{minipage}[b]{0.325\textwidth}
  \begin{overpic}[width=\textwidth,trim={0mm 0mm 0mm 0mm},clip]{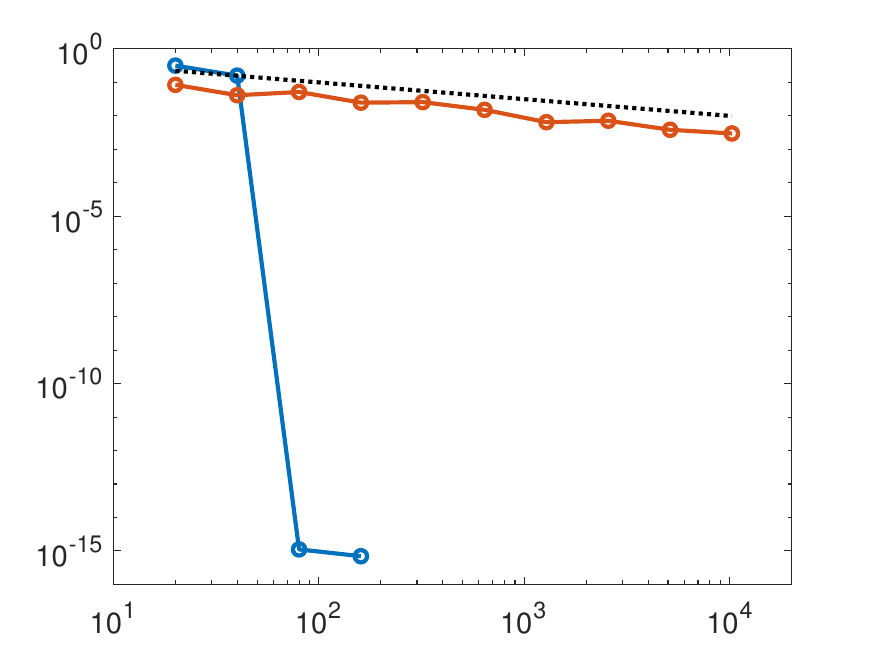}
		\put (15,73) {Error, $g=\sqrt{2} \cos(100\pi x)$}
		\put (45,55) {\small\rotatebox{-4} {$\mathcal{O}((M_1M_2)^{-1/2})$}}
   \put (39,-3) {$M_1M_2$}
   \end{overpic}
 \end{minipage}
  \caption{Maximum absolute error in computing $\widehat{\nu_g}(n)$ for $-10\leq n \leq 10$ with a sample size of $M_1M_2$ for quadrature (blue) and stabilized ergodic (red). Each $g$ is normalized so that $\|g\|=1$.} 
\label{fig:tent2}
\end{figure}

\subsubsection{Nonlinear pendulum}\label{exam:1dpendulum}

\begin{figure}
 \centering
 \begin{minipage}[b]{0.49\textwidth}
  \begin{overpic}[width=\textwidth,trim={0mm 0mm 0mm 0mm},clip]{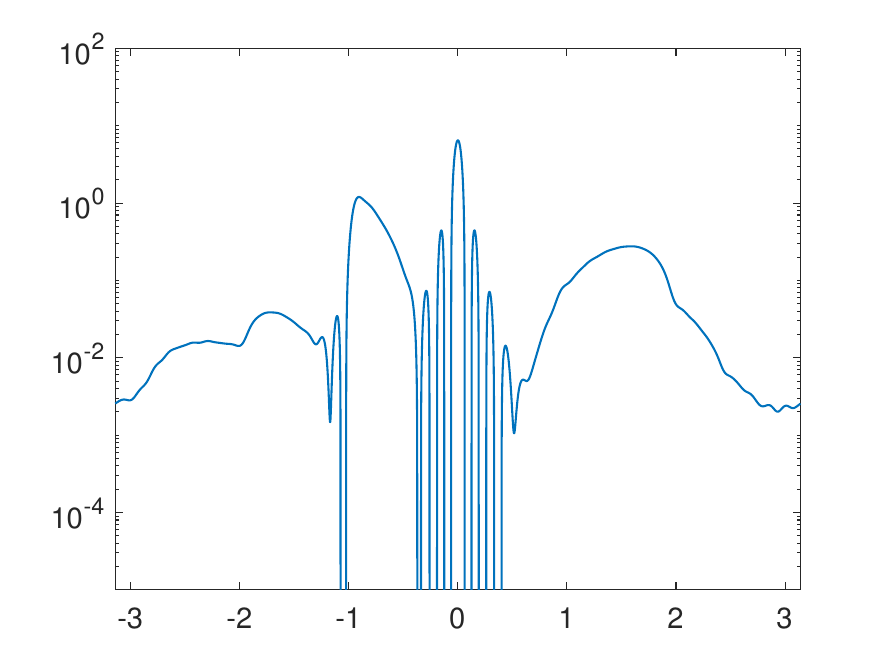}
		\put (45,73) {$\nu_g^{100}(\theta)$}
   \put (50,-3) {$\theta$}
   \end{overpic}
 \end{minipage}
	\hfill
	\begin{minipage}[b]{0.49\textwidth}
  \begin{overpic}[width=\textwidth,trim={0mm 0mm 0mm 0mm},clip]{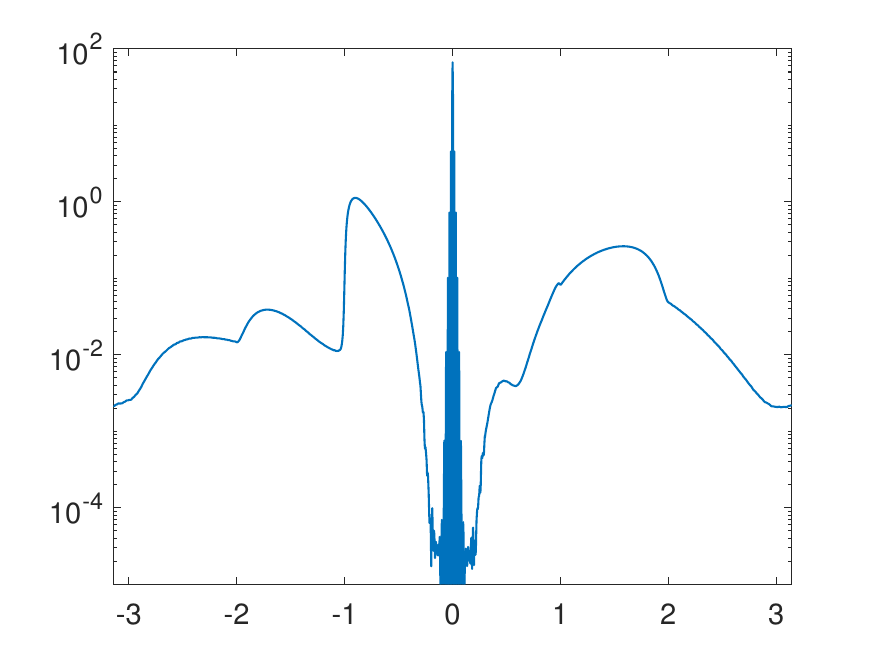}
		\put (45,73) {$\nu_g^{1000}(\theta)$}
   \put (50,-3) {$\theta$}
   \end{overpic}
 \end{minipage}
  \caption{Computed spectral measure using \cref{alg:spec_meas_poly} with \eqref{cpt_kernel} for the nonlinear pendulum.} 
\label{fig:pendulum1}
\end{figure}
The nonlinear pendulum is a nonchaotic Hamiltonian system with continuous spectra and challenging Koopman operator theory. Here, we consider a corresponding discrete-time system by sampling~\eqref{eq:non_lin_pend_ham} with a time-step of $\Delta_t=1$. We collect $M_1$ data points on an equispaced tensor product grid corresponding to the periodic trapezoidal quadrature rule in the $x_1$ direction and a truncated\footnote{We select a truncation to $x_2\in[-L,L]$ so that the $g$ and $\mathcal{K}^ng$ are negligible for $|x_2|>L$.} trapezoidal quadrature rule in the $x_2$ direction. To simulate the collection of trajectory data, we compute trajectories starting at each initial condition using the \texttt{ode45} command in MATLAB. We stress that we only use \texttt{ode45} as a black-box integrator - all of our algorithms in this paper are purely data-driven.

We look at the following observable that involves nontrivial dynamics in each coordinate:
$$
g(x_1,x_2)=C(1+i\sin(x_1))(1-\sqrt{2}x_2)e^{-x_2^2/2},
$$
where $C\approx 0.24466788518668$ is a normalization constant. \cref{fig:pendulum1} shows high resolution approximations of the spectral measure $\nu_g$ for $N=100$ using $M_1=50000$ and $N=1000$ using $M_1=10^6$. The spectral measure is purely continuous (no atoms) away from $\theta=0$, consistent with the general theory of integrable Hamiltonian systems with one degree of freedom \cite{mezic2020spectrum}. Note that the constant function $1$ is not in $L^2([0,2\pi]_{\mathrm{per}}\times\mathbb{R})$ and hence cannot be an eigenfunction. We confirmed this by using \cref{atom_theorem} for larger $N$ and observing that the peak at $\theta=0$ seen in \cref{fig:pendulum1} does not grow as fast as $\propto N$. However, the spectral measure behaves singularly at $\theta=0$.

\section{Residual DMD (ResDMD)} \label{sec:RES_DMD}
In this section, we consider dynamical systems of the form~\eqref{eq:DynamicalSystem} that are not necessarily measure-preserving. Therefore, we cannot assume that $\mathcal{K}$ is an isometry. Instead, all we assume about $\mathcal{K}$ is that it is a closed and densely defined operator. We assume that we have access to a sequence of snapshots, i.e., a trajectory data matrix (see~\eqref{eq:TrajectoryData}) with two columns: 
\begin{equation}
\label{data:snapshots}
B_{\rm data} = \begin{bmatrix}\pmb{x}^{(1)} & \pmb{y}^{(1)}\\
\vdots & \vdots \\
\pmb{x}^{(M)} & \pmb{y}^{(M)}\end{bmatrix},
\end{equation}
where $\pmb{y}^{(j)}=F(\pmb{x}^{(j)})$ for $1\leq j\leq M=M_1$. An $M_1\times M_2$ trajectory data matrix can be converted to the form of~\eqref{data:snapshots} by forgetting that the data comes from longer runs and reshaping the matrix.

Using~\eqref{data:snapshots}, we develop a new algorithm, Residual DMD (ResDMD), that approximates the associated Koopman operator of the dynamics. Our approach allows for Koopman operators $\mathcal{K}$ with no non-trivial finite-dimensional invariant subspace. The critical difference between ResDMD and other DMD algorithms (such as EDMD) is that we construct Galerkin approximations for not only $\mathcal{K}$, but also $\mathcal{K}^*\mathcal{K}$. This difference allows us to have rigorous convergence guarantees for ResDMD when recovering the spectral information of $\mathcal{K}$ and computing spectra and pseudospectra. In particular, we avoid spectral pollution (see~\cref{fig:pseudo1}). 

\subsection{Extended DMD (EDMD) and a new matrix for computing residuals}\label{extendEDMD}
Before discussing our ResDMD approach, we describe EDMD. EDMD constructs a matrix $K_{\mathrm{EDMD}}\in\mathbb{C}^{N_K\times N_K}$ that approximates the action of $\mathcal{K}$ from the snapshot data in~\eqref{data:snapshots}. The original description of EDMD assumes that $\smash{\{\pmb{x}^{(j)}\}_{j=1}^{M}\subset \Omega}$ are drawn independently according to $\omega$~\cite{williams2015data}. Here, we describe EDMD for arbitrary initial states and use $\smash{\{\pmb{x}^{(j)}\}_{j=1}^{M}}$ as quadrature nodes.

\subsubsection{EDMD viewed as a Galerkin method}\label{sec:basic_EDMD}
Given a dictionary $\{\psi_1,\ldots,\psi_{N_K}\}\subset\mathcal{D}(\mathcal{K})$ of observables, EDMD selects a matrix $K_{\mathrm{EDMD}}$ that approximates $\mathcal{K}$ on the subspace ${V}_{{N_K}}=\mathrm{span}\{\psi_1,\ldots,\psi_{N_K}\}$, i.e., ${[\mathcal{K}\psi_j](\pmb{x}) = \psi_j(F(\pmb{x})) \approx \sum_{i=1}^{N_K} (K_{\mathrm{EDMD}})_{ij} \psi_i(\pmb{x})}$ for $1\leq j\leq {N_K}$. Define the vector-valued feature map $\Psi(\pmb{x})=\begin{bmatrix}\psi_1(\pmb{x}) & \cdots& \psi_{{N_K}}(\pmb{x}) \end{bmatrix}\in\mathbb{C}^{1\times {N_K}}.$ Then any $g\in V_{N_K}$ can be written as $g(\pmb{x})=\sum_{j=1}^{N_K}\psi_j(\pmb{x})g_j=\Psi(\pmb{x})\,\pmb{g}$ for some vector $\pmb{g}\in\mathbb{C}^{N_K}$. It follows that
$$
[\mathcal{K}g](\pmb{x})=\Psi(F(\pmb{x}))\,\pmb{g}=\Psi(\pmb{x})(K_{\mathrm{EDMD}}\,\pmb{g})+\underbrace{\left(\sum_{j=1}^{N_K}\psi_j(F(\pmb{x}))g_j-\Psi(\pmb{x})(K_{\mathrm{EDMD}}\,\pmb{g})\right)}_{r(\pmb{g},\pmb{x})}.
$$
Typically, $V_{N_K}$ is not an invariant subspace of $\mathcal{K}$ so there is no choice of $K_{\mathrm{EDMD}}$ that makes $r(\pmb{g},\pmb{x})$ zero for all $g\in V_N$ and $\pmb{x}\in\Omega$. Instead, it is natural to select $K_{\mathrm{EDMD}}$ as a solution of
\begin{equation} 
\mathrm{argmin}_{B\in\mathbb{C}^{N_K\times N_K}} \left\{\int_\Omega \max_{\pmb{g}\in\mathbb{C}^{N_K},\|\pmb{g}\|=1}|r(\pmb{g},\pmb{x})|^2\,d\omega(\pmb{x})=\int_\Omega \left\|\Psi(F(\pmb{x})) - \Psi(\pmb{x})B\right\|^2_{\ell^2}\,d\omega(\pmb{x})\right\}.
\label{eq:ContinuousLeastSquaresProblem}
\end{equation} 
Here, $\|\cdot\|_{\ell^2}$ denotes the standard Euclidean norm of a vector.

In practice, one cannot directly evaluate the integral in~\eqref{eq:ContinuousLeastSquaresProblem}. Instead, we approximate it via a quadrature rule with nodes $\{\pmb{x}^{(j)}\}_{j=1}^{M}$ and weights $\{w_j\}_{j=1}^{M}$. The discretized version of~\eqref{eq:ContinuousLeastSquaresProblem} is therefore the following weighted least-squares problem:
\begin{equation}
\label{EDMD_opt_prob2}
\mathrm{argmin}_{B\in\mathbb{C}^{N_K\times N_K}}\sum_{j=1}^{M} w_j\left\|\Psi(\pmb{y}^{(j)})-\Psi(\pmb{x}^{(j)})B\right\|^2_{\ell^2}.
\end{equation}
A solution to~\eqref{EDMD_opt_prob2} can be written down explicitly as $K_{\mathrm{EDMD}}=(\Psi_X^*W\Psi_X)^{\dagger}(\Psi_X^*W\Psi_Y)$, where `$\dagger$' denotes the pseudoinverse and $W=\mathrm{diag}(w_1,\ldots,w_{M})$. Here, $\Psi_X$ and $\Psi_Y$ are the $M\times N_K$ matrices given by 
\begin{equation}
\begin{split}
\Psi_X=\begin{pmatrix}
\Psi(\pmb{x}^{(1)})\\
\vdots\\
\Psi(\pmb{x}^{(M)})
\end{pmatrix}\in\mathbb{C}^{M\times N_K},\quad
\Psi_Y=\begin{pmatrix}
\Psi(\pmb{y}^{(1)})\\
\vdots \\
\Psi(\pmb{y}^{(M)})
\end{pmatrix}\in\mathbb{C}^{M\times N_K}.
\label{psidef}
\end{split}
\end{equation}
By reducing the size of the dictionary if necessary, we may assume without loss of generality that $\Psi_X^*W\Psi_X$ is invertible. In practice, regularization through truncated singular value decompositions or a change of basis representation may also be considered. Since $\Psi_X^*W\Psi_X = \sum_{j=1}^{M} w_j \Psi(\pmb{x}^{(j)})^*\Psi(\pmb{x}^{(j)})$ and $\Psi_X^*W\Psi_Y = \sum_{j=1}^{M} w_j \Psi(\pmb{x}^{(j)})^*\Psi(\pmb{y}^{(j)})$, if the quadrature converges (see~\cref{sec:matrix_conv_galerkin}) then
\[
\lim_{M\rightarrow\infty}[\Psi_X^*W\Psi_X]_{jk} = \langle \psi_k,\psi_j \rangle\quad \text{ and }\quad \lim_{M\rightarrow\infty}[\Psi_X^*W\Psi_Y]_{jk} = \langle \mathcal{K}\psi_k,\psi_j \rangle,
\]
where $\langle \cdot,\cdot \rangle$ is the inner product associated with $L^2(\Omega,\omega)$. Thus, EDMD can be viewed as a Galerkin method in the large data limit $M\rightarrow \infty$. Let $\mathcal{P}_{V_{N_K}}$ denote the orthogonal projection onto $V_{N_K}$. In the large data limit, $K_{\rm EDMD}$ approaches a matrix representation of $\mathcal{P}_{V_{N_K}}\mathcal{K}\mathcal{P}_{V_{N_K}}^*$ and the EDMD eigenvalues approach the spectrum of $\mathcal{P}_{V_{N_K}}\mathcal{K}\mathcal{P}_{V_{N_K}}^*$. Thus, approximating $\sigma(\mathcal{K})$ by the eigenvalues of $K_{\mathrm{EDMD}}$ is closely related to the so-called finite section method~\cite{bottcher1983finite}. Since the finite section method can suffer from spectral pollution (see~\cref{pseudospec_intro_NLP}), spectral pollution is also a concern for EDMD\footnote{It is pointed out in~\cite{korda2018convergence} that if $\mathcal{K}$ is bounded and the corresponding eigenvectors of a sequence of eigenvalues do not weakly converge to zero as the discretization size increases, then the EDMD eigenvalues have a convergent subsequence to an element of the spectrum. Unfortunately, one can prove that no algorithm can determine whether a sequence of eigenvectors converges weakly to zero.} and it is important to have an independent way to measure the accuracy of the candidate eigenvalue-eigenvector pairs.

\subsubsection{Measuring the accuracy of candidate eigenvalue-eigenvector pairs}
Suppose that we have a candidate eigenvalue-eigenvector pair $(\lambda,g)$ of $\mathcal{K}$, where $\lambda\in\mathbb{C}$ and $g=\Psi\,\pmb{g}\in V_{N_K}$. One way to measure the accuracy of $(\lambda,g)$ is by estimating the squared relative residual
\begin{align}
\label{residual_form1}
\frac{\int_{\Omega}\left|[\mathcal{K}g](\pmb{x})-\lambda g(\pmb{x})\right|^2\, d\omega(\pmb{x})}{\int_{\Omega}\left|g(\pmb{x})\right|^2\, d\omega(\pmb{x})}&=\frac{\langle (\mathcal{K}-\lambda)g,(\mathcal{K}-\lambda)g\rangle}{\langle g,g \rangle}\\
&=\frac{\sum_{j,k=1}^{N_K}\overline{{g}_j}{g}_k\left[\langle \mathcal{K}\psi_k,\mathcal{K}\psi_j\rangle -\lambda\langle \psi_k,\mathcal{K}\psi_j\rangle -\overline{\lambda}\langle \mathcal{K}\psi_k,\psi_j\rangle+|\lambda|^2\langle \psi_k,\psi_j\rangle\right]}{\sum_{j,k=1}^{N_K}\overline{{g}_j}{g}_k\langle \psi_k,\psi_j\rangle}.\notag
\end{align}
If $\mathcal{K}$ is a normal operator, then the minimum of~\eqref{residual_form1} over all normalized $g\in \mathcal{D}(\mathcal{K})$ is exactly the square distance of $\lambda$ to the spectrum of $\mathcal{K}$; otherwise, for nonnormal $\mathcal{K}$ the residual can still provide a measure of accuracy (see~\cref{sec:computing_spectra_limits}). One can also use the residual to bound the distance between $g$ and the eigenspace associated with $\lambda$, assuming $\lambda$ is a point in the discrete spectrum of $\mathcal{K}$~\cite[Chapter V]{stewart1990matrix}.

We approximate the residual in~\eqref{residual_form1} by 
\begin{equation} 
\mathrm{res}(\lambda,g)^2 = \frac{\sum_{j,k=1}^{N_K}\overline{{g}_j}{g}_k\left[(\Psi_Y^*W\Psi_Y)_{jk} - \lambda(\Psi_Y^*W\Psi_X)_{jk} - \overline{\lambda}(\Psi_X^*W\Psi_Y)_{jk} + |\lambda|^2(\Psi_X^*W\Psi_X)_{jk}\right]}{\sum_{j,k=1}^{N_K}\overline{{g}_j}{g}_k(\Psi_X^*W\Psi_X)_{jk}}. 
\label{eq:abs_res}
\end{equation} 
All the terms in this residual can be computed using the snapshot data. Note that, as well as the matrices found in EDMD, \eqref{eq:abs_res} has the \textit{additional matrix} $\Psi_Y^*W\Psi_Y$. Moreover, under certain conditions, $\lim_{M\rightarrow\infty} \mathrm{res}(\lambda,g)^2 = \int_{\Omega}\left|[\mathcal{K}g](\pmb{x})-\lambda g(\pmb{x})\right|^2\, d\omega(\pmb{x})/\int_{\Omega}\left|g(\pmb{x})\right|^2\, d\omega(\pmb{x})$ for any $g\in V_{N_K}$ (see~\cref{sec:matrix_conv_galerkin}). In particular, $\lim_{M\rightarrow\infty}[\Psi_Y^*W\Psi_Y]_{jk} = \langle \mathcal{K}\psi_k,\mathcal{K}\psi_j \rangle$ and $\Psi_Y^*W\Psi_Y$ formally corresponds to a Galerkin approximation of $\mathcal{K}^*\mathcal{K}$ as $M\rightarrow \infty$. In~\cref{sec:computing_spectra_limits}, we show that the quantity $\mathrm{res}(\lambda,g)$ can be used to rigorously compute spectra and pseudospectra of $\mathcal{K}$.

\subsubsection{Convergence of matrices in the large data limit}\label{sec:matrix_conv_galerkin}
Echoing \cref{sec:ComputingAutoCorrelations}, we focus on three situations where 
\begin{equation} 
\lim_{M\rightarrow\infty}[\Psi_X^*W\Psi_X]_{jk} = \langle \psi_k,\psi_j \rangle,\quad \lim_{M\rightarrow\infty}[\Psi_X^*W\Psi_Y]_{jk} = \langle \mathcal{K}\psi_k,\psi_j \rangle,\quad \lim_{M\rightarrow\infty}[\Psi_Y^*W\Psi_Y]_{jk} = \langle \mathcal{K}\psi_k,\mathcal{K}\psi_j \rangle.
\label{eq:convergenceMatrices} 
\end{equation} 
For initial conditions at quadrature nodes, if the dictionary functions and $F$ are sufficiently regular, then it is beneficial to select $\{\pmb{x}^{(j)}\}_{j=1}^{M}$ as an $M$-point quadrature rule with weights $\{w_j\}_{j=1}^{M}$. This can lead to much faster convergence rates in~\eqref{eq:convergenceMatrices} (see \cref{examp:mouse}). For example, if $\Omega$ is unbounded, then we can use quadrature rules such as the trapezoidal rule (see \cref{non_lin_pend_pse_exam}), and if $\Omega$ is a simple bounded domain, then one can use Gaussian quadrature (see~\cref{examp:mouse}). When the state-space dimension $d$ is moderately large, we can use sparse grids and a kernelized approach for large $d$ (see~\cref{sec:LARGEDIM}).

If $\omega$ is a probability measure and the initial points $\{\pmb{x}^{(j)}\}_{j=1}^{M}$ are drawn independently and at random according to $\omega$, the strong law of large numbers shows that $\lim_{M\rightarrow\infty}[\Psi_X^*W\Psi_X]_{jk} = \langle \psi_k,\psi_j \rangle$ and $\lim_{M\rightarrow\infty}[\Psi_X^*W\Psi_Y]_{jk} = \langle \mathcal{K}\psi_k,\psi_j \rangle$ holds with probability one~\cite[Section 3.4]{2158-2491_2016_1_51} provided that $\omega$ is not supported on a zero level set that is a linear combination of the dictionary~\cite[Section 4]{korda2018convergence}. This is with the quadrature weights $w_j=1/M$, and the convergence is typically at a Monte Carlo rate of $\mathcal{O}(M^{-1/2})$. This argument is straightforward to adapt to show the convergence $\lim_{M\rightarrow\infty}[\Psi_Y^*W\Psi_Y]_{jk} = \langle \mathcal{K}\psi_k,\mathcal{K}\psi_j \rangle$.

For a single fixed initial condition,  if the dynamical system is ergodic, then one can use Birkhoff's Ergodic theorem to show that $\lim_{M\rightarrow\infty}[\Psi_X^*W\Psi_X]_{jk} = \langle \psi_k,\psi_j \rangle$ and $\lim_{M\rightarrow\infty}[\Psi_X^*W\Psi_Y]_{jk} = \langle \mathcal{K}\psi_k,\psi_j \rangle$~\cite{korda2018convergence}. One chooses $w_j=1/M$, but the convergence rate is problem dependent~\cite{kachurovskii1996rate}. This argument is straightforward to adapt to show the convergence $\lim_{M\rightarrow\infty}[\Psi_Y^*W\Psi_Y]_{jk} = \langle \mathcal{K}\psi_k,\mathcal{K}\psi_j \rangle$.

\subsection{ResDMD: Avoiding spectral pollution and computing pseudospectra}\label{sec:RESDMD_spec_pollll1}

We now present two ResDMD algorithms. The first, shown in~\cref{alg:mod_EDMD}, is a simple modification of EDMD to remove spectral pollution. However, this algorithm can miss parts of the spectrum because the finite section method is not guaranteed to approximate the whole of the spectrum. Our second algorithm (see~\cref{alg:res_EDMD}) overcomes the limitation of missing spectra and computes pseudospectra with error control. Further methods and results are given in~\cref{sec:computing_spectra_limits}.

Throughout, our only assumption is that the dictionary $\{\psi_j\}_{j=1}^{N_K}$ lies in the domain of $\mathcal{K}$. For convenience, we assume that $\{\psi_j\}_{j=1}^{N_K}$ are linearly independent; otherwise, one can compute a truncated SVD and use that dictionary instead. To obtain convergence as $N_K\rightarrow\infty$, we assume that the span of all dictionary elements in the limit forms a core of $\mathcal{K}$ (see~\cref{sec:computing_spectra_limits}). Since our algorithms come with error bounds associated with residuals, we can also perform aposteri verification of a dictionary by simply looking at the errors, which is particularly useful in~\cref{sec:LARGEDIM} (see also~\cref{ap_sec:inner_prod_err_control}).

Our first ResDMD algorithm computes the residual using the snapshot data to avoid spectral pollution. As is usually done, the algorithm assumes that $K_{\mathrm{EDMD}}$ is diagonalizable. First, we compute the three matrices $\Psi_X^*W\Psi_X$, $\Psi_X^*W\Psi_Y$, and $\Psi_Y^*W\Psi_Y$, where $\Psi_X$ and $\Psi_Y$ are given in~\eqref{psidef}. Then, we find the eigenvalues and eigenvectors of $K_{\mathrm{EDMD}}$, i.e., we solve $(\Psi_X^*W\Psi_X)^\dagger(\Psi_X^*W\Psi_Y){\pmb g} = \lambda {\pmb g}$. One can solve this eigenproblem directly, but it is often numerically more stable to solve the generalized eigenproblem $(\Psi_X^*W\Psi_Y){\pmb g} = \lambda (\Psi_X^*W\Psi_X){\pmb g}$. Afterward, to avoid spectral pollution, we discard computed eigenpairs with a larger relative residual than an accuracy goal of $\epsilon>0$. 
\begin{algorithm}[t]o
\textbf{Input:} Snapshot data $\{\pmb{x}^{(j)}\}_{j=1}^{M},\{\pmb{y}^{(j)}\}_{j=1}^{M}$ (such that $\pmb{y}^{(j)}=F(\pmb{x}^{(j)})$), quadrature weights $\{w_j\}_{j=1}^{M}$, a dictionary of observables $\{\psi_j\}_{j=1}^{N_K}$ and an accuracy goal $\epsilon>0$.\\
\vspace{-4mm}
\begin{algorithmic}[1]
\State Compute $\Psi_X^*W\Psi_X$, $\Psi_X^*W\Psi_Y$, and $\Psi_Y^*W\Psi_Y$, where $\Psi_X$ and $\Psi_Y$ are given in~\eqref{psidef}. 
\State Solve $(\Psi_X^*W\Psi_Y)\pmb{g}=\lambda (\Psi_X^*W\Psi_X)\pmb{g}$ for eigenpairs $\{(\lambda_j,g_{(j)}=\Psi\pmb{g}_j)\}$.
\State Compute $\mathrm{res}(\lambda_j,g_{(j)})$ for all $j$ (see~\eqref{eq:abs_res}) and discard if $\mathrm{res}(\lambda_j,g_{(j)})>\epsilon$.
\end{algorithmic} \textbf{Output:} A collection of accurate eigenpairs $\{(\lambda_j,\pmb{g}_j):\mathrm{res}(\lambda_j,g_{(j)})\leq\epsilon\}$.
\caption{: \textbf{ResDMD for computing eigenpairs without spectral pollution.}}\label{alg:mod_EDMD}
\end{algorithm}

\Cref{alg:mod_EDMD} summarizes the procedure and is a simple modification of EDMD, as the only difference is a clean-up where spurious eigenpairs are discarded based on their residual. This clean-up avoids spectral pollution and also removes eigenpairs that are inaccurate because of numerical errors associated with nonnormal operators, up to the relative tolerance $\epsilon$. The following result makes this precise.

\begin{theorem}\label{triv_prop}
Let $\mathcal{K}$ be the associated Koopman operator of~\eqref{eq:DynamicalSystem} from which snapshot data is collected. Let $\Lambda_{M}(\epsilon)$ denote the eigenvalues in the output of~\cref{alg:mod_EDMD}. Then, assuming~\eqref{eq:convergenceMatrices},
\[
\limsup_{M\rightarrow\infty} \max_{\lambda\in\Lambda_{M}(\epsilon)}\|(\mathcal{K}-\lambda)^{-1}\|^{-1}\leq \epsilon.
\]
\end{theorem}
\begin{proof}
Seeking a contradiction, assume that $\limsup_{M\rightarrow\infty} \max_{\lambda\in\Lambda_{M}(\epsilon)}\|(\mathcal{K}-\lambda)^{-1}\|^{-1} > \epsilon$. Then, there is a subsequence of eigenpairs $(\lambda_k,\pmb{g}_k)$ in the output of~\cref{alg:mod_EDMD} such that $\lambda_k\in\Lambda_{n_k}(\epsilon)$, $n_k\rightarrow\infty$, and $\|(\mathcal{K}-\lambda_k)^{-1}\|^{-1} >\epsilon+\delta$ for some $\delta >0$ and all $k$. Since~\eqref{eq:convergenceMatrices} is satisfied and $K_{\mathrm{EDMD}}=(\Psi_X^*W\Psi_X)^{\dagger}(\Psi_X^*W\Psi_Y)$ is a sequence of bounded finite matrices for fixed $N_K$ (recall that without loss of generality the Gram matrix $\lim_{M\rightarrow\infty}\Psi_X^*W\Psi_X$ is invertible), the sequence $\lambda_1,\lambda_2,\ldots$ stays bounded. By taking a subsequence if necessary, we may assume that $\lim_{k\rightarrow\infty}\lambda_k=\lambda$. It follows that
$$
\epsilon+\delta \leq \|(\mathcal{K}-\lambda)^{-1}\|^{-1} =\lim_{k\rightarrow\infty}\|(\mathcal{K}-\lambda_k)^{-1}\|^{-1} \leq \limsup_{k\rightarrow\infty} {\mathrm{res}(\lambda_k,\pmb{g}_k)} \leq \epsilon,
$$
which is the desired contradiction.
\end{proof}
In the large data limit,~\cref{triv_prop} tells us that ResDMD computes eigenvalues inside the $\epsilon$-pseudospectrum of $\mathcal{K}$ and hence, avoids spectral pollution and returns reasonable eigenvalues. Despite this, \cref{alg:mod_EDMD} may not approximate the whole $\epsilon$-pseudospectrum of $\mathcal{K}$, even as $M\rightarrow \infty$ and $N_K\rightarrow\infty$. This is because the eigenvalues of $K_{\mathrm{EDMD}}$ may not approximate the whole spectrum of $\mathcal{K}$. For example, consider the shift operator of~\cref{example_shift_operator}, which is unitary. Suppose our dictionary consists of the functions $\psi_j(k)=\delta_{k,q(j)}$, where $q:\mathbb{N}\rightarrow\mathbb{Z}$ is an enumeration of $\mathbb{Z}$. Then, in the large data limit, $K_{\mathrm{EDMD}}$ corresponds to a finite section of the shift operator and has spectrum $\{0\}$, whereas $\sigma(\mathcal{K})=\mathbb{T}$. Hence, for $\epsilon<1$, the output of \cref{alg:mod_EDMD} is the empty set. This issue is known as \textit{spectral inclusion}.

To overcome this issue, we discuss how to compute spectra and pseudospectra in~\cref{sec:computing_spectra_limits}. For example,~\cref{alg:res_EDMD} computes practical approximations of $\epsilon$-pseudospectra with rigorous convergence guarantees. Assuming \eqref{eq:convergenceMatrices}, the output of~\cref{alg:res_EDMD} is guaranteed to be inside the $\epsilon$-pseudospectrum of $\mathcal{K}$. \cref{alg:res_EDMD} also computes observables $g$ with $\mathrm{res}(\lambda,g)\!<\! \epsilon$, which are known as $\epsilon$-approximate eigenfunctions.\footnote{Previous methods to compute $\epsilon$-approximate eigenfunctions for Koopman operators include \cite{mezic2020numerical}, which requires the absence of continuous spectra, \cite{giannakis2021delay}, which uses delay coordinate maps \cite{giannakis2019data,das2019delay} to deal with certain isolated eigenvalues, and \cite{das2021reproducing}, which uses a compactification to jointly approximate $\epsilon$-approximate eigenfunctions of the Koopman operator associated with ergodic systems for multiple time steps.}

\begin{algorithm}[t]
\textbf{Input:} Snapshot data $\{\pmb{x}^{(j)}\}_{j=1}^{M},\{\pmb{y}^{(j)}\}_{j=1}^{M}$ (such that $\pmb{y}^{(j)}=F(\pmb{x}^{(j)})$), quadrature weights $\{w_j\}_{j=1}^{M}$, a dictionary of observables $\{\psi_j\}_{j=1}^{N_K}$, an accuracy goal $\epsilon>0$, and a grid $z_1,\ldots,z_k\in\mathbb{C}$ (see~\eqref{eq:grid_def}).\\
\vspace{-4mm}
\begin{algorithmic}[1]
\State Compute $\Psi_X^*W\Psi_X$, $\Psi_X^*W\Psi_Y$, and $\Psi_Y^*W\Psi_Y$, where $\Psi_X$ and $\Psi_Y$ are given in~\eqref{psidef}. 
\State For each $z_j$, compute $\tau_j = \min_{\pmb{g}\in\mathbb{C}^{N_{K}}} \mathrm{res}(z_j,\Psi\pmb{g})$ (see~\eqref{eq:abs_res}), which is a generalized SVD problem, and the corresponding singular vectors $\pmb{g}_j$.
\end{algorithmic} \textbf{Output:} Estimate of the $\epsilon$-pseudospectrum $\{z_j: \tau_j<\epsilon\}$ and approximate eigenfunctions $\{\pmb{g}_j: \tau_j<\epsilon\}$.
\caption{: \textbf{ResDMD for estimating $\epsilon$-pseudospectra.}}\label{alg:res_EDMD}
\end{algorithm}

\subsection{Numerical examples} 
We now apply ResDMD to several dynamical systems: (1) Nonlinear pendulum, (2) Gauss iterated map, and (3) Lorenz system. 

\subsubsection{Nonlinear pendulum revisited} \label{non_lin_pend_pse_exam}
We first return to the nonlinear pendulum from~\cref{pseudospec_intro_NLP}. For the dictionary of observables $\psi_1,\ldots,\psi_{N_K}$, we use a hyperbolic cross approximation with the standard Fourier basis (in $x_1\in[-\pi,\pi]_{\mathrm{per}}$) and Hermite functions (in $x_2\in \mathbb{R}$). The hyperbolic cross approximation is an efficient way to represent functions that have bounded mixed derivatives~\cite[Chapter 3]{dung2018hyperbolic}.  We use the trapezoidal quadrature rule discussed in~\cref{exam:1dpendulum} to compute $\Psi_X^*W\Psi_X$, $\Psi_X^*W\Psi_Y$, and $\Psi_Y^*W\Psi_Y$, where $\Psi_X$ and $\Psi_Y$ are given in~\eqref{psidef}. The pseudospectrum in~\cref{fig:pseudo1} (left) is computed using \cref{alg:res_EDMD} with $N_K=152$ basis functions corresponding to a hyperbolic cross approximation of order $20$ and $M = 10^4$ data points. The pseudospectrum in \cref{fig:pseudo1} (right) is computed using \cref{alg:res_EDMD} with $N_K=964$ basis functions corresponding to a hyperbolic cross approximation of order $100$ and $M = 9\times 10^4$ data points. Note that we can now see which EDMD eigenvalues are reliable. The output of \cref{alg:res_EDMD} lies inside $\sigma_{\epsilon}(\mathcal{K})$ and converges to the set $\sigma_{\epsilon}(\mathcal{K})$ as $N_K$ increases (see~\cref{half_pseudospectrum}). Using \cref{alg:res_EDMD} and $N_K=964$, we also compute some approximate eigenfunctions corresponding to $\lambda=\exp(0.4932i)$, $\lambda = \exp(0.9765i)$, $\lambda=\exp(1.4452i)$, and $\lambda = \exp(1.8951i)$ (see~\cref{fig:pseudo2}). As $\lambda$ moves further from $1$, we typically see increased oscillations in the approximate eigenfunctions.

\begin{figure}[!tbp]
 \centering
	\begin{minipage}[b]{0.24\textwidth}
  \begin{overpic}[width=\textwidth,trim={30mm 0mm 30mm 0mm},clip]{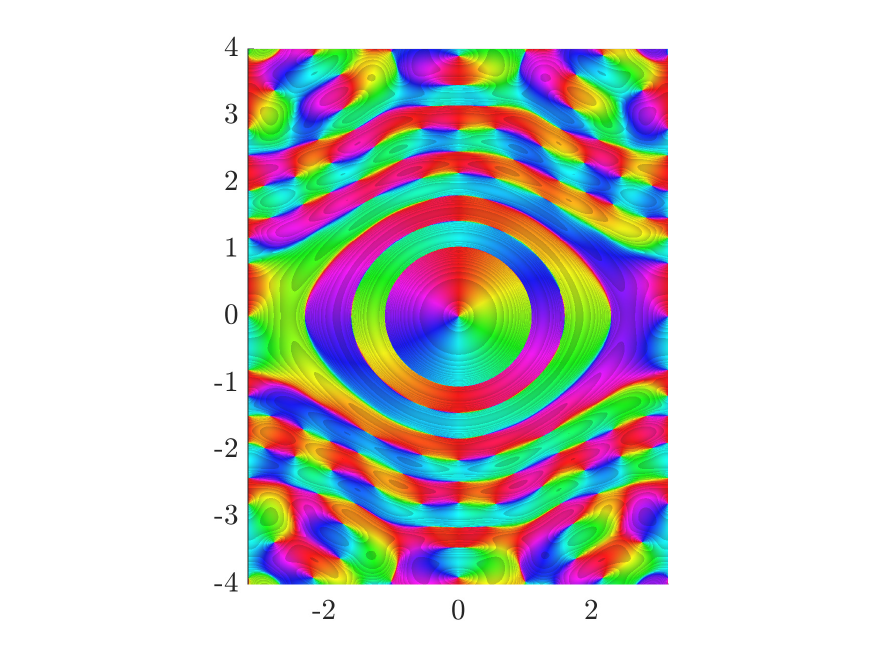}
		\put (16,96) {$\lambda=\exp(0.4932i)$}
		\put (40,-1) {$x_1$}
		\put (-2,51) {$x_2$}
   \end{overpic}
 \end{minipage}
	\begin{minipage}[b]{0.24\textwidth}
  \begin{overpic}[width=\textwidth,trim={30mm 0mm 30mm 0mm},clip]{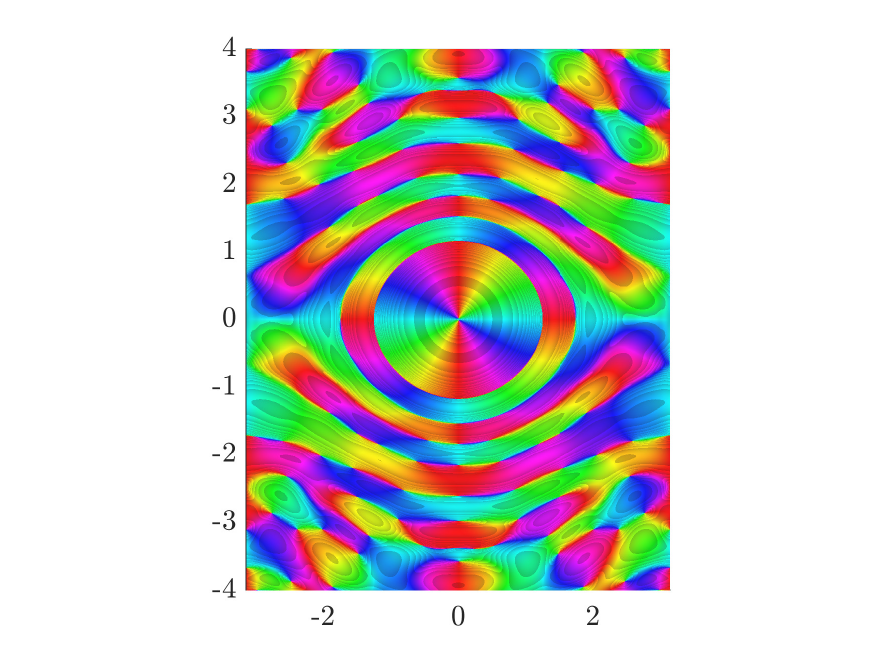}
		\put (16,96) {$\lambda=\exp(0.9765i)$}
		\put (40,-1) {$x_1$}
		\put (-2,51) {$x_2$}
   \end{overpic}
 \end{minipage}
	\begin{minipage}[b]{0.24\textwidth}
  \begin{overpic}[width=\textwidth,trim={30mm 0mm 30mm 0mm},clip]{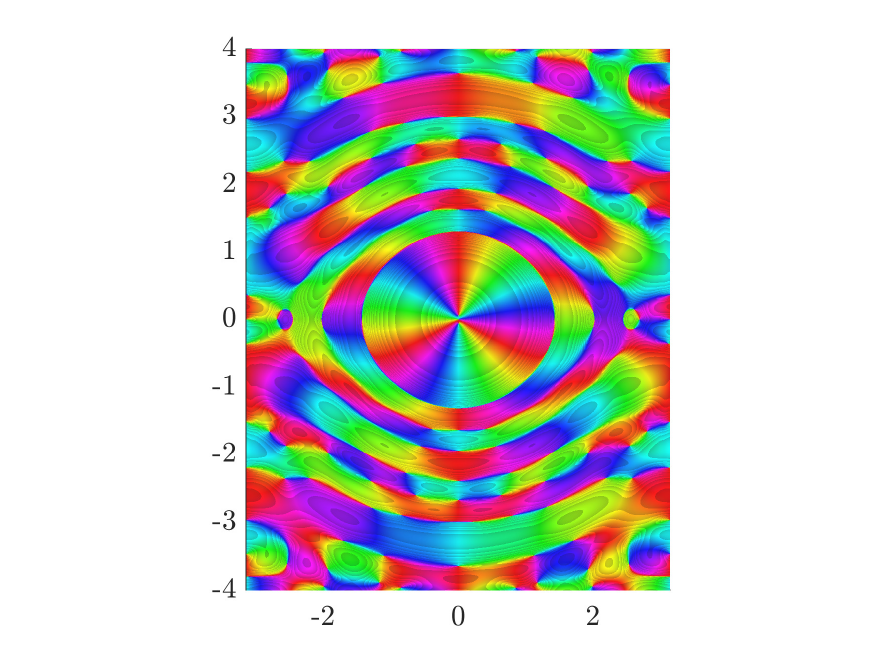}
		\put (16,96) {$\lambda=\exp(1.4452i)$}
		\put (40,-1) {$x_1$}
		\put (-2,51) {$x_2$}
   \end{overpic}
 \end{minipage}
		\begin{minipage}[b]{0.24\textwidth}
  \begin{overpic}[width=\textwidth,trim={30mm 0mm 30mm 0mm},clip]{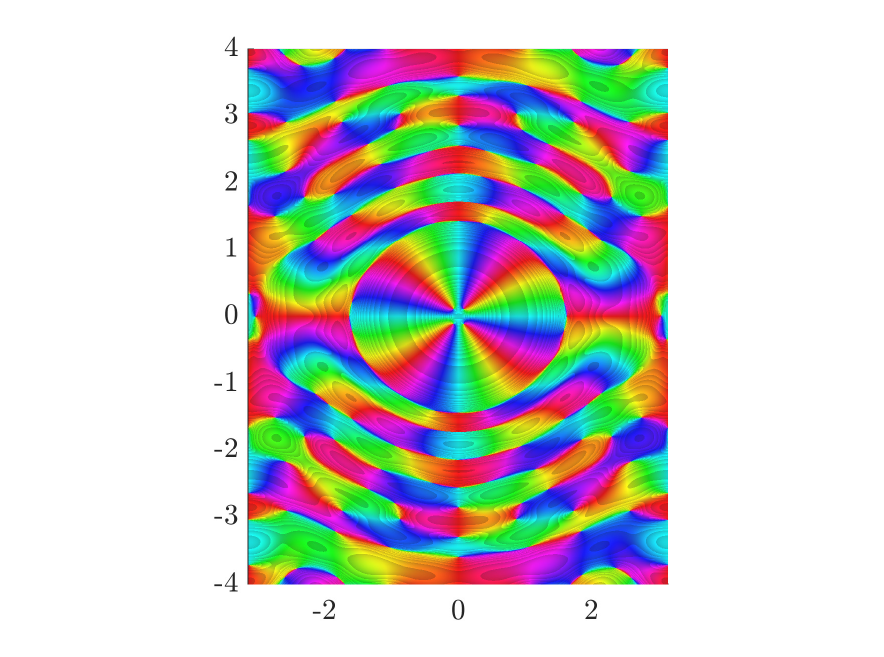}
		\put (16,96) {$\lambda=\exp(1.8951i)$}
		\put (40,-1) {$x_1$}
		\put (-2,51) {$x_2$}
   \end{overpic}
 \end{minipage}
	  \caption{The approximate eigenfunctions of the nonlinear pendulum visualized as phase portraits, where the color illustrates the complex argument of the eigenfunction~\cite{wegert2012visual}. We also plot lines of constant modulus as shadowed steps. All of these approximate eigenfunctions have residuals at most $\epsilon=0.05$ as judged by~\eqref{eq:abs_res}, which can be made smaller by increasing $N_K$.} 
\label{fig:pseudo2}
\end{figure}
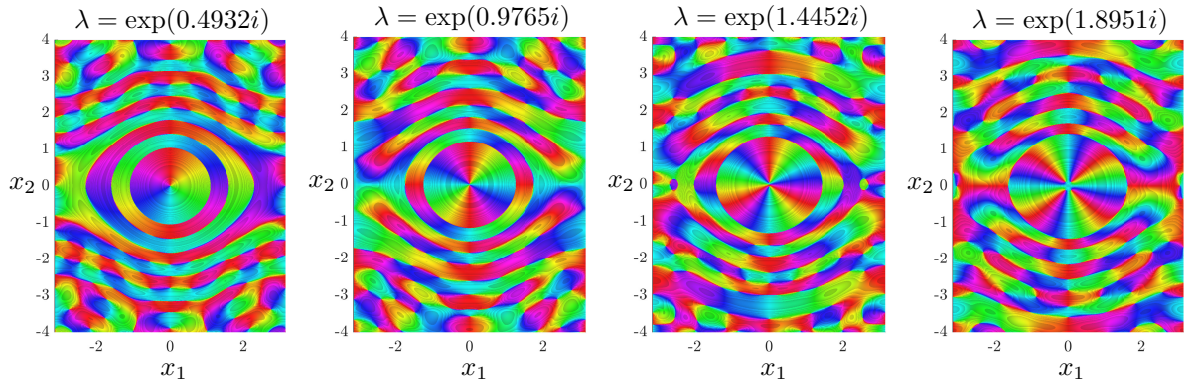

\subsubsection{Gauss iterated map}\label{examp:mouse}
The Gauss iterated map is a real-valued function on the real line given by $F(x) = \exp(-\alpha x^2)+\beta$, where $\alpha$ and $\beta$ are real parameters. It has a bifurcation diagram that resembles a mouse~\cite[Figure 5.17]{hilborn2000chaos}. We consider the choice $\alpha=2$ and $\beta=-1-\exp(-\alpha)$, and restrict ourselves to the state-space $\Omega=[-1,0]$ with the usual Lebesgue measure. We select a dictionary of Legendre polynomials transplanted to the interval $[-1,0]$, i.e., $c_0P_0(2x+1),\dots,c_{N_K-1}P_{N_K-1}(2x+1)$, where $P_j$ is the degree $j$ Legendre polynomial on $[-1,1]$ and $c_j$ are such that the dictionary is orthonormal. \cref{fig:pseudo3} (left) shows the pseudospectra computed with $N_K = 40$ using~\cref{alg:res_EDMD}. We also show the EDMD eigenvalues, several of which are reliable (blue crosses) and correspond to the output of \cref{alg:mod_EDMD} with $\epsilon=0.01$.

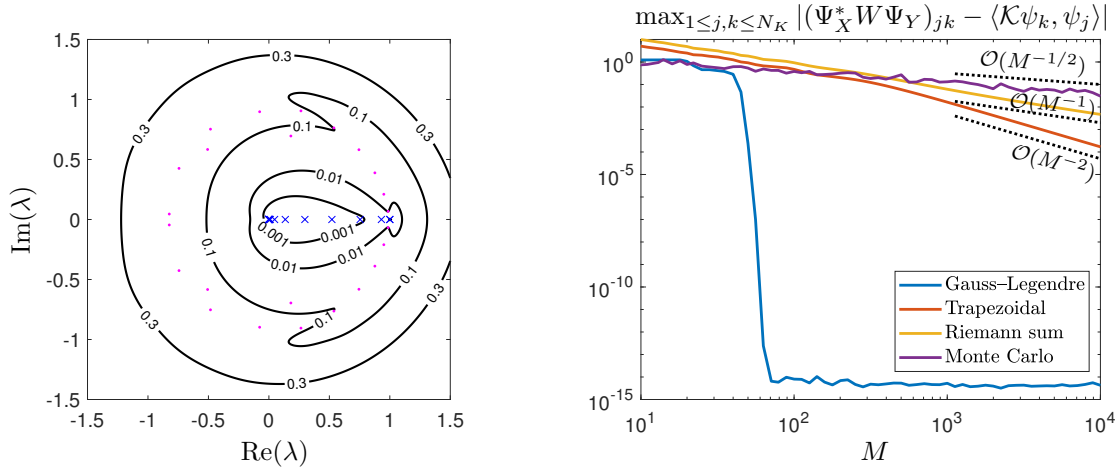
\begin{figure}[!tbp]
 \centering
 \begin{minipage}[b]{0.49\textwidth}
  \begin{overpic}[width=\textwidth,trim={0mm 0mm 0mm 0mm},clip]{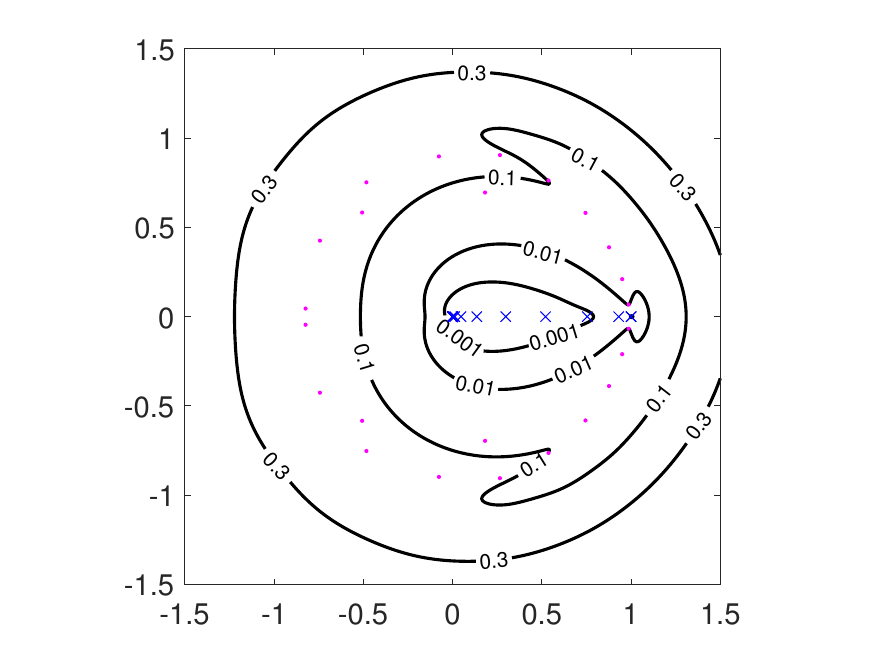}
   \put (47,-2) {$\mathrm{Re}(\lambda)$}
		\put (8,33) {\rotatebox{90}{$\mathrm{Im}(\lambda)$}}
   \end{overpic}
 \end{minipage}
	\begin{minipage}[b]{0.49\textwidth}
  \begin{overpic}[width=\textwidth,trim={0mm 0mm 0mm 0mm},clip]{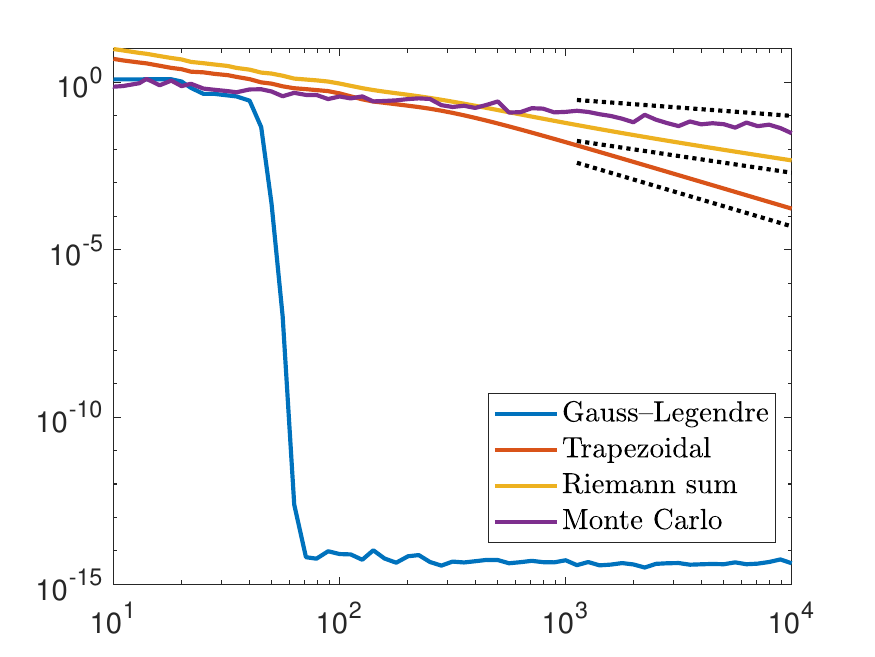}
		\put (12,72) {$\max_{1\leq j,k\leq N_K}|(\Psi_X^*W\Psi_Y)_{jk} -\langle \mathcal{K}\psi_k,\psi_j \rangle|$}
		\put (50,-2) {$M$}
		\put(70,65) {\rotatebox{-4}{\footnotesize{$\mathcal{O}(M^{-1/2})$}}}
		\put(75,58) {\rotatebox{-6}{\footnotesize{$\mathcal{O}(M^{-1})$}}}
		\put(75,49.5) {\rotatebox{-16}{\footnotesize{$\mathcal{O}(M^{-2})$}}}
   \end{overpic}
 \end{minipage}
		  \caption{Pseudospectral contours (black lines) of the Koopman operator associated with the Gauss iterated map computed using $N_K=40$ for $\epsilon=0.001,0.01,0.1$ and $0.3$. The EDMD eigenvalues (magenta dots and blue crosses) contain a mix of accurate and spurious eigenvalues. Reliable eigenvalues corresponding to the output of \cref{alg:mod_EDMD} with $\epsilon=0.01$ are shown as blue crosses. Right: Convergence of the Galerkin matrix $\Psi_X^*W\Psi_Y$ for different quadrature rules as the number of data points $M$ increases.} 
\label{fig:pseudo3}
\end{figure}

Rather than showing the convergence of the eigenvalues and pseudospectra as $N_K\rightarrow\infty$, we demonstrate the importance of the initial conditions of the snapshot data as it determines the quadrature rule in~\cref{extendEDMD}. If $M$ is not sufficiently large, or the quadrature rule is poor, the Galerkin method in \cref{alg:mod_EDMD,alg:res_EDMD} is not accurate. To compare initial conditions, we consider the computed matrix $\Psi_X^*W\Psi_Y$ for $N_K=40$ with four choices of quadrature:
\begin{itemize}[leftmargin=*]
	\item Gauss--Legendre: Select $\pmb{x}^{(1)},\ldots,\pmb{x}^{(M)}$ as the Gauss--Legendre quadrature nodes transplanted to $[-1,0]$ and $w_1,\ldots,w_{M}$ as the corresponding Gauss--Legendre weights.
	\item Trapezoidal: Select $\pmb{x}^{(j)}=-1+\frac{j-1}{M-1}$ with $w_j=\frac{1}{M-1}$ for $1<j<M$ and $w_1=w_{M}=\frac{1}{2(M-1)}$.
	\item Riemann sum: Select $\pmb{x}^{(j)}=-1+\frac{j-1}{M-1}$ and $w_j=\frac{1}{M-1}$.
	\item Monte Carlo: Select $\pmb{x}^{(1)},\ldots,\pmb{x}^{(M)}$ independently at random from a uniform distribution over $[-1,0]$ and $w_j=\frac{1}{M}$.
\end{itemize}
The final two rules are the most commonly used in the EDMD/Koopman literature. The trapezoidal rule uses the same data points as the Riemann sum but changes the weights at the endpoints.

\cref{fig:pseudo3} (right) shows the accuracy, measured as $\max_{1\leq j,k\leq N_K}|(\Psi_X^*W\Psi_Y)_{jk} -\langle \mathcal{K}\psi_k,\psi_j \rangle|$, of the four quadrature methods. Gauss--Legendre quadrature converges exponentially, trapezoidal at a rate $\mathcal{O}(M^{-2})$, Riemann sums at a rate $\mathcal{O}(M^{-1})$ and Monte Carlo at a rate $\mathcal{O}(M^{-1/2})$. The benefit of choosing the initial states at quadrature nodes is clear. Wisely choosing the weights can drastically affect convergence even when restricted to equispaced data. We never recommend using the Riemann sum, as one immediately obtains a better convergence rate by altering just two weights to employ a trapezoidal rule. 
 
\subsubsection{Lorenz system}

\begin{figure}[!tbp]
 \centering
\begin{minipage}[b]{0.32\textwidth}
  \begin{overpic}[width=\textwidth,trim={0mm 0mm 0mm 0mm},clip]{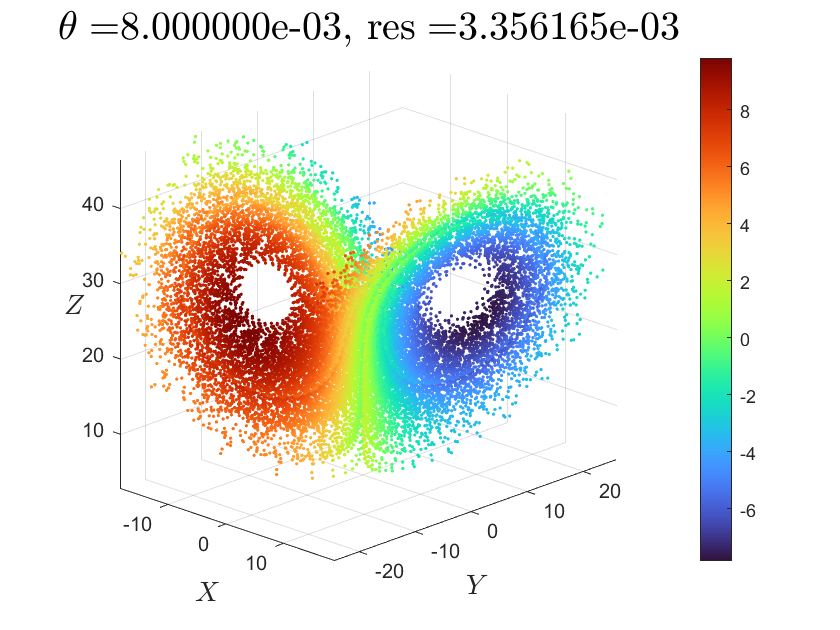}
   \end{overpic}
 \end{minipage}
\begin{minipage}[b]{0.32\textwidth}
  \begin{overpic}[width=\textwidth,trim={0mm 0mm 0mm 0mm},clip]{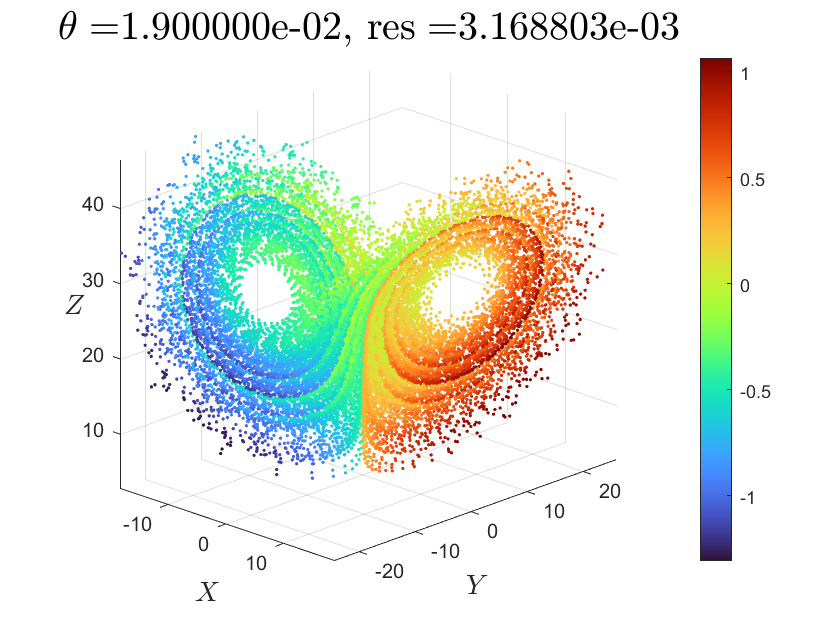}
   \end{overpic}
 \end{minipage}
 \begin{minipage}[b]{0.32\textwidth}
  \begin{overpic}[width=\textwidth,trim={0mm 0mm 0mm 0mm},clip]{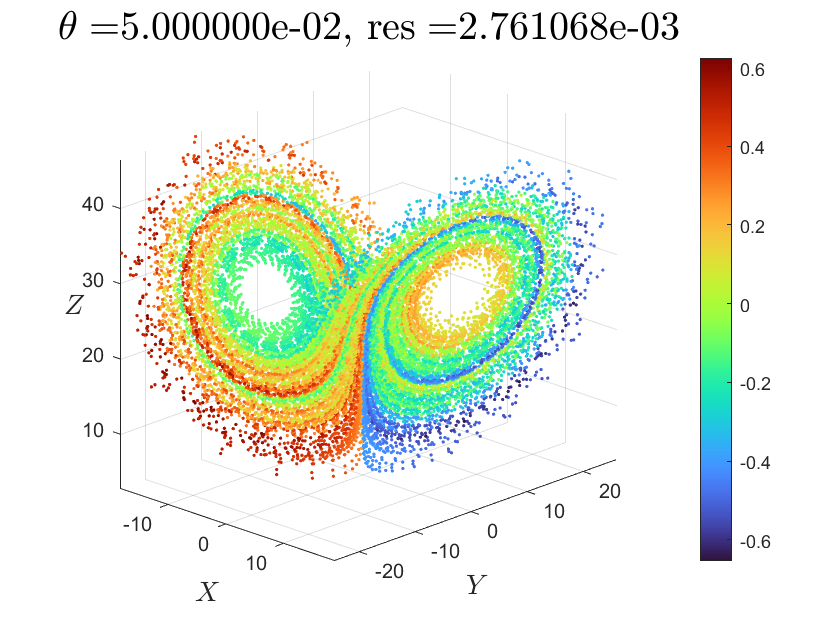}
   \end{overpic}
 \end{minipage}
\begin{minipage}[b]{0.32\textwidth}
  \begin{overpic}[width=\textwidth,trim={0mm 0mm 0mm 0mm},clip]{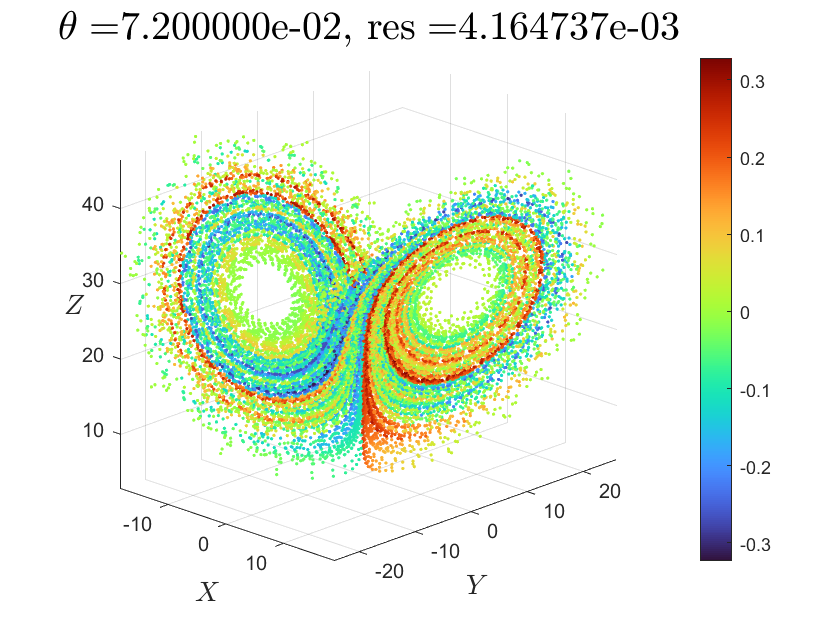}
   \end{overpic}
 \end{minipage}
\begin{minipage}[b]{0.32\textwidth}
  \begin{overpic}[width=\textwidth,trim={0mm 0mm 0mm 0mm},clip]{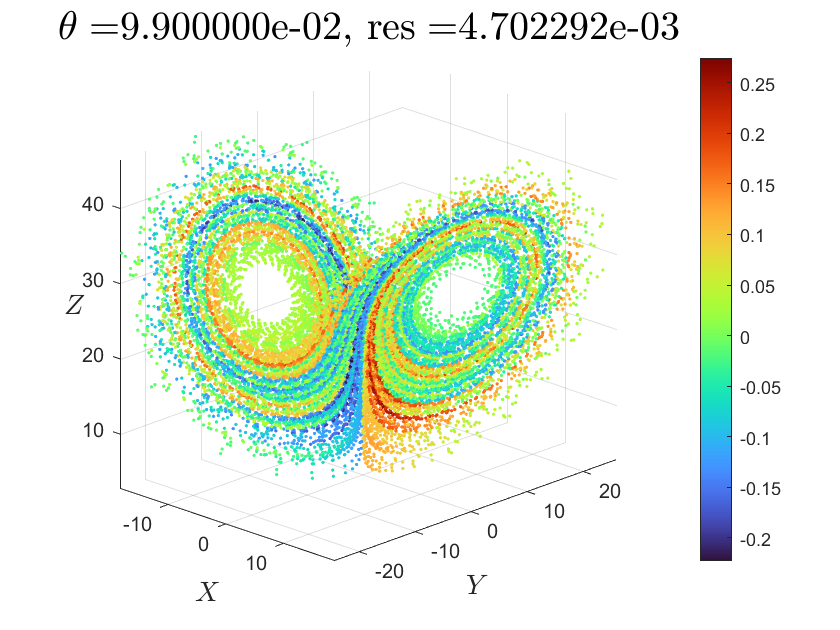}
   \end{overpic}
 \end{minipage}
\begin{minipage}[b]{0.32\textwidth}
  \begin{overpic}[width=\textwidth,trim={0mm 0mm 0mm 0mm},clip]{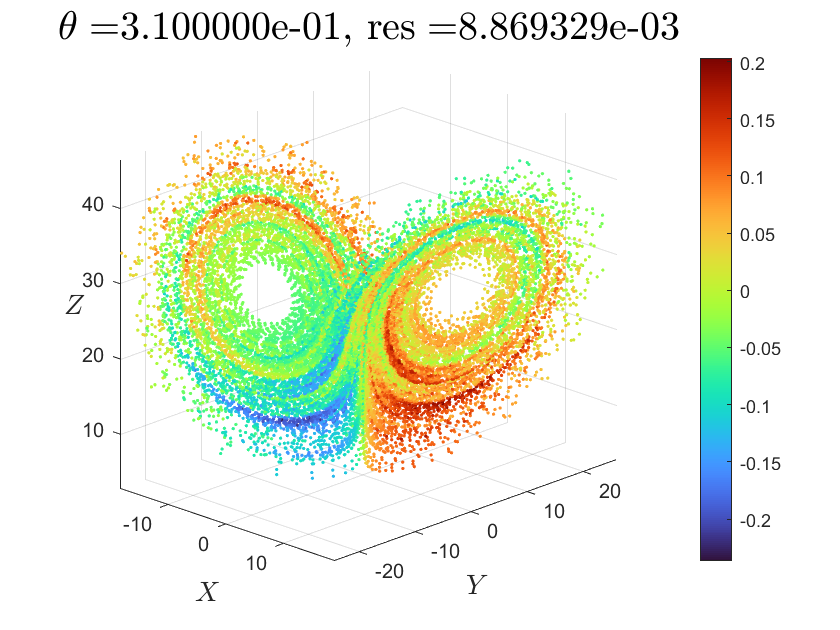}
   \end{overpic}
 \end{minipage}
\begin{minipage}[b]{0.32\textwidth}
  \begin{overpic}[width=\textwidth,trim={0mm 0mm 0mm 0mm},clip]{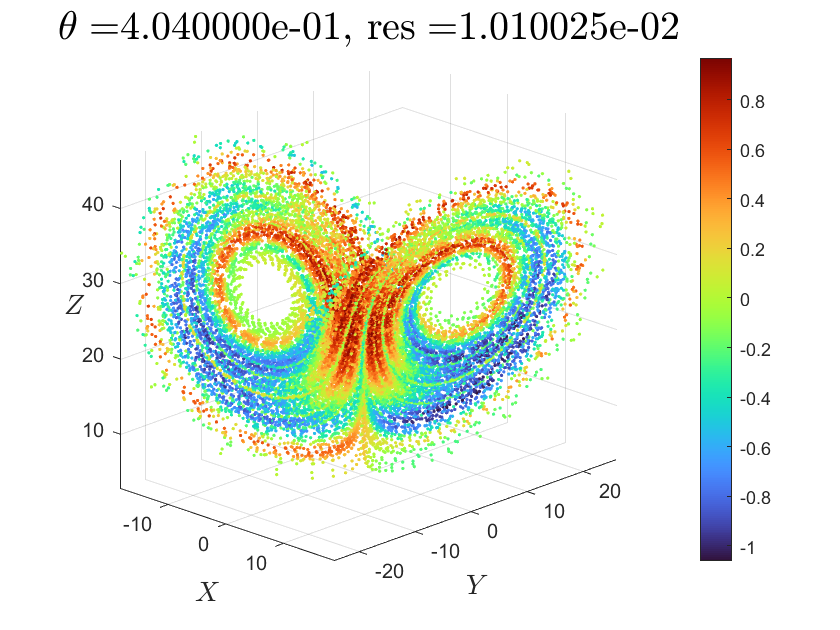}
   \end{overpic}
 \end{minipage}
\begin{minipage}[b]{0.32\textwidth}
  \begin{overpic}[width=\textwidth,trim={0mm 0mm 0mm 0mm},clip]{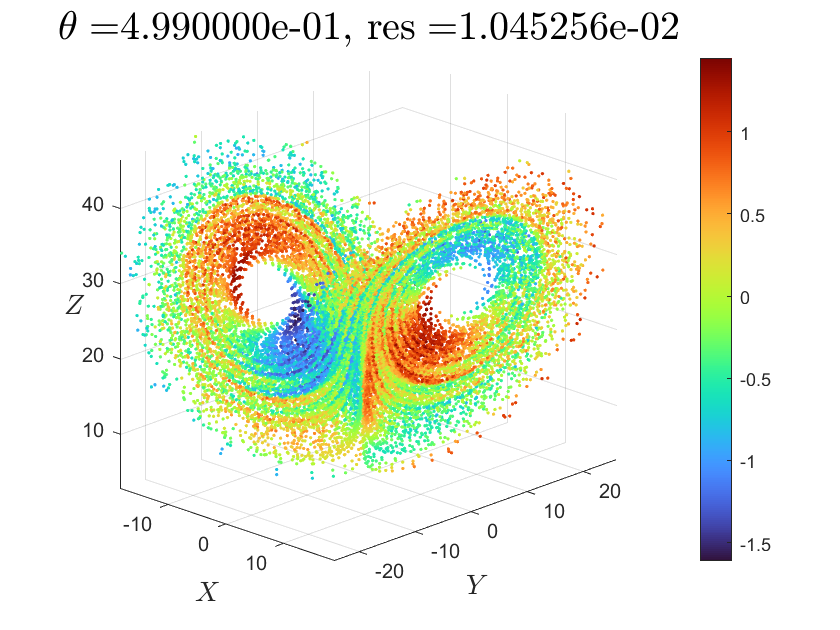}
   \end{overpic}
 \end{minipage}
\begin{minipage}[b]{0.32\textwidth}
  \begin{overpic}[width=\textwidth,trim={0mm 0mm 0mm 0mm},clip]{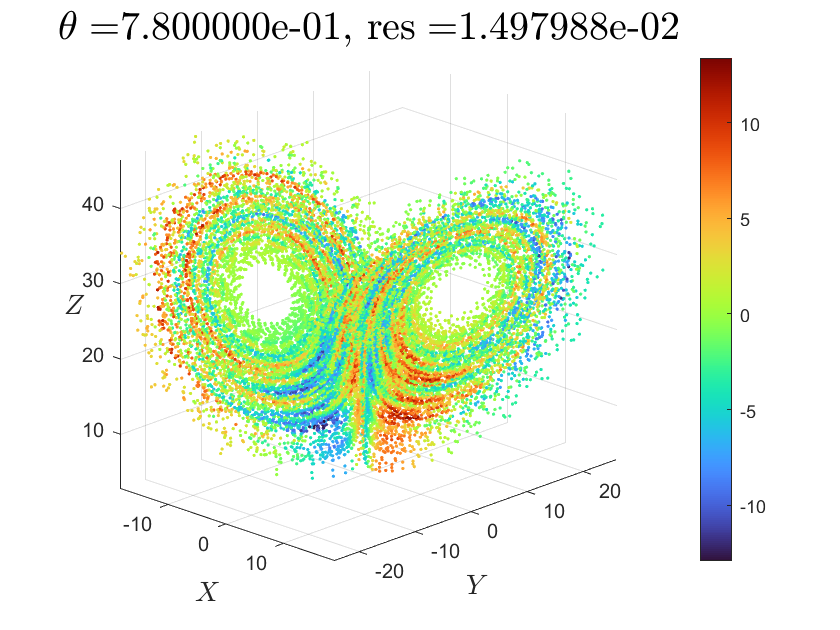}
   \end{overpic}
 \end{minipage}
\caption{The approximate eigenfunctions of the Lorenz system (real part shown), computed using \cref{alg:res_EDMD} with delay embedding. The spectral parameter is given by $\lambda=\exp(i\theta)$, and we show the residuals computed using~\eqref{eq:abs_res}.} 
\label{fig:lorenz_pseudo}
\end{figure}

The Lorenz system~\cite{lorenz1963deterministic} is the following system of three coupled ordinary differential equations~\cite{saltzman1962finite}:
$$
\dot{X}=\sigma\left(Y-X\right),\quad\dot{Y}=X\left(\rho-Z\right)-Y,\quad \dot{Z}=XY-\beta Z.
$$
The parameters $\sigma$, $\rho$, and $\beta$ are proportional to the Prandtl number, Rayleigh number, and the physical proportions of the fluid, respectively. The system describes a truncated model of Rayleigh--B\'{e}nard convection. We take $\sigma=10$, $\beta=8/3$, and $\rho=28$, corresponding to the original system studied by Lorenz.

We consider the Koopman operator corresponding to the time-step $\Delta_t=0.05$, with dynamics on the attractor. For this $\Omega$, one cannot explicitly write down a suitable dictionary with knowledge of the attractor. We use delay-embedding of the observables $g_j(\pmb{x})=[\pmb{x}]_j$, where the subscript corresponds to the $j$th coordinate. As a dictionary, we consider a union of Krylov subspaces
$$
\{g_1,\mathcal{K}g_1,\ldots,\mathcal{K}^{99}g_1,g_2,\mathcal{K}g_2,\ldots,\mathcal{K}^{99}g_2,g_3,\mathcal{K}g_3,\ldots,\mathcal{K}^{99}g_3\}.
$$
Time-delay embedding is a popular method for DMD-type algorithms~\cite{arbabi2017ergodic,kamb2020time}. For this example, we build the matrices $\Psi_X$ and $\Psi_Y$ using a single trajectory of length $M+1=10^5+1$ with a random initial condition on the attractor.\footnote{Another suitable choice of dictionary for this example is a set of radial basis functions with centers on the attractor sampled according to the SRB measure. One could use relatively short trajectories to compute the centers and build $\Psi_X$ and $\Psi_Y$ using multiple trajectories of length one from the centers.}~\cref{fig:lorenz_pseudo} shows various approximate eigenfunctions computed using \cref{alg:res_EDMD}. We have visualized each function as a cloud of points on the attractor. Some of these modes bear a resemblance to unstable periodic orbits, which in a sense, form a backbone of the attractor \cite{eckmann1985ergodic,tufillaro1992experimental}.

\section{Computing spectral measures using a ResDMD framework}\label{res_kern_fin}
In \cref{sec:section3_filter_fin}, we developed algorithms that computed the spectral properties of isometries from autocorrelations. Computing autocorrelations typically requires data collected as long trajectories, i.e., large $M_2$. In this section, we develop methods that build on ResDMD to compute spectral measures of Koopman operators for measure-preserving dynamical systems from arbitrary snapshot data. Throughout this section, we assume that $\mathcal{K}$ is an isometry and that we have access to snapshot data, $B_{\rm data}$, of the form given in~\eqref{data:snapshots}. That is, we have pairs $\{(\pmb{x}^{(j)},\pmb{y}^{(j)}=F(\pmb{x}^{(j)}))\}_{j=1}^M$. We develop rational kernels that allow us to compute smoothed approximations of spectral measures from the ResDMD matrices $\Psi_X^*W\Psi_X$, $\Psi_X^*W\Psi_Y$, and $\Psi_Y^*W\Psi_Y$ (see~\cref{extendEDMD}). Moreover, these matrices can be reused to compute spectral measures with respect to different observable functions $g$.

\subsection{Generalized Cauchy transform and the Poisson kernel for the unit disc}
As in \cref{sec:unitaryExtension}, we begin by considering a unitary extension $\mathcal{K}'$ of $\mathcal{K}$, which is defined on a Hilbert space $\mathcal{H}'$ that is an extension of $L^2(\Omega,\omega)$. Let $z\in\mathbb{C}$ with $|z|> 1$ and $g\in L^2(\Omega,\omega)$. Since $\|\mathcal{K}\|=1<|z|$ for $z\not\in\sigma(\mathcal{K})$, $(\mathcal{K}'-z)^{-1}g=(\mathcal{K}-z)^{-1}g$ and
\begin{equation}
\label{eq:keyres1}
\langle (\mathcal{K}-z)^{-1}g,\mathcal{K}^*g\rangle=\langle \mathcal{K}'(\mathcal{K}'-z)^{-1}g,g\rangle_{\mathcal{H}'}=\int_{\mathbb{T}}\frac{\lambda}{\lambda-z}\,d\mu_g(\lambda)=\int_{[-\pi,\pi]_{\mathrm{per}}}\frac{e^{i\theta}}{e^{i\theta}-z}\, d\nu_g(\theta),
\end{equation}
where the last equality follows from a change-of-variables.\footnote{Here we use the fact that the resolvent can be written as $(\mathcal{K}'-z)^{-1}=\int_{\mathbb{T}}(\lambda-z)\,d\mathcal{E}(\lambda)$, where $\mathcal{E}$ is the projection-valued spectral measure of $\mathcal{K}'$. For expansions of the resolvent in the context of different dynamical systems, see \cite{susuki2021koopman}.} If $z\neq 0$ with $|z|<1$, then $z$ may be in $\sigma(\mathcal{K})$ since $\mathcal{K}$ is not necessarily unitary. However, since $|\overline{z}^{-1}|>1$, $\overline{z}^{-1}\not\in\sigma(\mathcal{K})$ and hence $(\mathcal{K}'-\overline{z}^{-1})^{-1}g=(\mathcal{K}-\overline{z}^{-1})^{-1}g$. Since $\nu_g$ is a real-valued measure, we find that 
\begin{equation}
\label{eq:keyres1b}
\langle g,(\mathcal{K}-\overline{z}^{-1})^{-1}g\rangle=\langle g,(\mathcal{K}'-\overline{z}^{-1})^{-1}g\rangle_{\mathcal{H}'}=\overline{\int_{[-\pi,\pi]_{\mathrm{per}}}\frac{1}{e^{i\theta}-\overline{z}^{-1}}\, d\nu_g(\theta)}
=-z\int_{[-\pi,\pi]_{\mathrm{per}}}\frac{e^{i\theta}}{e^{i\theta}-z}\, d\nu_g(\theta).
\end{equation}
The leftmost and rightmost sides of \eqref{eq:keyres1} and \eqref{eq:keyres1b} are independent of $\mathcal{K}'$, so we can safely dispense with the extension and have an expression for a generalized Cauchy transform of $\nu_g$, i.e., 
\begin{equation}
\label{eq:keyres2}
\mathtt{C}_{\nu_g}\!(z)=\frac{1}{2\pi}\int_{[-\pi,\pi]_{\mathrm{per}}}\frac{e^{i\theta}}{e^{i\theta}-z}\, d\nu_g(\theta)=\frac{1}{2\pi}\begin{dcases}
\langle (\mathcal{K}-z)^{-1}g,\mathcal{K}^*g\rangle,\quad &\text{if $|z|>1$},\\
-z^{-1}\langle g,(\mathcal{K}-\overline{z}^{-1})^{-1}g\rangle,\quad &\text{if $z\neq0$ with $|z|<1$}.
\end{dcases}
\end{equation}
The importance of \eqref{eq:keyres2} is that it relates $\mathtt{C}_{\nu_g}$ to the resolvent operator $(\mathcal{K}-z)^{-1}$ for $|z|>1$. In~\cref{sec:comput_res}, we show how to compute the resolvent operator from snapshot data for $|z|>1$. Since $|z|>1$, we can provide convergence and stability results even when we replace $\mathcal{K}$ with a discretization.

To recover $\nu_g$ from $\mathtt{C}_{\nu_g}$, a derivation motivated by the Sokhotski--Plemelj formula shows that
\begin{equation}
\label{cauchy_compcalc}
\mathtt{C}_{\nu_g}\!\left(e^{i\theta_0}(1+\epsilon)^{-1}\right)-\mathtt{C}_{\nu_g}\!\left(e^{i\theta_0}(1+\epsilon )\right)=\int_{[-\pi,\pi]_{\mathrm{per}}} K_{\epsilon}(\theta_0-\theta)\,d\nu_g(\theta),
\end{equation}
where $K_\epsilon$ is the Poisson kernel for the unit disc given in~\eqref{pois_def_kern}. The Poisson kernel for the unit disc is a first-order kernel (see~\cref{rat_kern_linear_system}). In practice, it is important to develop high-order kernels for at least three reasons: (1) A numerical tradeoff, (2) A stability tradeoff, and (3) For localization of the kernel. We now explain each of these reasons in turn.

\subsubsection{The need for high-order kernels I: A numerical tradeoff}\label{sec:high_order_needed1}

\begin{figure}
  \centering
  \begin{minipage}[b]{0.49\textwidth}
    \begin{overpic}[width=\textwidth,trim={0mm 0mm 0mm 0mm},clip]{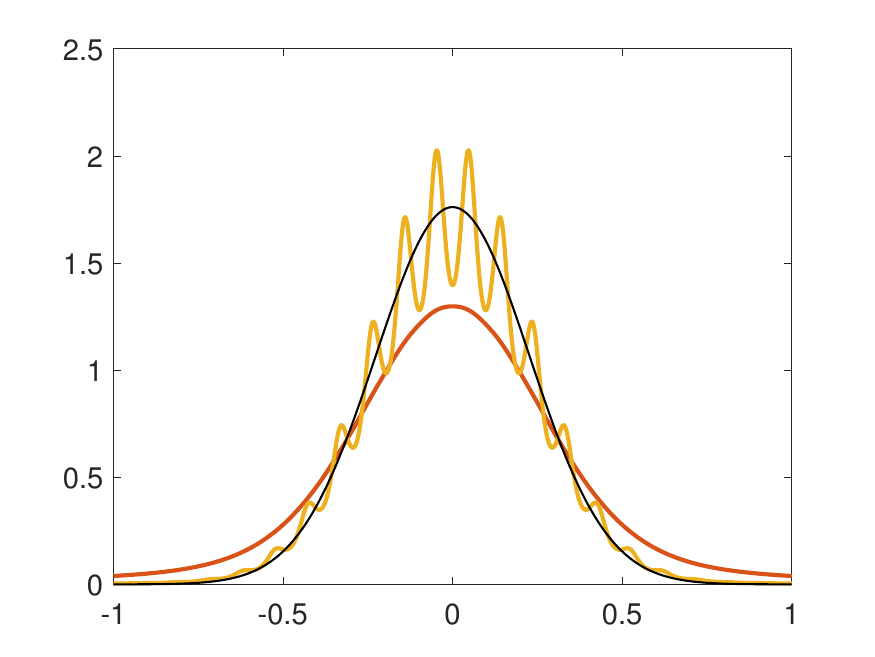}
		\put (10,73) {Convolution with Poisson kernel, $N_K=40$}
     \put (50,-3) {$\theta$}
		\put (73,53) { {$\displaystyle \rho_g$}}
	\put(73,53)  {\vector(-1,-1){14}}
	\put (17,22) { {$\displaystyle\nu_g^{0.1}$}}
	\put(20,20)  {\vector(1,-1){8}}
	\put (38,64) { {$\displaystyle\nu_g^{0.01}$}}
	\put(41,62)  {\vector(1,-1){8}}
     \end{overpic}
  \end{minipage}
	\begin{minipage}[b]{0.49\textwidth}
    \begin{overpic}[width=\textwidth,trim={0mm 0mm 0mm 0mm},clip]{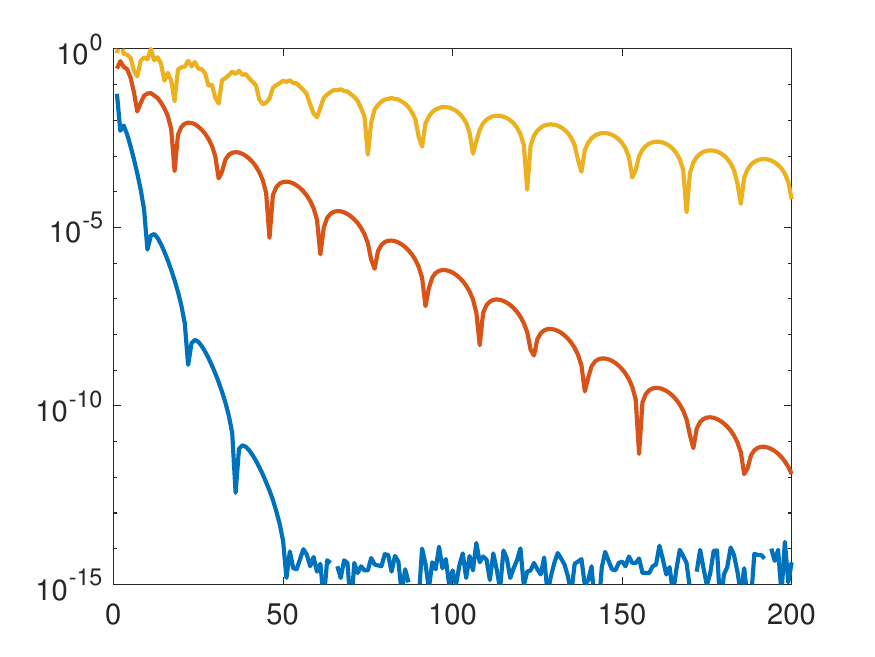}
		\put (23,73) {$|\nu_{g,N_K}^\epsilon(0.2)-\nu_{g}^\epsilon(0.2)|/|\nu_{g}^\epsilon(0.2)|$}
     \put (50,-3) {$N_K$}
     \put (70,61) {\rotatebox{-10}{$\epsilon = 0.01$}}
      \put (70,35) {\rotatebox{-25}{$\epsilon = 0.1$}}
       \put (70,13) {\rotatebox{0}{$\epsilon = 0.5$}}
     \end{overpic}
  \end{minipage}
    \caption{Left: Approximation of the spectral measure of $U$ in \eqref{U_mat_example} with respect to $g = e_1$ for a fixed truncation size of $N_K=40$. The oscillations for $\epsilon=0.01$ are an artifact caused by the discrete spectrum of the truncation. Right: The relative error in the approximation $\nu_{g,N_K}^\epsilon$ with varying $N_K$ of the smoothed measure $\nu_{g}^\epsilon=[K_\epsilon*\nu_g]$ for the Poisson kernel evaluated at $\theta=0.2$.} 
\label{fig:CMV1}
\end{figure}
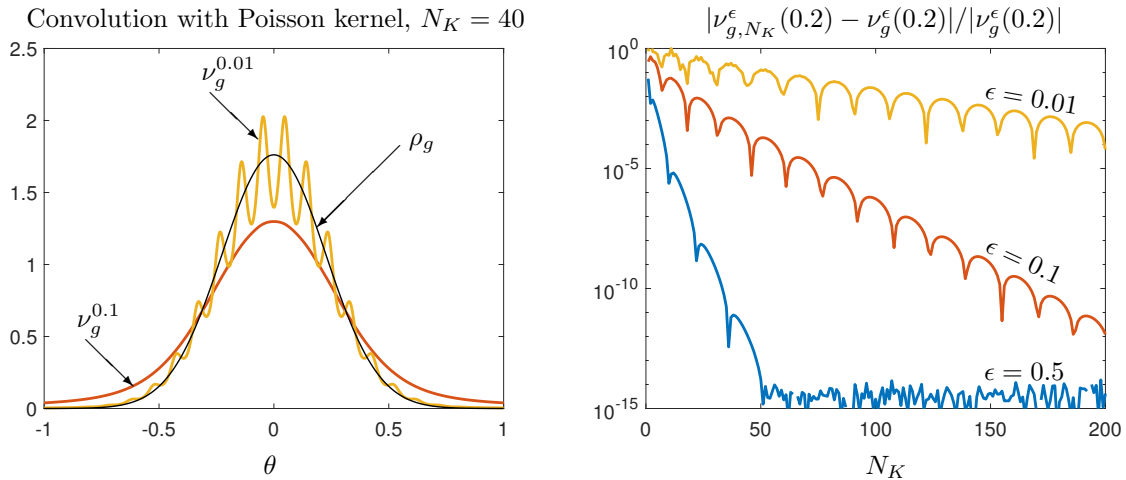

As a model problem, consider the following generalization of the shift operator in \cref{example_shift_operator}:
\begin{equation}
\label{U_mat_example}
U=\begin{pmatrix}
\overline{\alpha_0} & \overline{\alpha_1}\rho_0 & \rho_1\rho_0 & & & \\[3pt]
\rho_0 & -\overline{\alpha_1}\alpha_0 & -\rho_1\alpha_0 & 0 &  & \\[3pt]
0 & \overline{\alpha_2}\rho_1 & -\overline{\alpha_2}\alpha_1 & \overline{\alpha_3}\rho_2 & \rho_3\rho_2 & \\
& \rho_2\rho_1 & -\rho_2\alpha_1 & -\overline{\alpha_3}\alpha_2 & -\rho_3\alpha_2& \ddots\\
 &  & 0 & \overline{\alpha_4}\rho_3 & -\overline{\alpha_4}\alpha_3 & \ddots\\
 & & &\ddots &\ddots &\ddots 
\end{pmatrix},\quad \alpha_j=(-1)^j0.95^{\frac{j+1}{2}}, \quad \rho_j=\sqrt{1-|\alpha_j|^2}.
\end{equation}
We seek the spectral measure with respect to the first canonical unit vector $g=e_1$, which is a cyclic vector. $U$ is the Cantero--Moral--Vel\'{a}zquez (CMV) matrix for the Rogers--Szeg{\H o} orthogonal polynomials on the unit circle~\cite{rogers1893second}. CMV matrices for measures on the unit circle are analogous to Jacobi matrices for measures on the real line~\cite{simon2007cmv}.\footnote{One applies Gram--Schmidt to $\{1,z,z^{-1},z^2,z^{-2},\ldots\}$, which leads to a pentadiagonal matrix.} The spectral measure $\nu_g$ is absolutely continuous with an analytic Radon--Nikodym derivative given by
$$
\rho_g(\theta)=\sqrt{\frac{2\pi}{\log(0.95^{-1})}}\sum_{m=-\infty}^\infty\exp\left(-\frac{(\theta-2\pi m)^2}{2\log(0.95^{-1})}\right).
$$

To compute the resolvent and hence $\mathtt{C}_{\nu_g}$, we use finite square truncations of the matrix $U$ (see \cref{sec:comput_res}). Since $U$ is banded, once it is truncated to an $N_K\times N_K$ matrix it costs $\mathcal{O}(N_K)$ operations to compute $\mathtt{C}_{\nu_g}$. The dominating computational cost in evaluating $[K_\epsilon*\nu_g]$ is solving the shifted linear systems to compute the resolvent $(\mathcal{K}-e^{i\theta_0}(1+\epsilon))^{-1}$. This cost increases as $\epsilon\downarrow 0$ and $e^{i\theta_0}(1+\epsilon)$ approaches the continuous spectrum of $\mathcal{K}$. This increase in cost is due to a larger truncation size of $U$ needed to capture the function $(\mathcal{K}-e^{i\theta_0}(1+\epsilon))^{-1}g$ because it becomes more oscillatory.  

\cref{fig:CMV1} (left) shows the approximations $[K_\epsilon*\nu_g]$ for the Poisson kernel and $\epsilon=0.1,0.01$ using $N_K=40$ basis functions for the operator $U$ in \eqref{U_mat_example}. The exact Radon--Nikodym derivative is also shown. Whilst the computation for $\epsilon=0.1$ is an accurate approximation of $[K_\epsilon*\nu_g]$, we must decrease $\epsilon$ to approximate $\rho_g$. However, decreasing $\epsilon$ to $0.01$ introduces oscillatory artifacts. This is because we have fixed the truncation size, and the truncated operator's spectral measure is a sum of Dirac delta distributions located in the closed unit disc. To avoid this, we must select $N_K$ adaptively for a given $\epsilon$ (see~\cref{sec:comput_res}). \cref{fig:CMV1} (right) shows the relative error of approximating $[K_\epsilon*\nu_g](0.2)$ for a truncation size $N_K$ and varying $\epsilon$. A smaller $\epsilon$ leads to slower convergence as the truncation size increases.

Therefore, there is a numerical balancing act. On the one hand, we wish to stay as far away from the spectrum as possible so that the evaluation of $[K_\epsilon*\nu_g]$ is computationally cheap. On the other hand, we desire a good approximation of $\nu_g$, which requires small $\epsilon>0$. The Poisson kernel leads to $\mathcal{O}(\epsilon\log(\epsilon^{-1}))$ convergence as $\epsilon\downarrow0$, making it computationally infeasible to compute the spectral measure to more than a few digits of accuracy. High-order kernels allow us to compute spectral measures with better accuracy. This trade-off is significant for data-driven computations since we want to keep $N_K$ as small as possible. 
 
\subsubsection{The need for high-order kernels II: A stability tradeoff}\label{sec:high_order_needed2}

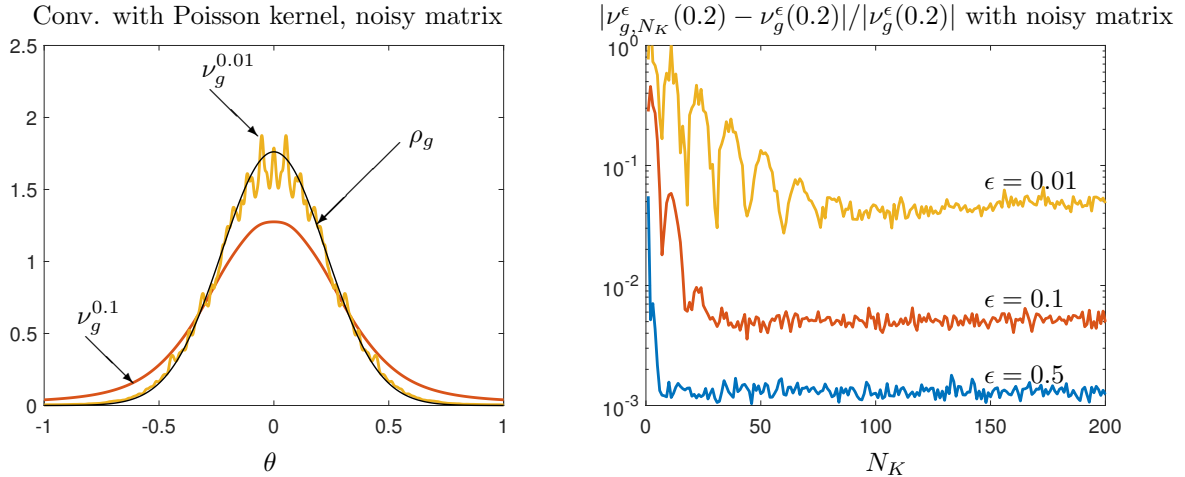
\begin{figure}
  \centering
  \begin{minipage}[b]{0.49\textwidth}
    \begin{overpic}[width=\textwidth,trim={0mm 0mm 0mm 0mm},clip]{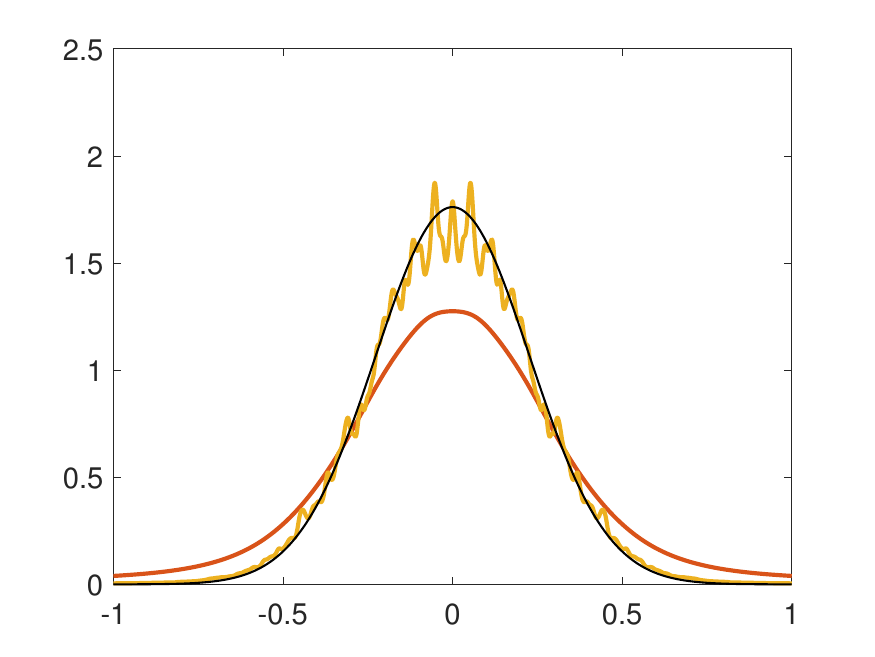}
		\put (11,73) {Conv. with Poisson kernel, noisy matrix}
     \put (50,-3) {$\theta$}
		\put (73,53) { {$\displaystyle \rho_g$}}
	\put(73,53)  {\vector(-1,-1){14}}
	\put (17,22) { {$\displaystyle\nu_g^{0.1}$}}
	\put(20,20)  {\vector(1,-1){8}}
	\put (38,64) { {$\displaystyle\nu_g^{0.01}$}}
	\put(41,62)  {\vector(1,-1){8}}
     \end{overpic}
  \end{minipage}
	\begin{minipage}[b]{0.49\textwidth}
    \begin{overpic}[width=\textwidth,trim={0mm 0mm 0mm 0mm},clip]{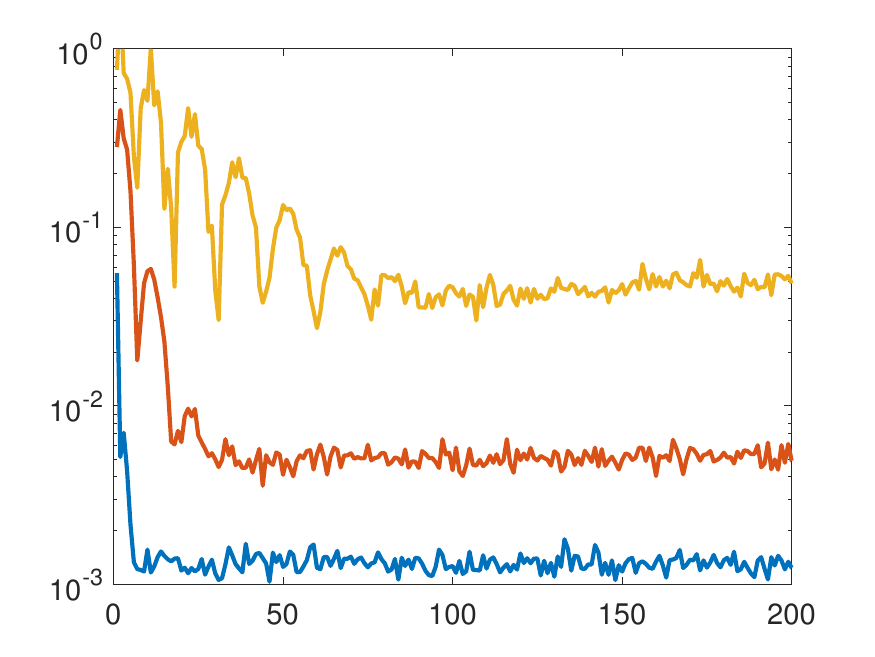}
		\put (5,73) {$|\nu_{g,N_K}^\epsilon(0.2)-\nu_{g}^\epsilon(0.2)| / |\nu_{g}^\epsilon(0.2)|$ with noisy matrix}
     \put (50,-3) {$N_K$}
		\put (70,45) {{$\epsilon = 0.01$}}
      \put (70,25) {{$\epsilon = 0.1$}}
       \put (70,12) {{$\epsilon = 0.5$}}
     \end{overpic}
  \end{minipage}
    \caption{Left: Approximation of the spectral measure of $U$ in \eqref{U_mat_example} with respect to $g = e_1$ for an adaptive choice of $N_K$, but with the truncated matrix perturbed by normal random variables of mean $0$ and standard deviation $\delta=0.001$. The oscillations for $\epsilon=0.01$ are an artifact caused by the ill-conditioning of the resolvent as $\epsilon\downarrow 0$. Right: The relative error in the numerical approximation $\nu_{g,N_K}^\epsilon$, corresponding to truncation size $N_K$, of the smoothed measure $\nu_{g}^\epsilon=[K_\epsilon*\nu_g]$ for the Poisson kernel in \eqref{pois_def_kern} with the same model of matrix perturbation. The error is averaged over 50 realizations.} 
\label{fig:CMV2}
\end{figure}

Continuing with~\eqref{U_mat_example}, the shifted linear system involved in computing the resolvent operator becomes increasingly ill-conditioned as $\epsilon\downarrow 0$ when $e^{i\theta_0}\in\sigma(\mathcal{K})$. This is because $\|(\mathcal{K}-e^{i\theta_0}(1+\epsilon))^{-1}\|=\mathcal{O}(\epsilon^{-1})$ (see \cref{sec:stability_argument}). The ResDMD framework developed in \cref{sec:RES_DMD} uses approximations of Galerkin truncations of $\mathcal{K}$, with convergence to the exact truncation in the large data limit $M\rightarrow\infty$ as the number of data points increases (see \cref{sec:matrix_conv_galerkin}). Therefore, the stability of the resolvent operator is crucial for data-driven approximations as we have inexact matrix entries. 

We now repeat the experiment from \cref{sec:high_order_needed1}, but with an adaptive $N_K$ for fixed $\epsilon$, and we perturb each entry of an $N_K\times N_K$ truncation of $U$ with independent Gaussian random variables of mean $0$ and standard deviation $\delta=0.001$. \cref{fig:CMV2} (left) shows the approximation of the spectral measure $\nu_g$ with $g = e_1$. While $\epsilon=0.1$ appears stable to the eye, decreasing to $\epsilon=0.01$ yields an inaccurate approximation since $\|(\mathcal{K}-e^{i\theta_0}(1+\epsilon))^{-1}\|$ increases and the shifted linear systems become more sensitive to noise. \cref{fig:CMV2} (right) shows the mean error of the approximation of $[K_\epsilon*\nu_g]$ over 50 independent realizations. In contrast to the convergence to machine precision in \cref{fig:CMV1} (right), the errors in \cref{fig:CMV2} (right) plateau because of the perturbation of the matrix, with a larger final error for smaller $\epsilon$.

The error in approximating $\nu_g$ with a perturbed truncation of the operator is unavoidable, but the Poisson kernel is suboptimal in this regard. For a perturbation of order $\delta$, the accuracy achieved by the Poisson kernel scales as $\mathcal{O}(\epsilon\log(\epsilon^{-1})+\delta/\epsilon)$. Balancing these terms while ignoring logarithmic factors leads to the noise-dependent choice $\epsilon\sim\delta^{1/2}$. This choice yields an overall error scaling of $\mathcal{O}(\delta^{1/2}\log(\delta^{-1}))$. High-order kernels can get arbitrarily close to the optimal scaling of $\mathcal{O}(\delta)$ as $\delta\rightarrow 0$.

\subsubsection{The need for high-order kernels III: Localization of the kernel}\label{sec:high_order_needed3}
We revisit the tent map of \cref{sec:tent_map} with the choice of $g$ given in \eqref{f_tent2}. \cref{fig:CMV3} shows the computed measure using the Poisson kernel, where we have chosen $N_K$ adaptively to avoid the problem in~\cref{sec:high_order_needed1} and $\epsilon=0.1$ so that with $M=2^5=32$ the error in~\cref{sec:high_order_needed2} is negligible. For comparison, we do the same computation, but with a $6$th order kernel (see~\cref{sec:kern_cjo}). As a function, the Poisson kernel decays slowly away from $0$, and hence the singular part of the spectral measure at $\theta=0$ is not localized. In contrast, $m$th order kernels obey a decay condition (see~\eqref{unitary_decay}) so are more localized for a particular $\epsilon$. For this example, this localization leads to a better approximation of the Radon--Nikodym derivative away from the singularity. We have also considered examples where the spectral measure is singular and observed that increased localization has the added benefit of avoiding over-smoothing.

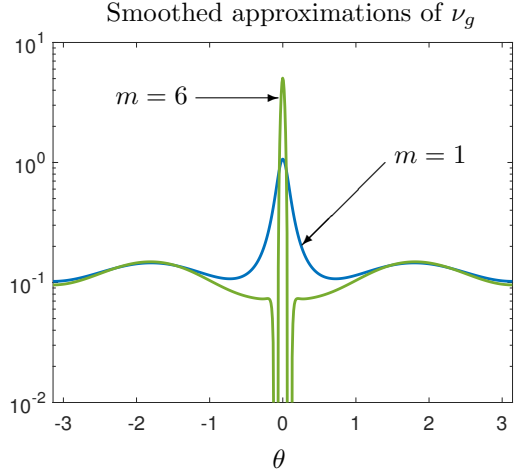
\begin{figure}
  \centering
  \begin{minipage}[b]{0.49\textwidth}
    \begin{overpic}[width=\textwidth,trim={0mm 0mm 0mm 0mm},clip]{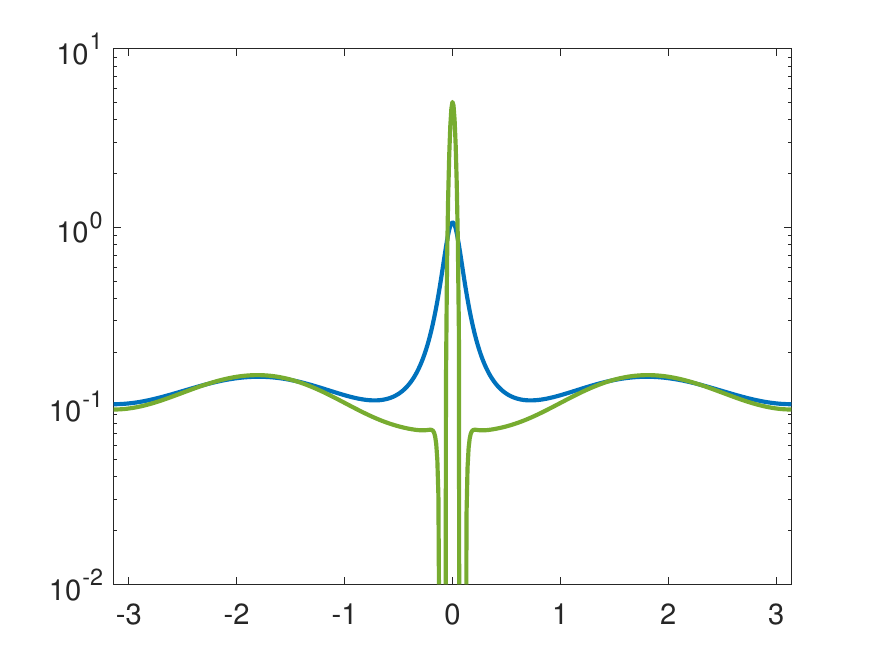}
		\put (22,73) {Smoothed approximations of $\nu_g$}
     \put (50,-3) {$\theta$}
		\put (69,49) { {$\displaystyle m=1$}}
	\put(69,49)  {\vector(-1,-1){14}}
	\put (22,59) { {$\displaystyle m=6$}}
	\put(37,60)  {\vector(1,0){14}}
     \end{overpic}
  \end{minipage}
    \caption{Convolutions of spectral measure with Poisson kernel ($m=1$) and a $6$th order rational kernel (see~\cref{sec:rat_kern_build22}). High-order kernels have better localization properties.}
\label{fig:CMV3}
\end{figure}

\subsection{Building high-order rational kernels}\label{sec:rat_kern_build22}
We now aim to generalize the Sokhotski--Plemelj-like formula in~\eqref{cauchy_compcalc} to develop high-order rational kernels. Let $\{z_j\}_{j=1}^{m}$ be distinct points with positive real part and consider the rational function: 
\begin{equation}\label{rat_kern_unitary}
K_{\epsilon}(\theta)=\frac{e^{-i\theta}}{2\pi}\sum_{j=1}^{m}\left[\frac{c_j}{e^{-i\theta}-(1+\epsilon \overline{z_j})^{-1}}-\frac{d_j}{e^{-i\theta}-(1+\epsilon z_j)}\right].
\end{equation}
A short derivation using \eqref{eq:keyres2} shows that
\begin{equation}
\label{unit_res_comp}
\begin{split}
\int_{[-\pi,\pi]_{\mathrm{per}}} K_{\epsilon}(\theta_0-\theta)\,d\nu_g(\theta)&=\sum_{j=1}^{m}\left[c_j\mathtt{C}_{\nu_g}\left(e^{i\theta_0}(1+\epsilon \overline{z_j})^{-1}\right)-d_j\mathtt{C}_{\nu_g}\left(e^{i\theta_0}(1+\epsilon z_j)\right)\right]\\
&\hspace{-30mm}=\frac{-1}{2\pi}\sum_{j=1}^{m}\big[c_je^{-i\theta_0}(1+\epsilon\overline{z_j})\langle g,(\mathcal{K}-e^{i\theta_0}(1+\epsilon z_j))^{-1}g \rangle+d_j\langle (\mathcal{K}-e^{i\theta_0}(1+\epsilon z_j))^{-1}g,\mathcal{K}^*g \rangle\big].
\end{split}
\end{equation}
It follows that we can compute the convolution $[K_\epsilon*\nu_g](\theta_0)$ by evaluating the resolvent at the $m$ points $\{e^{i\theta_0}(1+\epsilon z_j)\}_{j=1}^m$. We use rational kernels because they allow us to compute smoothed approximations of the spectral measure by applying the resolvent operator to functions. 

However, for~\eqref{unit_res_comp} to be a good approximation of $\nu_g$, we must carefully select the points $z_j$ and the coefficients $\{c_j,d_j\}$ in~\eqref{rat_kern_unitary}. In particular, we would like $\{K_\epsilon\}$ to be an $m$th order kernel. To design such a kernel, we carefully consider~\cref{unitary_rational_tool}. First, we define $\zeta_j(\epsilon)$ by the relationship $1+\epsilon\zeta_j(\epsilon)=(1+\epsilon \overline{z_j})^{-1}$ and use Cauchy's Residue theorem to show that for any integer $n\geq 1$,
\begin{equation*}
\begin{split}
\int_{[-\pi,\pi]_{\mathrm{per}}} K_{\epsilon}(-\theta) e^{in\theta}\,d\theta&=\frac{1}{2\pi i}\int_{\mathbb{T}}\left[\sum_{j=1}^{m}\frac{c_j}{\lambda-(1+\epsilon \overline{z_j})^{-1}}-\sum_{j=1}^{m}\frac{d_j}{\lambda-(1+\epsilon z_j)}\right]\lambda^{n}\,d\lambda\\
&=\sum_{j=1}^{m}c_j(1+\epsilon \overline{z_j})^{-n}=\left(\sum_{j=1}^mc_j\right)+\sum_{k=1}^{n} \epsilon^k \binom{n}{k} \sum_{j=1}^{m}c_j\zeta_j(\epsilon)^{k}.
\end{split}
\end{equation*}
It follows that condition~\eqref{fourier_cond} in~\cref{unitary_rational_tool} holds if
\begin{equation}\label{eqn:vandermonde_condition900}
\begin{pmatrix}
1 & \dots & 1 \\
{\zeta_1(\epsilon)} & \dots & {\zeta_{m}(\epsilon)} \\
\vdots & \ddots & \vdots \\
{\zeta_1(\epsilon)^{m-1}} &  \dots & {\zeta_{m}(\epsilon)^{m-1}}
\end{pmatrix}
\begin{pmatrix}
c_1(\epsilon) \\ c_2(\epsilon)\\ \vdots \\ c_{m}(\epsilon)
\end{pmatrix}
=\begin{pmatrix}
1 \\ 0 \\ \vdots \\0
\end{pmatrix}.
\end{equation}
Note also that, if this holds, the coefficients $c_j=c_j(\epsilon)$ remain bounded as $\epsilon\downarrow 0$. To ensure that the decay condition in~\eqref{bound_lemma} is satisfied, let $\omega=(e^{-i\theta}-1)/\epsilon$. The kernel in \eqref{rat_kern_unitary} can then be re-written as
\begin{equation}
K_{\epsilon}(\theta)=\frac{\epsilon^{-1}e^{-i \theta}}{2\pi}\sum_{j=1}^{m}\left[\frac{c_j}{\omega-\zeta_j(\epsilon)}-\frac{d_j}{\omega- z_j}\right].
\end{equation}
Therefore, we have
\begin{equation*}
\begin{split}
\omega K_{\epsilon}(\theta)&=\frac{\epsilon^{-1}e^{-i \theta}}{2\pi}\sum_{j=1}^{m}\left[c_j+\frac{c_j\zeta_j(\epsilon)}{\omega-\zeta_j(\epsilon)}-d_j-\frac{d_jz_j}{\omega- z_j}\right]\\
&=\frac{\epsilon^{-1}e^{-i \theta}}{2\pi}\sum_{j=1}^{m}(c_j-d_j)+\frac{\epsilon^{-1}e^{-i \theta}}{2\pi}\sum_{j=1}^{m}\left[\frac{c_j\zeta_j(\epsilon)}{\omega-\zeta_j(\epsilon)}-\frac{d_jz_j}{\omega- z_j}\right].
\end{split}
\end{equation*}
By repeating the same argument $m$ times, we arrive at
\begin{equation}
\label{trick87}
\omega^m K_{\epsilon}(\theta)=\frac{\epsilon^{-1}e^{-i \theta}}{2\pi}\left[\sum_{k=0}^{m-1}\omega^{m-1-k}\sum_{j=1}^{m}(c_j\zeta_j(\epsilon)^k-d_jz_j^k)+\sum_{j=1}^{m}\left(\frac{c_j\zeta_j(\epsilon)^m}{\omega-\zeta_j(\epsilon)}-\frac{d_jz_j^m}{\omega- z_j}\right)\right].
\end{equation}
This means that we should select the $d_k$'s so that
\begin{equation}\label{eqn:vandermonde_condition87}
\begin{pmatrix}
1 & \dots & 1 \\
z_1 & \dots & z_{m} \\
\vdots & \ddots & \vdots \\
z_1^{m-1} &  \dots & z_{m}^{m-1}
\end{pmatrix}
\begin{pmatrix}
d_1 \\ d_2\\ \vdots \\ d_{m}
\end{pmatrix}
=\begin{pmatrix}
1 \\ 0 \\ \vdots \\0
\end{pmatrix}.
\end{equation}
We conclude that if the coefficients $\{c_j\}_{j=1}^m$ and $\{d_j\}_{j=1}^m$ satisfy \eqref{eqn:vandermonde_condition900} and \eqref{eqn:vandermonde_condition87}, respectively, then
$$
\sum_{k=0}^{m-1}\omega^{m-1-k}\sum_{j=1}^{m}(c_j\zeta_j(\epsilon)^k-d_jz_j^k)=0, \qquad \left|\sum_{j=1}^{m}\left(\frac{c_j\zeta_j(\epsilon)^m}{\omega-\zeta_j(\epsilon)}-\frac{d_jz_j^m}{\omega- z_j}\right)\right|\lesssim |\omega|^{-1}.
$$
By~\eqref{trick87}, this means that $|\omega|^{m}\left|K_{\epsilon}(\theta)\right|\lesssim \epsilon^{-1}|\omega|^{-1}$. Moreover, since $\left|K_{\epsilon}(\theta)\right|\lesssim \epsilon^{-1}$ we see that $\left|K_{\epsilon}(\theta)\right|\lesssim \min\{\epsilon^{-1},\epsilon^{m}|\theta|^{-(m+1)}\}\lesssim \epsilon^m(\epsilon+|\theta|)^{-(m+1)}$. Using \cref{unitary_rational_tool}, we have proved the following proposition.

\begin{proposition}\label{rat_kern_linear_system}
Let $\{z_j\}_{j=1}^{m}$ be distinct points with positive real part and let $K_\epsilon$ be given by \eqref{rat_kern_unitary}. Then, $\{K_{\epsilon}\}$ is an $m$th order kernel for $[-\pi,\pi]_{\mathrm{per}}$ if the coefficients $\{c_j\}_{j=1}^m$ and $\{d_j\}_{j=1}^m$ satisfy \eqref{eqn:vandermonde_condition900} and \eqref{eqn:vandermonde_condition87}, respectively.
\end{proposition}

\subsubsection{The choice of rational kernel}\label{sec:kern_cjo}

\begin{figure}
  \centering
  \begin{minipage}[b]{0.5\textwidth}
    \begin{overpic}[width=\textwidth,trim={0mm 0mm 0mm 0mm},clip]{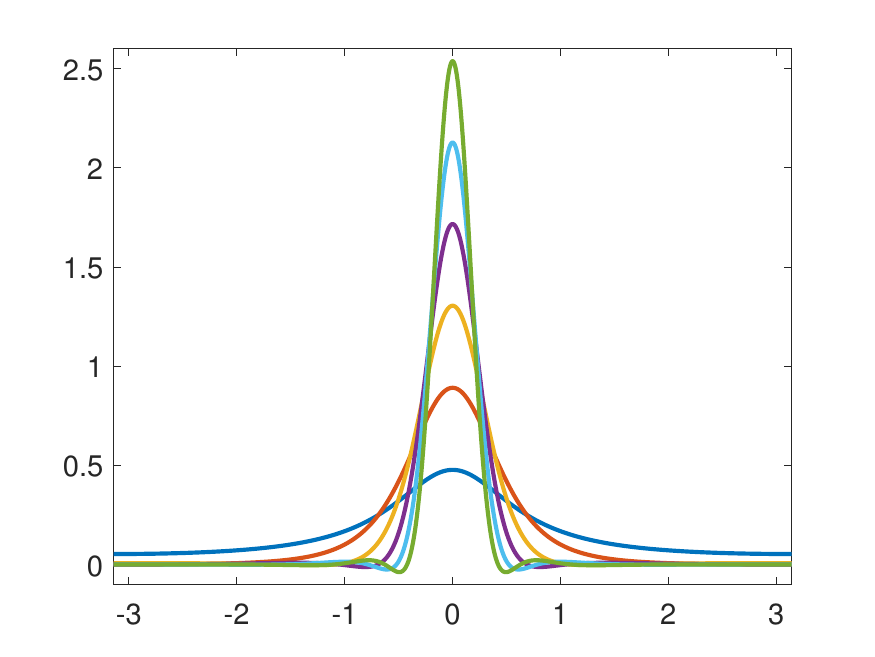}
		\put (42,73) {$\displaystyle \mathrm{Re}(K_1(\theta))$}
 	\put (50,-2) {$\displaystyle \theta$}
	\put (66,21) { {$\displaystyle m=1$}}
	\put(66,21)  {\vector(-1,0){14}}
	\put (19,30.5) { {$\displaystyle m=2$}}
	\put(37.5,30.5)  {\vector(1,0){14}}
	\put (66,40) { {$\displaystyle m=3$}}
	\put(66,40)  {\vector(-1,0){14}}
	\put (19,49) { {$\displaystyle m=4$}}
	\put(37.5,49)  {\vector(1,0){14}}
	\put (66,59) { {$\displaystyle m=5$}}
	\put(66,59)  {\vector(-1,0){14}}
	\put (19,66) { {$\displaystyle m=6$}}
	\put(37.5,67.5)  {\vector(1,0){14}}
     \end{overpic}
  \end{minipage}
    \caption{The $m$th-order kernels \eqref{rat_kern_unitary} constructed with the choice \eqref{gen_poisson} for $\epsilon=1$ and $1 \leq m \leq 6$.} 
\label{fig:rat_kernels}
\end{figure}
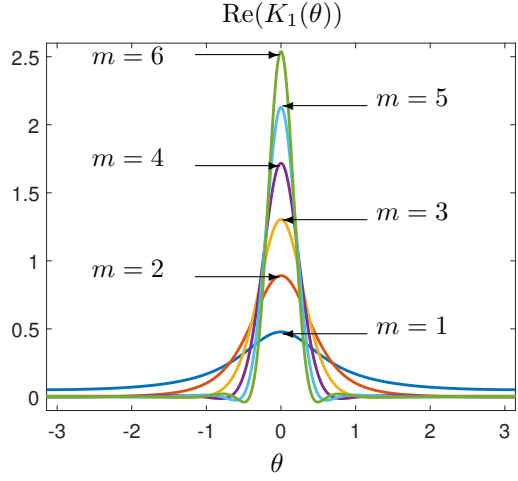

We are free to choose the points $\{z_j\}_{j=1}^{m}$ in \cref{rat_kern_linear_system} subject to $\mathrm{Re}(z_j)>0$, after which the linear systems \eqref{eqn:vandermonde_condition900} and \eqref{eqn:vandermonde_condition87} provide suitable $\{c_j\}_{j=1}^m$ (dependent on $\epsilon$) and $\{d_j\}_{j=1}^m$ (independent of $\epsilon$). As a natural extension of the Poisson kernel in \eqref{pois_def_kern}, we select the points $\{z_j\}_{j=1}^{m}$ as 
\begin{equation}
\label{gen_poisson}
z_j=1+\left(\frac{2j}{m+1}-1\right)i,\qquad 1\leq j\leq m.
\end{equation}
The kernels that we have developed are typically not real-valued. Since $\nu_g$ is a probability measure and hence real-valued, we often gain better accuracy for a particular $\epsilon$ by considering the kernel $\mathrm{Re}(K_{\epsilon})$, and this is what we do throughout this paper. The convolution with $\mathrm{Re}(K_{\epsilon})$ can be computed by taking the real part of the right-hand side of \eqref{unit_res_comp}. The first six kernels with the choice $\epsilon=1$ are shown in \cref{fig:rat_kernels}. The exact coefficients $\{c_j,d_j\}_{j=1}^m$ for $\epsilon=0.1$ are shown in \cref{rat_unit_kern_table} for the first six kernels.

\begin{table}
\centering
\begin{tabular}{l|c|c}
$m$ & $\{d_1,\ldots,d_{\ceil{m/2}}\}$ & $\{c_1(\epsilon),\ldots,c_{\ceil{m/2}}(\epsilon)\}$, $\epsilon=0.1$\\[3pt]
\hline
$2$ & $\left\{\frac{1-3i}{2}\right\}$ & $\left\{\frac{3+10i}{6}\right\}$\\[3pt]
$3$ & $\left\{-2-i,5\right\}$ & $\left\{\frac{-202+79i}{80},\frac{121}{20}\right\}$\\[3pt]
$4$ &  $\left\{\frac{-39+65i}{24},\frac{17-85i}{8}\right\}$ & $\left\{\frac{-1165710-2944643i}{750000},\frac{513570+3570527i}{250000}\right\}$\\[3pt]
$5$ &  $\left\{\frac{15+10i}{4},\frac{-39-13i}{2},\frac{65}{2}\right\}$ & $\left\{\frac{4052283 - 1460282i}{648000},\frac{-2393157+486551i}{81000},\frac{190333}{4000}\right\}$
\\[3pt]
$6$ & \tiny{$\left\{\frac{725-1015i}{192},\frac{-2775+6475i}{192},\frac{1073-7511i}{96}\right\}$} &\tiny{
$\left\{\frac{24883929805+81589072062i}{8067360000},\frac{-19967590755-93596942182i}{1613472000},\frac{7898770397+102424504746i}{806736000}\right\}$}\\
\end{tabular}
\normalsize
\caption{Coefficients in the rational kernels in~\eqref{rat_kern_unitary} for $1\leq m \leq 6$, the choice \eqref{gen_poisson}, and $\epsilon = 0.1$. We give the first $\lceil m/2\rceil$ coefficients as $c_{m+1-j}= \overline{c_j}$ and $d_{m+1-j}= \overline{d_j}$.}\label{rat_unit_kern_table}
\end{table}

\subsubsection{Revisiting the examples in \texorpdfstring{\cref{sec:high_order_needed1}}{} and \texorpdfstring{\cref{sec:high_order_needed2}}{}}\label{sec:stability_argument}
To demonstrate the advantages of high-order rational kernels, we revisit the operator $U$ given in \eqref{U_mat_example} and look again at two of the numerical needs for high-order kernels. 

\begin{itemize}[leftmargin=*,noitemsep]
\item {\bf A numerical tradeoff.}  \cref{fig:CMV4} (left) shows the convergence of $[K_\epsilon*\nu_g]$ for the rational kernels in \cref{sec:kern_cjo}. We see the convergence rates predicted from~\cref{thm:unitary_pointwise_convergence}. In practice, this means that we rarely need to take $\epsilon$ below $0.01$ for very accurate results, and often we take $\epsilon=0.1$.

\item{\bf A stability tradeoff.}
Let $K_{N_K}$ be an ${N_K}\times {N_K}$ truncation of $\mathcal{K}$ (see~\cref{sec:comput_res} for how we do this). An important property of our truncations is that $\|K_{N_K}\|\leq \|\mathcal{K}\|=1$. Using the standard Neumann series for the inverse, it follows that any $\lambda$ such that $|\lambda|>1$ has $\|(K_{N_K}-\lambda)^{-1}\|\leq 1/(|\lambda|-1)$. In general, this is the best one can hope for. Suppose that instead of $K_{N_K}$, we have access to $\tilde K_{N_K}=K_{N_K}+\Delta$ with $\|\Delta\|\leq\delta$. If $|\lambda|=1+\mathcal{O}(\epsilon)>1+\delta$, then since $\|\tilde K_{N_K}\|\leq 1+\delta$ we have $\|(\tilde K_{N_K}-\lambda)^{-1}\|\leq1/(|\lambda|-1-\delta)$. Using the second resolvent identity for the first equality, the relative error is bounded via 
\begin{align*}
\frac{\|[(\tilde K_{N_K}-\lambda)^{-1}-(K_{N_K}-\lambda)^{-1}]g\|}{\|(K_{N_K}-\lambda)^{-1}g\|}&=\frac{\|[-(\tilde K_{N_K}-\lambda)^{-1}\Delta] (K_{N_K}-\lambda)^{-1}g\|}{\|(K_{N_K}-\lambda)^{-1}g\|}\\
&\leq \|(\tilde K_{N_K}-\lambda)^{-1}\Delta\| \leq \frac{\delta}{|\lambda|-1-\delta} = \mathcal{O}(\epsilon^{-1}\delta).
\end{align*}
\cref{thm:unitary_pointwise_convergence,thm:unitary_weak_convergence} show that an $m$th order kernel provides error rates of $\mathcal{O}(\epsilon^m\log(\epsilon^{-1}))$ as $\epsilon\rightarrow 0$. Therefore, we have an error that is asymptotically of size $\mathcal{O}(\epsilon^m\log(\epsilon^{-1}))$ from using a kernel and $\mathcal{O}(\epsilon^{-1}\delta)$ from using a perturbed version of $K_{N_K}$. We assume that ${N_K}$ is sufficiently large to render the truncation error negligible. By selecting $\epsilon=\delta^{1/(m+1)}$, we balance these two errors and have an overall error of $\mathcal{O}(\delta^{\frac{m}{m+1}}\log(\delta^{-1}))$. For large $m$, the convergence rate is close to the optimal rate of $\mathcal{O}(\delta)$. In~\cref{fig:CMV4} (right), we show the relative error $|\nu_{g}^\epsilon(0.2)-\rho_{g}(0.2)|/|\rho_{g}(0.2)|$ when using truncations where each entry is perturbed by independent normal random variables of mean $0$ and standard deviation $\delta$.\footnote{Strictly speaking, the operator norm of this perturbation need not be bounded by $\delta$, but the scaling argument for the choice $\epsilon=\delta^{1/(m+1)}$ remains the same.} We have chosen ${N_K}$ sufficiently large and $\epsilon\sim \delta^{1/(m+1)}$. The error is averaged over $50$ trials so that we show the expected convergence rates. 

\end{itemize} 

\begin{figure}[!tbp]
  \centering
  \begin{minipage}[b]{0.49\textwidth}
    \begin{overpic}[width=\textwidth,trim={0mm 0mm 0mm 0mm},clip]{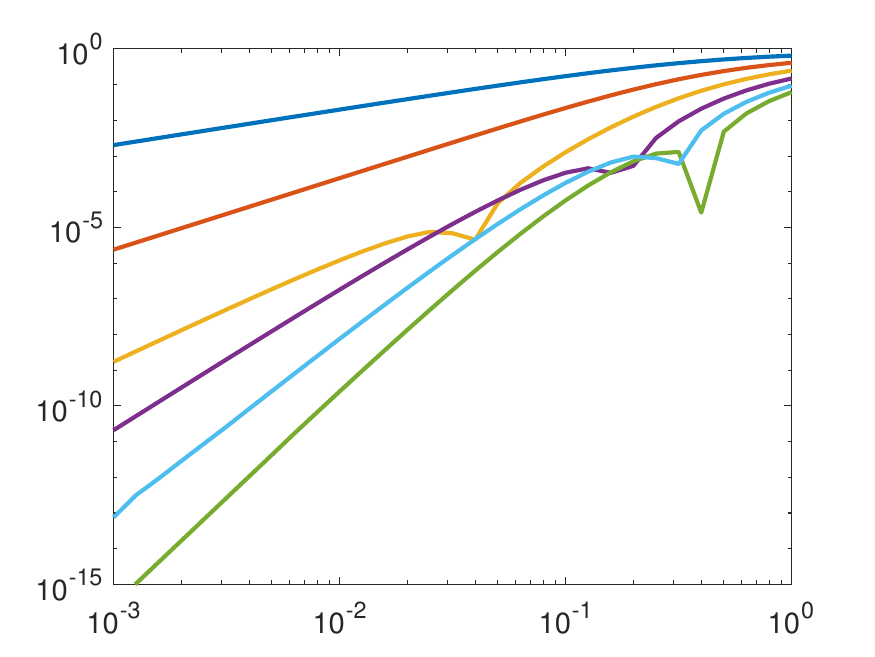}
		\put (20,73) {$|\rho_g(0.2)-[K_\epsilon*\nu_g](0.2)|/|\rho_g(0.2)|$}
     \put (50,-3) {$\epsilon$}
		\put (14,62) {\rotatebox{7} {$\displaystyle m=1$}}
	\put (14,49) {\rotatebox{18} {$\displaystyle m=2$}}
	\put (14,37) {\rotatebox{26} {$\displaystyle m=3$}}
	\put (14,28) {\rotatebox{34} {$\displaystyle m=4$}}
	\put (14,20) {\rotatebox{38} {$\displaystyle m=5$}}
	\put (20,10.5) {\rotatebox{40} {$\displaystyle m=6$}}
     \end{overpic}
  \end{minipage}
	\begin{minipage}[b]{0.49\textwidth}
    \begin{overpic}[width=\textwidth,trim={0mm 0mm 0mm 0mm},clip]{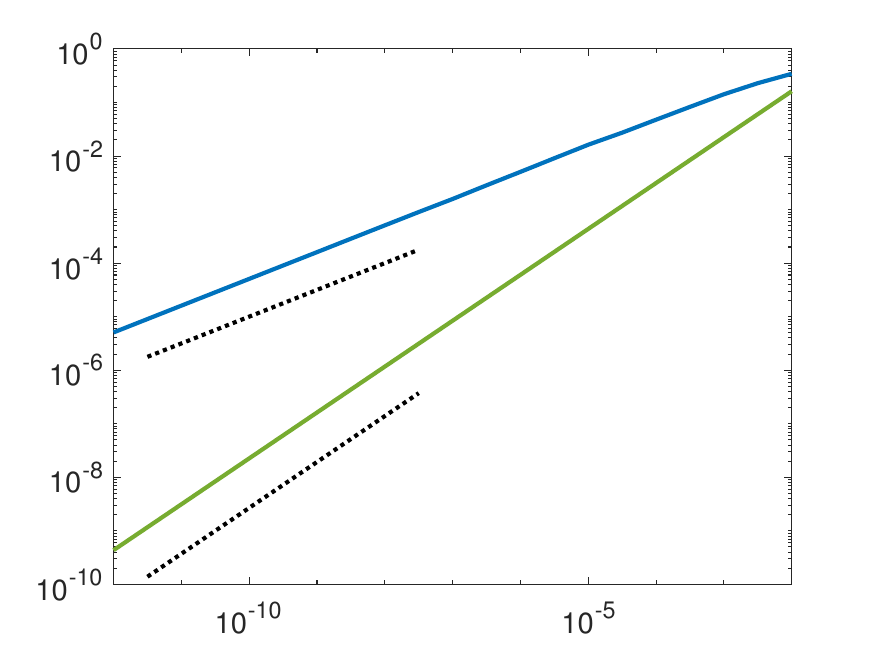}
		\put (15,73) {$\frac{|\rho_g(0.2)-[K_\epsilon*\nu_g](0.2)|}{|\rho_g(0.2)|}$ with noisy matrix}
     \put (50,-3) {$\delta$}
		\put (17,29) {\rotatebox{22} {$\mathcal{O}(\delta^{\frac{1}{2}}\log(\delta^{-1}))$}}
	\put (22,8) {\rotatebox{35} {$\mathcal{O}(\delta^{\frac{6}{7}}\log(\delta^{-1}))$}}
	\put (60,58) {\rotatebox{19} {$m = 1$}}
	\put (60,47) {\rotatebox{31} {$m = 6$}}
     \end{overpic}
  \end{minipage}
    \caption{Left: The pointwise relative error in smoothed measures of the operator in \eqref{U_mat_example} computed using the high-order kernels in \cref{sec:kern_cjo} for $1 \leq m \leq 6$. Right: The pointwise relative error in the computed smoothed measure using the perturbed $K_{N_K}$ and $\epsilon\sim \delta^{1/(m+1)}$.} 
\label{fig:CMV4}
\end{figure}
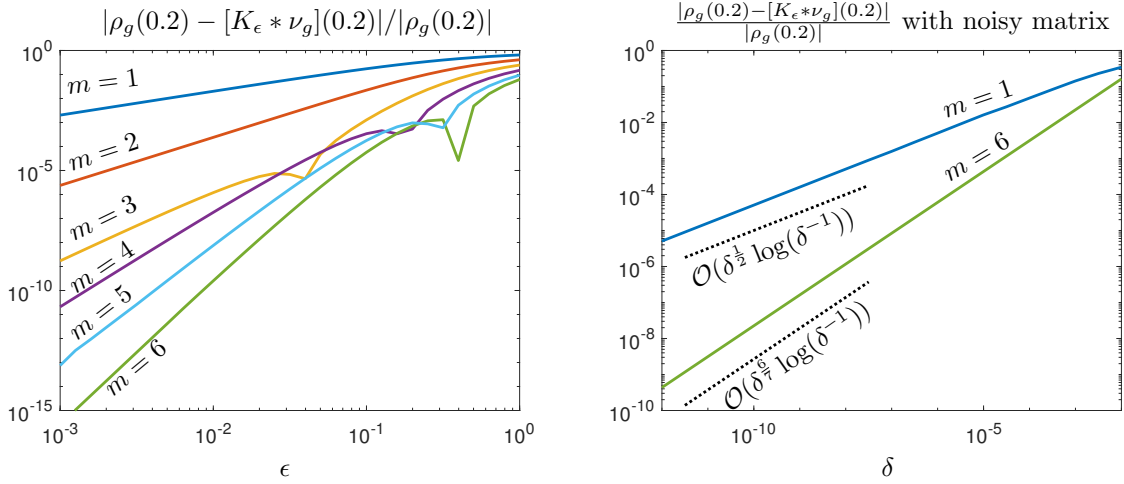

\subsection{An algorithm for evaluating a smoothed spectral measure}\label{sec:comput_res}
We now seek to evaluate $[K_\epsilon*\nu_g]$ at $\theta_1,\ldots,\theta_P\in[-\pi,\pi]_{\rm per}$, where $\{K_\epsilon\}$ is a rational kernel of the form \eqref{rat_kern_unitary}. This requires the evaluation of $(\mathcal{K}-\lambda)^{-1}$ for different values of $\lambda$ (see~\eqref{unit_res_comp}). We first describe how to evaluate $[K_\epsilon*\nu_g]$ at a single point $\theta_0$ before precomputing a factorization to derive an efficient scheme for evaluating at $\theta_1,\ldots,\theta_P$. 

\subsubsection{Evaluating a smoothed spectral measure at a single point}
To evaluate $[K_\epsilon*\nu_g]$ at a single point $\theta_0\in[-\pi,\pi]_{\rm per}$, we use the setup of ResDMD (see~\cref{extendEDMD}) so that we can obtain rigorous a posteriori error bounds on the computed resolvents, allowing us to adaptively select the dictionary size $N_K$ based on the smoothing parameter $\epsilon$. Since $\mathcal{K}$ is an isometry, we only need to compute $(\mathcal{K}-z)^{-1}$ for $|z|>1=\|\mathcal{K}\|$ and so we can achieve our goal with standard Galerkin truncations of $\mathcal{K}$.

\begin{theorem}
\label{res_square_truncate}
Suppose that $\mathcal{K}$ is an isometry and $\lambda\in\mathbb{C}$ with $|\lambda|>1$. Let $\psi_1,\psi_2,\ldots$ be a dictionary of observables and $V_{N_K} = {\rm span}\left\{\psi_1,\ldots,\psi_{N_K}\right\}$, so that $\cup_{N_K\in\mathbb{N}}V_{N_K}$ is dense in $L^2(\Omega,\omega)$. Then, for any sequence of observables $g_{N_K}\in{V}_{N_K}$ such that $\lim_{N_K\rightarrow\infty}g_{N_K}=g \in L^2(\Omega,\omega)$,
$$
\lim_{N_K\rightarrow\infty}\left(\mathcal{P}_{V_{N_K}}{\mathcal{K}}\mathcal{P}_{V_{N_K}}^*-\lambda I_{N_K}\right)^{-1}g_{N_K}= (\mathcal{K}-\lambda)^{-1}g,
$$
where $\mathcal{P}_{V_{N_K}}$ is the orthogonal projection operator onto $V_{N_K}$ and $I_{N_K}$ is the $N_K\times N_K$ identity matrix. 
\end{theorem}
\begin{proof}
Since $|\lambda|>\|{\mathcal{K}}\|$ and $\|\mathcal{P}_{{V}_{N_K}}{\mathcal{K}}\mathcal{P}_{{V}_{N_K}}^*\|\leq \|{\mathcal{K}}\|$, we have the following absolutely convergent Neumann series:
$$
\left(\mathcal{P}_{V_{N_K}}{\mathcal{K}}\mathcal{P}_{V_{N_K}}^*-\lambda I_{N_K}\right)^{-1}=-\frac{1}{\lambda}\sum_{k=0}^\infty \frac{1}{\lambda^k}(\mathcal{P}_{V_{N_K}}{\mathcal{K}}\mathcal{P}_{V_{N_K}}^*)^k,\quad (\mathcal{K}-\lambda)^{-1}=-\frac{1}{\lambda}\sum_{k=0}^\infty \frac{1}{\lambda^k}{\mathcal{K}}^k.
$$
Using $|\lambda|>\|{\mathcal{K}}\|$, we can bound the tails of these series uniformly in $N_K$ so we just need to show that $(\mathcal{P}_{V_{N_K}}{\mathcal{K}}\mathcal{P}_{V_{N_K}}^*)^k$ converges strongly to ${\mathcal{K}}^k$ for $k\geq 0$. This clearly holds for $k=0$ and $k=1$. For $k>1$, we apply induction using the fact that if a sequence of matrices $T_{N_K}$ and $S_{N_K}$ converge strongly to $\mathcal{T}$ and $\mathcal{S}$, respectively, then the product $T_{N_K}S_{N_K}$ converges strongly to the composition $\mathcal{T}\circ \mathcal{S}$.
\end{proof}
We now apply~\cref{res_square_truncate} to evaluate $[K_\epsilon*\nu_g]$ at $\theta_0$. Recall from~\eqref{unit_res_comp} that there are two types of inner products to compute: (i) $\langle g,(\mathcal{K}-\lambda)^{-1}g \rangle$ and (ii) $\langle (\mathcal{K}-\lambda)^{-1}g,\mathcal{K}^*g \rangle$ for some observable $g$. We form a sequence of observables $g_{N_K}\in V_{N_K}$ by setting $g_{N_K} = \mathcal{P}_{V_{N_K}}g$, which can be approximated from the snapshot data as
\begin{equation}
\label{f_fN_approx}
\tilde{g}_{N_K}=\sum_{j=1}^{N_K}\pmb{a}_j\psi_j, \qquad \pmb{a}=(\Psi_X^*W\Psi_X)^{-1}\Psi_X^*W\begin{bmatrix}g(\pmb{x}^{(1)})\\\vdots\\g(\pmb{x}^{(M)})\end{bmatrix}\in\mathbb{C}^{N_K},
\end{equation}
where $\Psi_X$ and $W$ are given in~\cref{sec:basic_EDMD}.  Under suitable conditions, such as those discussed in~\cref{sec:matrix_conv_galerkin}, $\lim_{M\rightarrow\infty}\tilde{g}_{N_K}=g_{N_K}$. Since
$$
\left(\mathcal{P}_{V_{N_K}}{\mathcal{K}}\mathcal{P}_{V_{N_K}}^*-\lambda I_{N_K}\right)^{-1}{g}_{N_K}=\lim_{M\rightarrow\infty}\sum_{j=1}^{N_K}\left[(\Psi_X^*W\Psi_Y-\lambda \Psi_X^*W\Psi_X)^{-1}\Psi_X^*W\Psi_X \pmb{a}\right]_j\psi_j,
$$
it follows that our two types of inner products satisfy
\begin{align}
&\left\langle \!{g}_{N_K},\left(\mathcal{P}_{V_{N_K}}{\mathcal{K}}\mathcal{P}_{V_{N_K}}^*-\lambda I_{N_K}\right)^{-1}\!\!\!{g}_{N_K}\!\right\rangle=\lim_{M\rightarrow\infty}\overline{\pmb{a}^*\Psi_X^*W\Psi_X(\Psi_X^*W\Psi_Y-\lambda \Psi_X^*W\Psi_X)^{-1}\Psi_X^*W\Psi_X\pmb{a}},\label{eq:ra_kern_inner_prod_type1}\\
&\left\langle\! \left(\mathcal{P}_{V_{N_K}}{\mathcal{K}}\mathcal{P}_{V_{N_K}}^*-\lambda I_{N_K}\right)^{-1}\!\!\!{g}_{N_K},\mathcal{K}^* {g}_{N_K}\!\right\rangle=\lim_{M\rightarrow\infty}\pmb{a}^*\Psi_X^*W\Psi_Y(\Psi_X^*W\Psi_Y-\lambda\Psi_X^*W\Psi_X)^{-1}\Psi_X^*W\Psi_X\pmb{a}.\label{eq:ra_kern_inner_prod_type2}
\end{align}
For a given value of $M$, the right-hand side of \eqref{eq:ra_kern_inner_prod_type1} and \eqref{eq:ra_kern_inner_prod_type2} can then be substituted into~\eqref{unit_res_comp} to evaluate $[K_\epsilon*\nu_g](\theta_0)$. Often we can estimate the error between these computed inner products and the limiting value as $M\rightarrow\infty$ by comparing the computations for different $M$ or by using a priori knowledge of the convergence rates. In \cref{ap_sec:inner_prod_err_control}, we show how $N_K$ can be adaptively chosen (by approximating the error in the large data limit and adding observables to the dictionary if required) so that we approximate the inner products $\langle g,(\mathcal{K}-\lambda)^{-1}g \rangle$ and $\langle (\mathcal{K}-\lambda)^{-1}g,\mathcal{K}^*g \rangle$ to a desired accuracy. Thus, for a given smoothing parameter $\epsilon$, we have a principled way of selecting (a) the sample size $M$ and (b) the truncation size $N_K$ to ensure that our approximations of the inner products in~\eqref{unit_res_comp} are accurate. In general, the cost of point evaluation of $[K_\epsilon*\nu_g]$ using these formulas is $\mathcal{O}(N_K^3)$ operations as it requires $m$ solutions of $N_K\times N_K$ dense linear systems.

\subsubsection{Evaluating a smoothed spectral measure at multiple points}
To evaluate $[K_\epsilon*\nu_g]$ at $\theta_1,\ldots,\theta_P\in[-\pi,\pi]_{\rm per}$, one can be more computationally efficient than independently computing each of the inner products in \eqref{eq:ra_kern_inner_prod_type1} and \eqref{eq:ra_kern_inner_prod_type2} for each $\theta_k$ for $1\leq k\leq P$. 
Instead, one can compute a generalized Schur decomposition and to speed up the evaluation. Let $\Psi_X^*W\Psi_Y=Q S Z^*$ and $\Psi_X^*W\Psi_X=Q T Z^*$ be a generalized Schur decomposition, where $Q$ and $Z$ are unitary matrices, and $S$ and $T$ are upper-triangular matrices. With this decomposition in hand,
\begin{align*}
{\pmb{a}^*\Psi_X^*W\Psi_X(\Psi_X^*W\Psi_Y-\lambda \Psi_X^*W\Psi_X)^{-1}\Psi_X^*W\Psi_X\pmb{a}}&={\pmb{a}^*QT(S-\lambda T)^{-1}TZ^*\pmb{a}},\\
\pmb{a}^*\Psi_X^*W\Psi_Y(\Psi_X^*W\Psi_Y-\lambda\Psi_X^*W\Psi_X)^{-1}\Psi_X^*W\Psi_X\pmb{a}&=\pmb{a}^*QS(S-\lambda T)^{-1}TZ^*\pmb{a}.
\end{align*}
Now, after computing the generalized Schur decomposition costing $\mathcal{O}(N_K^3)$ operations, each evaluation requires solving $N_K\times N_K$ upper-triangular linear systems in $\mathcal{O}(N_K^2)$ operations. Additional computational savings can be realized if one is willing to do each evaluation at $\theta_1,\ldots,\theta_P$ in parallel. We summarize the evaluation scheme in~\cref{alg:spec_meas_rat}.

\begin{algorithm}[t]
\textbf{Input:} Snapshot data $\{\pmb{x}^{(j)}\}_{j=1}^{M},\{\pmb{y}^{(j)}\}_{j=1}^{M}$ (such that $\pmb{y}^{(j)}=F(\pmb{x}^{(j)})$), quadrature weights $\{w_j\}_{j=1}^{M}$, a dictionary of observables $\{\psi_j\}_{j=1}^{N_K}$, $m\in\mathbb{N}$, smoothing parameter $0<\epsilon<1$ (accuracy goal is $\epsilon^m$), distinct points $\{z_j\}_{j=1}^m\!\subset\!\mathbb{C}$ with ${\rm Re}(z_j)>0$ (recommended choice is \eqref{gen_poisson}), and evaluation points $\{\theta_k\}_{k=1}^P\subset[-\pi,\pi]_{\rm per}$.\\
\vspace{-4mm}
\begin{algorithmic}[1]
\State Solve~\eqref{eqn:vandermonde_condition900} and~\eqref{eqn:vandermonde_condition87} for $c_1(\epsilon),\dots,c_m(\epsilon)\in\mathbb{C}$ and $d_1,\dots,d_m\in\mathbb{C}$, respectively. 
\State Compute $\Psi_X^*W\Psi_X$ and $\Psi_X^*W\Psi_Y$, where $\Psi_X$ and $\Psi_Y$ are given in~\eqref{psidef}.
\State Compute a generalized Schur decomposition of $\Psi_X^*W\Psi_Y$ and $\Psi_X^*W\Psi_X$, i.e., $\Psi_X^*W\Psi_Y =Q S Z^*$ and $\Psi_X^*W\Psi_X =Q T Z^*$, where $Q,Z$ are unitary and $S,T$ are upper triangular.
\State Compute $\pmb{a}$ in~\eqref{f_fN_approx} and $v_1=TZ^*\pmb{a}$, $v_2= T^*Q^*\pmb{a}$, and $v_3=S^*Q^*\pmb{a}$.
\For{$k=1,\ldots,P$}
        \State Compute $I_j=(S-e^{i\theta_k}(1+\epsilon z_j) T)^{-1}v_1$ for $1\leq j\leq m$.
	\State Compute $\nu_g^\epsilon(\theta_k)=\frac{-1}{2\pi}\sum_{j=1}^{m}\mathrm{Re}\!\left[c_j(\epsilon)e^{-i\theta_k}(1+\epsilon\overline{z_j})(I_j^*v_2)+d_j (v_3^*I_j)\right].$
\EndFor
\end{algorithmic} \textbf{Output:} Values of the approximate spectral measure, i.e., $\{\nu_g^\epsilon(\theta_k)\}_{k=1}^P$.
\caption{A computational framework for evaluating an approximate spectral measure with respect to $g\in L^2(\Omega,\omega)$ at $\{\theta_k\}_{k=1}^P\subset[-\pi,\pi]_{\rm per}$ of an isometry $\mathcal{K}$ using snapshot data.}\label{alg:spec_meas_rat}
\end{algorithm}

\subsection{Numerical example: The double pendulum}\label{sec:double_pend_exam}

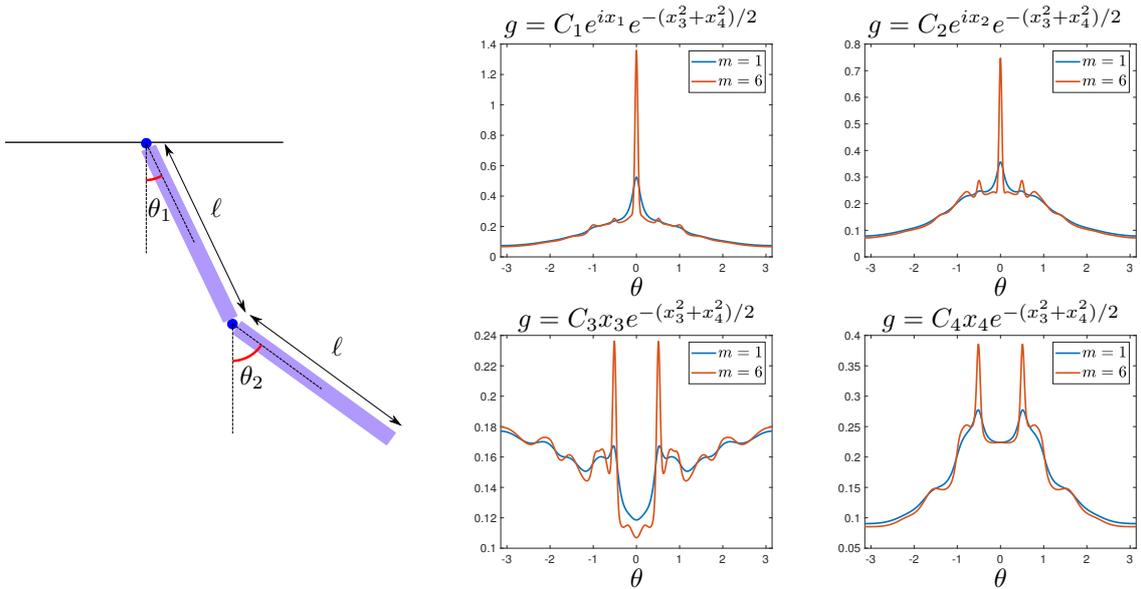
\begin{figure}[!tbp]
  \centering
	\begin{minipage}[b]{0.4\textwidth}
    \begin{overpic}[width=\textwidth,trim={0mm -40mm 0mm 0mm},clip]{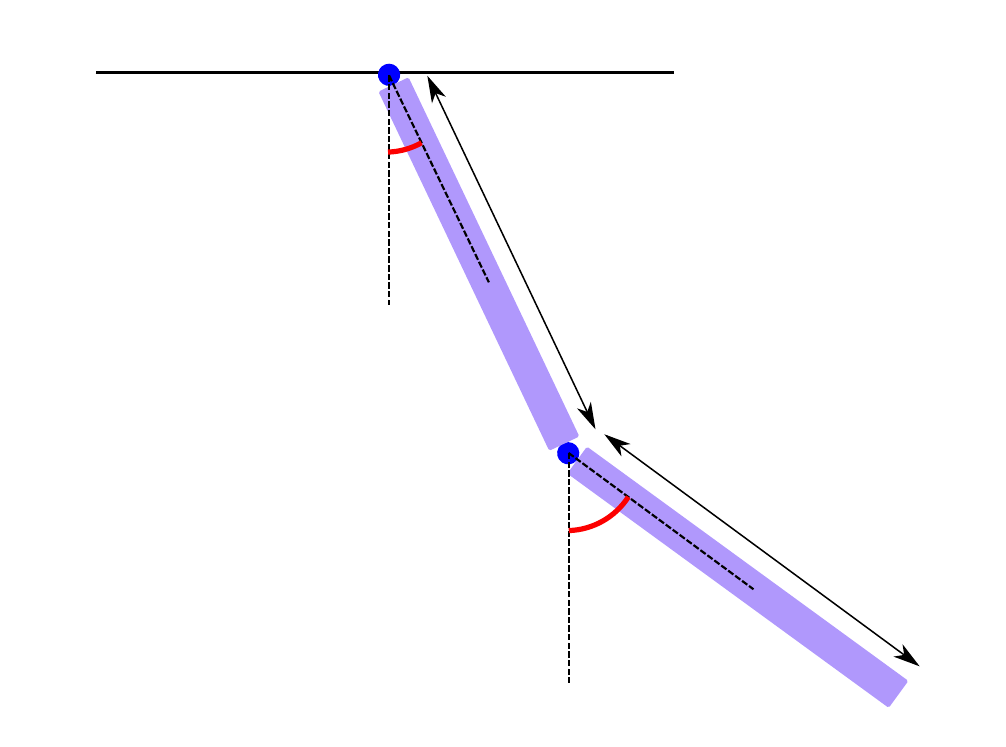}
		\put (39,75) {$\theta_1$}
		\put (58,40) {$\theta_2$}
     \put (52,75) {$\ell$}
		\put (77,46) {$\ell$}
     \end{overpic}
  \end{minipage}
	\hfill
	\begin{minipage}[b]{0.59\textwidth}
  \begin{minipage}[b]{0.49\textwidth}
    \begin{overpic}[width=\textwidth,trim={0mm 0mm 0mm 0mm},clip]{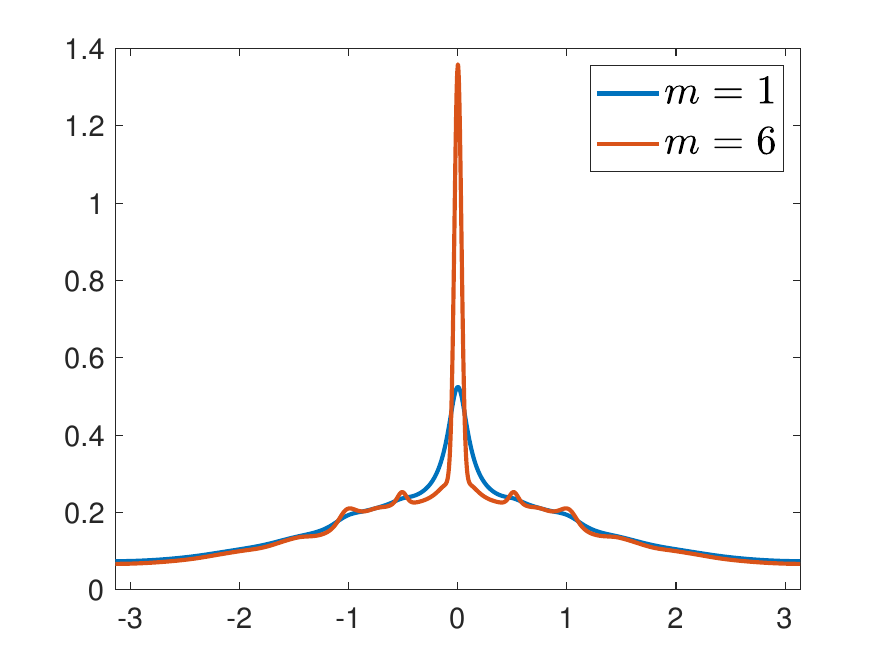}
		\put (14,73) {$g=C_1 e^{ix_1} e^{-(x_3^2+x_4^2)/{2}}$}
     \put (50,-3) {$\theta$}
     \end{overpic}
  \end{minipage}\hfill
	\begin{minipage}[b]{0.49\textwidth}
    \begin{overpic}[width=\textwidth,trim={0mm 0mm 0mm 0mm},clip]{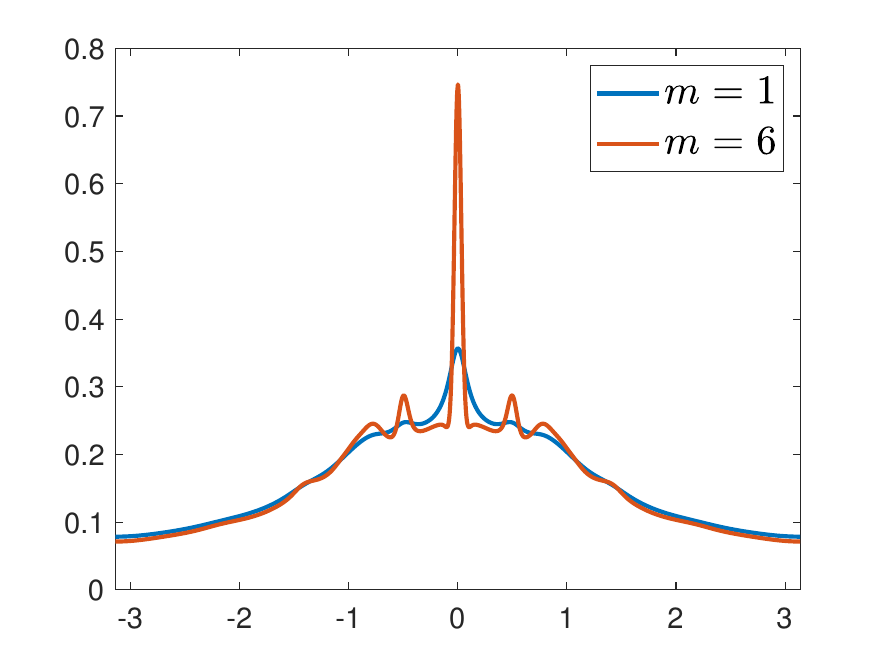}
		\put (14,73) {$g=C_2 e^{ix_2} e^{-(x_3^2+x_4^2)/{2}}$}
     \put (50,-3) {$\theta$}
     \end{overpic}
  \end{minipage}\\
	\\
	\begin{minipage}[b]{0.49\textwidth}
    \begin{overpic}[width=\textwidth,trim={0mm 0mm 0mm 0mm},clip]{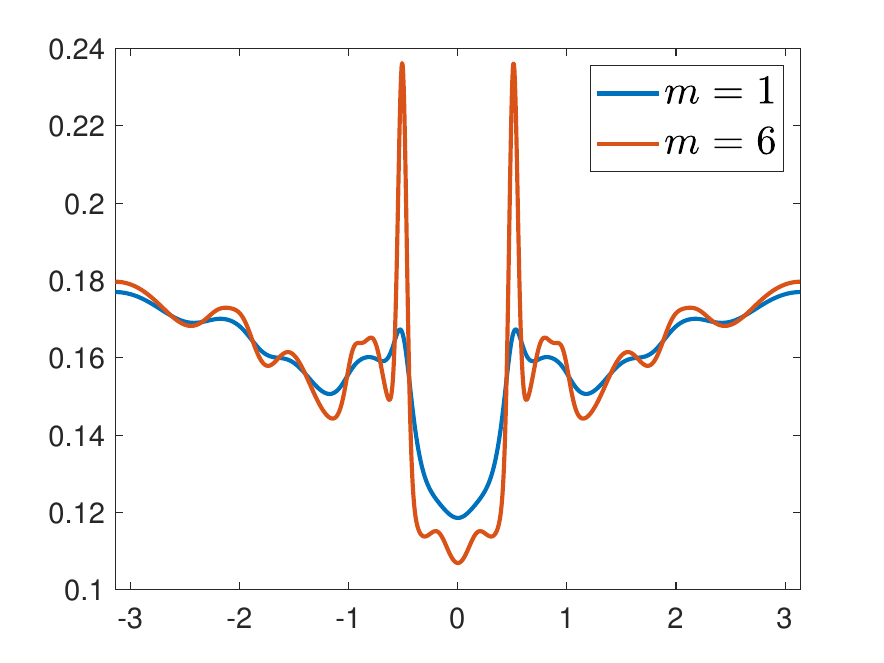}
		\put (18,72) {$g=C_3  x_3  e^{-(x_3^2+x_4^2)/{2}}$}
     \put (50,-3) {$\theta$}
     \end{overpic}
  \end{minipage}\hfill
	\begin{minipage}[b]{0.49\textwidth}
    \begin{overpic}[width=\textwidth,trim={0mm 0mm 0mm 0mm},clip]{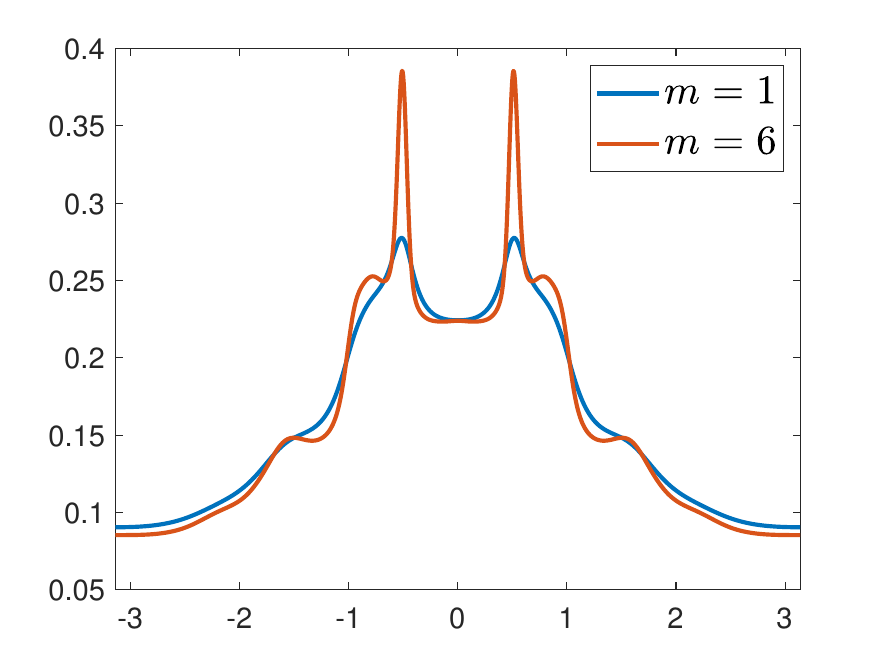}
		\put (18,72) {$g=C_4  x_4 e^{-(x_3^2+x_4^2)/{2}}$}
     \put (50,-3) {$\theta$}
     \end{overpic}
  \end{minipage}
	\end{minipage}
    \caption{Left: Physical setup of the double pendulum. Right: Computed spectral measures for four observables $g$. The constants $C_j$ are normalization factors so that $\|g\|=1$.} 
\label{fig:double_pendulum2}
\end{figure}

The double pendulum is shown in \cref{fig:double_pendulum2} (left) and is a physical system that exhibits rich dynamic behavior. It is chaotic and, after a suitable scaling of physical constants, governed by the following system of ordinary differential equations~\cite{levien1993double}:
\[
\begin{aligned}
&\dot{\theta_1}=\frac{2p_1-3p_2\cos(\theta_1-\theta_2)}{16-9\cos^2(\theta_1-\theta_2)},  \qquad \qquad \quad \dot{\theta_2}=\frac{8p_2-3p_1\cos(\theta_1-\theta_2)}{16-9\cos^2(\theta_1-\theta_2)},\\
&\dot{p_1}=-3\left(\dot{\theta_1}\dot{\theta_2}\sin(\theta_1-\theta_2)+\sin(\theta_1)\right), \quad 
\dot{p_2}=-3\left(-\dot{\theta_1}\dot{\theta_2}\sin(\theta_1-\theta_2)+\frac{1}{3}\sin(\theta_2)\right),
\end{aligned}
\]
where $p_1=8\dot{\theta_1}+3\dot{\theta_2}\cos(\theta_1-\theta_2)$ and $p_2=2\dot{\theta_2}+3\dot{\theta_1}\cos(\theta_1-\theta_2)$. The state is $\pmb{x}=(\theta_1,\theta_2,p_1,p_2)$ in $\Omega=[-\pi,\pi]_{\mathrm{per}}^2\times \mathbb{R}^2$. We consider the Koopman operator corresponding to discrete time-steps with $\Delta_t=1$. The system is Hamiltonian; hence, $\mathcal{K}$ is unitary on $L^2(\Omega,\omega)$ with the usual Lebesgue measure.

To compute the resolvent, we use \cref{alg:spec_meas_rat} with a dictionary consisting of tensor-products of the Fourier basis in $x_1$ and $x_2$, and Hermite functions in $x_3$ and $x_4$, together with an ordering corresponding to a hyperbolic cross approximation. We simulate the dynamics using the \texttt{ode45} command in MATLAB for a single time-step on an equispaced grid and use the trapezoidal rule. The chaotic nature of the dynamical system is not a problem as we only collect snapshot data from trajectories with one time step. Following the procedure outlined in \cref{sec:comput_res}, the maximum value of $M$ used in this example is 1,562,500, and the maximum value of $N_K$ used is $2528$. 

\cref{fig:double_pendulum2} (right) shows spectral measures for four observables computed using the $6$th order kernel in \cref{rat_unit_kern_table} with $\epsilon=0.1$. The results are verified against the $6$th order kernel with smaller $\epsilon$ and a larger $N_K$. We also show the results for the Poisson kernel ($m=1$), which produces an over-smoothed estimate of $\nu_g$. The Fourier observables (top row) display a pronounced peak at $\theta=0$ and relatively flat spectral measures away from zero. In contrast, the Hermite observables (bottom row) display no peak at $\theta=0$ and a broader spectral measure away from zero. These observables have a richer interplay with the dynamics, and their spectral measures are more spread out over the spectrum of the Koopman operator.

\section{High-dimensional dynamical systems}\label{sec:LARGEDIM}

In the numerical examples of \cref{sec:RES_DMD,res_kern_fin}, the dimension of the state-space of~\eqref{eq:DynamicalSystem} is modest. However, it can be very difficult to explicitly store a dictionary for larger dimensions that arise in applications such as fluid dynamics, molecular dynamics, and climate forecasting. DMD is a popular approach to studying Koopman operators associated with high-dimensional dynamics. It can yield accurate results for periodic or quasiperiodic systems but can often not adequately capture relevant nonlinear dynamics~\cite{williams2015data,baddoo2021kernel,brunton2016koopman} as it implicitly seeks to fit linear dynamics. Moreover, a rigorous connection with Galerkin approximations of Koopman operators only sometimes holds~\cite {tu2014dynamic}.

When $d$ is large, we are naturally in the setting of $N_K\gg M$ as a sufficiently rich dictionary must be selected. In other words, the dimension of the dictionary subspace is usually much larger than the snapshot data when $d$ is large. How can we approximate spectral properties of $\mathcal{K}$ when $d$ is large and $N_K\gg M$? This section extends our algorithms to high dimensions and provides applications to molecular dynamics and the study of noise leakage of turbulent flow past a cascade of airfoils.

\subsection{Adapting our algorithms to use data-driven dictionaries}
Our main idea to tackle large $d$ is to construct a data-driven dictionary using the kernel trick. We use the error control aspect of ResDMD to verify the learned dictionary aposteri.

\subsubsection{Kernelized EDMD}
A naive construction of the matrix $\Psi_X^*W\Psi_Y$ in \cref{sec:basic_EDMD} requires $\mathcal{O}(N_K^2M)$ operations, which becomes impractical when $N_K$ is large. Kernelized EDMD~\cite{williams2015kernel} aims to make EDMD practical for large $N_K$. The idea is to compute a much smaller matrix $\widetilde{K}_{\mathrm{EDMD}}$ that has a subset of the same eigenvalues as $K_{\mathrm{EDMD}}$.

\begin{proposition}[Proposition 1 of~\cite{williams2015kernel} with additional weight matrix]
\label{prop_kern_EDMD}
Let $\sqrt{W}\Psi_X = U\Sigma V^*$ be a SVD, where $U\in\mathbb{C}^{M\times M}$ is a unitary matrix, $\Sigma\in\mathbb{C}^{M\times M}$ is a diagonal matrix with nonincreasing nonnegative entries, and $V\in \mathbb{C}^{N_K\times M}$ is an isometry. Define the $M\times M$ matrix
$$
\widetilde{K}_{\mathrm{EDMD}}=(\Sigma^\dagger U^*)(\sqrt{W}\Psi_Y\Psi_X^*\sqrt{W})(U\Sigma^\dagger),
$$
where `$\dagger$' is the pseudoinverse.  Then, for some $\lambda\neq 0$ and $\tilde v\in \mathbb{C}^{M}$, $(\lambda,\tilde{v})$ is an eigenvalue-eigenvector pair of $\widetilde{K}_{\mathrm{EDMD}}$ if and only if $(\lambda,V\tilde{v})$ is an eigenvalue-eigenvector pair of $K_{\mathrm{EDMD}}$.
\end{proposition}

Suitable matrices $U$ and $\Sigma$ in \cref{prop_kern_EDMD} can be recovered from the eigenvalue decomposition $\sqrt{W}\Psi_X\Psi_X^*\sqrt{W}=U\Sigma^2U^*$. 
Moreover, both matrices $\sqrt{W}\Psi_X\Psi_X^*\sqrt{W}\in\mathbb{C}^{M\times M}$ and $\sqrt{W}\Psi_Y\Psi_X^*\sqrt{W}\in\mathbb{C}^{M\times M}$ can be computed using inner products:
	\begin{equation}
	\label{kern_trick1}
	[\sqrt{W}\Psi_X\Psi_X^*\sqrt{W}]_{jk}=\sqrt{w_j}\Psi(\pmb{x}^{(j)})\Psi(\pmb{x}^{(k)})^*\sqrt{w_k},\quad [\sqrt{W}\Psi_Y\Psi_X^*\sqrt{W}]_{jk}=\sqrt{w_j}\Psi(\pmb{y}^{(j)})\Psi(\pmb{x}^{(k)})^*\sqrt{w_k},
	\end{equation}
where we recall that $\Psi(\pmb{x})$ is a row vector of the dictionary evaluated at $\pmb{x}$.\footnote{This is the transpose of the convention in~\cite{williams2015kernel}.} Kernelized EDMD applies the kernel trick to compute the inner products in \eqref{kern_trick1} in an implicitly defined reproducing Hilbert space $\mathcal{H}$ with inner product $\langle\cdot,\cdot\rangle_{\mathcal{H}}$~\cite{scholkopf2001kernel}. A positive-definite kernel function $\mathcal{S}:\Omega\times \Omega\rightarrow\mathbb{R}$ induces a feature map $\varphi:\mathbb{R}^d\rightarrow\mathcal{H}$ so that $\langle\varphi(\pmb{x}),\varphi(\pmb{x}')\rangle_{\mathcal{H}}=\mathcal{S}(\pmb{x},\pmb{x}')$. This leads to a choice of dictionary (or reweighted feature map) $\Psi(\pmb{x})$ so that $\Psi(\pmb{x})\Psi(\pmb{x}')^*=\langle\varphi(\pmb{x}),\varphi(\pmb{x}')\rangle_{\mathcal{H}}=\mathcal{S}(\pmb{x},\pmb{x}')$. Often $\mathcal{S}$ can be evaluated in $\mathcal{O}(d)$ operations so that the matrices $\sqrt{W}\Psi_X\Psi_X^*\sqrt{W}$ and $\sqrt{W}\Psi_Y\Psi_X^*\sqrt{W}$ can be computed in $\mathcal{O}(dM^2)$ operations. $\widetilde{K}_{\mathrm{EDMD}}$ can thus be constructed in $\mathcal{O}(dM^2)$ operations, a considerable saving, with a significant reduction in memory consumption.

As a naive first attempt at extending our algorithms this way, one could consider \cref{alg:mod_EDMD} with kernelized EDMD. The residual for a vector $\tilde{v}\in\mathbb{C}^{M}$ becomes
\begin{equation}
\label{biagbfvuibi}
\frac{\tilde{v}^*[(\Psi_YV)^*W(\Psi_YV)-\lambda (\Psi_YV)^*W(\Psi_XV)-\overline{\lambda}(\Psi_XV)^*W(\Psi_YV)+|\lambda|^2(\Psi_XV)^*W(\Psi_XV)]\tilde{v}}{\tilde{v}^*(\Psi_XV)^*W(\Psi_XV)\tilde{v}^*}.
\end{equation}
Unfortunately, the following simple proposition shows that the residual in \eqref{biagbfvuibi} always vanishes. This should be interpreted as \textit{over-fitting of the snapshot data} when $N_K\geq M$.

\begin{proposition}
Suppose that $N_K\geq M$ and $\sqrt{W}\Psi_X$ has rank $M$ (independent rows). For any eigenvalue-eigenvector pair $\lambda$ and $\tilde v\in \mathbb{C}^{M}$ of $\widetilde{K}_{\mathrm{EDMD}}$, the residual in \eqref{biagbfvuibi} vanishes.
\end{proposition}

\begin{proof} The numerator in \eqref{biagbfvuibi} is the square of the $\ell^2$-vector norm of $\sqrt{W}(\Psi_Y-\lambda \Psi_X)V\tilde{v}$.
From the decomposition $\sqrt{W}\Psi_X=U\Sigma V^*$, we see that
$$
\sqrt{W}(\Psi_Y-\lambda \Psi_X)V\tilde{v}=\sqrt{W}\Psi_YV\tilde{v}-\lambda U \Sigma\tilde{v}=\sqrt{W}\Psi_Y\Psi_X^*\sqrt{W}U\Sigma^\dagger \tilde{v}-\lambda U \Sigma\tilde{v}.
$$
Since $\Sigma$ is invertible, $\sqrt{W}\Psi_Y\Psi_X^*\sqrt{W}U\Sigma^\dagger=U\Sigma \widetilde{K}_{\mathrm{EDMD}}$. Hence, $\sqrt{W}(\Psi_Y-\lambda \Psi_X)V\tilde{v}=U\Sigma(\widetilde{K}_{\mathrm{EDMD}}-\lambda)\tilde{v}=0$.
\end{proof}

In other words, the restriction $N_K\geq M$ prevents the large data convergence $(M\rightarrow\infty)$ of \cref{alg:mod_EDMD} for fixed $N_K$. Moreover, the residual in \eqref{biagbfvuibi} may not fit into the $L^2(\Omega,\omega)$ Galerkin framework because of the implicit reproducing Hilbert space. A way to overcome these issues is to consider two sets of snapshot data (see~\cref{sec:kernelizedAlgorithms}).

\subsubsection{Kernelized versions of the algorithms}\label{sec:kernelizedAlgorithms}
Kernelized EDMD computes the eigenvalues, eigenvectors, and corresponding Koopman modes of $\widetilde{K}_{\mathrm{EDMD}}$. The required operations to compute these quantities are inner products and hence can be computed by evaluating $\mathcal{S}$.  In particular, the $k$th eigenfunction $\psi_k$ corresponding to eigenvector $\tilde v_k\in\mathbb{C}^{M}$ of $\widetilde{K}_{\mathrm{EDMD}}$ can be evaluated as~\cite[Remark 8]{williams2015kernel}
\begin{equation}
\label{hiohvikn}
\psi_k(\pmb{x})=\begin{bmatrix}
\mathcal{S}(\pmb{x},\pmb{x}^{(1)})&\mathcal{S}(\pmb{x},\pmb{x}^{(2)})&\cdots & \mathcal{S}(\pmb{x},\pmb{x}^{(M)})
\end{bmatrix}(U\Sigma^\dagger)\tilde v_k,\qquad k=1,\ldots, M.
\end{equation}
This leads to the simple kernelization of our algorithms summarized in~\cref{alg:kern_algs}. The idea is a two-step process: (1) We use kernelized EDMD with a data set $\{\pmb{x}^{(j)}, \pmb{y}^{(j)}\}_{j=1}^{M'}$ to compute a dictionary corresponding to the dominant eigenfunctions of $\widetilde{K}_{\mathrm{EDMD}}$ and (2) We apply our algorithms using this dictionary on another set $\{\hat{\pmb{x}}^{(j)}, \hat{\pmb{y}}^{(j)}\}_{j=1}^{M''}$ of snapshot data. Of course, there are many algorithmic details regarding the choice of kernel, the second set of snapshot data, and how to a posteriori check that the learned dictionary is suitable. We describe these details now: 

\begin{itemize}[leftmargin=*]
	
	\item \textbf{Choice of kernel:} The choice of kernel $\mathcal{S}$ determines the dictionary, and the best choice depends on the application. In the following experiments, we use the Gaussian radial basis function kernel 
\begin{equation}
\label{exp_kern_formula}
\mathcal{S}(\pmb{x},\pmb{y})=\exp\left(-\gamma{\|\pmb{x}-\pmb{y}\|^2}\right).
\end{equation}
We select $\gamma$ as the squared reciprocal of the average $\ell^2$-norm of the snapshot data after it is shifted to mean zero.

\item \textbf{Regularization:} The matrix $\sqrt{W}\Psi_X\Psi_X^*\sqrt{W}$ can be ill-conditioned~\cite{belkin2018approximation}, in which case we can regularize and consider $\sqrt{W}\Psi_X\Psi_X^*\sqrt{W}+\eta \|\sqrt{W}\Psi_X\Psi_X^*\sqrt{W}\|$ for small $\eta$ or consider a more stable representation of the range of the feature map. This is standard practice in DMD. However, we found that it was often not needed for~\cref{alg:kern_algs}, due to steps 3 and 4.

	\item \textbf{Acquiring a second set of snapshot data:}  Given snapshot data with random and independent initial conditions, one can split the snapshot data up into two parts $\{\pmb{x}^{(j)}, \pmb{y}^{(j)}\}_{j=1}^{M'}$ and $\{\hat{\pmb{x}}^{(j)}, \hat{\pmb{y}}^{(j)}\}_{j=1}^{M''}$, where $M = M' + M''$. For initial conditions distributed according to a quadrature rule, if one already has access to $M'$ snapshots, one must typically go back to the original dynamical system and request $M''$ further snapshots. In both cases, we denote the total amount of snapshot data as $M=M' + M''$. The first set of $M'$ snapshots is used in the kernelized approach to refine a large initial dictionary of size $N_K'$ (implicitly defined by the chosen kernel) to a smaller data-driven one of size $N_K''$, while the second set of $M''$ snapshots are used to apply our ResDMD algorithms. We recommend that $M''\geq M'$.
	
	\item \textbf{Selecting the second dictionary:} We start with a large dictionary of size $N_K'$ and apply a kernelized approach to construct a second dictionary of size $N_K''$. To construct the second dictionary, we compute the dominant $N_K''$ eigenvectors of $\widetilde{K}_{\mathrm{EDMD}}$ that hopefully capture the essential dynamics. In this paper, we compute a complete eigendecomposition of $\widetilde{K}_{\mathrm{EDMD}}$; however, for problems with extremely large $M'$, iterative methods are recommended to compute these dominant eigenvectors. 
	
\item \textbf{Rigorous results:} It is well-known that the eigenvalues computed by kernel EDMD may suffer from spectral pollution. In our setting, we do not directly use kernel EDMD to compute eigenvalues of $\mathcal{K}$. Instead, we only use kernel EDMD to select a good dictionary of size $N_K''$, after which our rigorous algorithms can be used (see~\cref{alg:spec_meas_rat,alg:res_EDMD,alg:mod_EDMD}, with convergence as $M''\rightarrow\infty$).
	
	\item \textbf{A posteriori check of the second dictionary:} We use $\{\hat{\pmb{x}}^{(j)}, \hat{\pmb{y}}^{(j)}\}_{j=1}^{M''}$ to check the quality of the constructed second dictionary. By studying the residuals and using the error control in \cref{half_pseudospectrum}, we can tell a posteriori whether the second dictionary is satisfactory and $N_K''$ is sufficiently large.
	
\end{itemize}

\begin{algorithm}[t]
\textbf{Input:} Snapshot data $\{\pmb{x}^{(j)}, \pmb{y}^{(j)}\}_{j=1}^{M'}$ and $\{\hat{\pmb{x}}^{(j)}, \hat{\pmb{y}}^{(j)}\}_{j=1}^{M''}$, positive-definite kernel function $\mathcal{S}:\Omega\times \Omega\rightarrow\mathbb{R}$, and positive integer $N_K''\leq M'$.\\
\vspace{-4mm}
\begin{algorithmic}[1]
\State Apply kernel EDMD to $\{\pmb{x}^{(j)}, \pmb{y}^{(j)}\}_{j=1}^{M'}$ with kernel $\mathcal{S}$ to compute the matrices $\sqrt{W}\Psi_X\Psi_X^*\sqrt{W}$ and $\sqrt{W}\Psi_Y\Psi_X^*\sqrt{W}$ using \eqref{kern_trick1} and the kernel trick.
\State Compute $U$ and $\Sigma$ from the eigendecomposition $\sqrt{W}\Psi_X\Psi_X^*\sqrt{W}=U\Sigma^2U^*$.
\State Compute the dominant $N_K''$ eigenvectors of $\widetilde{K}_{\mathrm{EDMD}}=(\Sigma^\dagger U^*)\sqrt{W}\Psi_Y\Psi_X^*\sqrt{W}(U\Sigma^\dagger)$ and stack them column-by-column into $Z\in\mathbb{C}^{M'\times N_K''}$. 
\State Apply a QR decomposition to orthogonalize $Z$ to $Q=\begin{bmatrix}Q_1 & \cdots& Q_{N_K''} \end{bmatrix}\in\mathbb{C}^{M'\times N_K''}$.
\State Apply \cref{alg:spec_meas_rat,alg:res_EDMD,alg:mod_EDMD} with $\{\hat{\pmb{x}}^{(j)}, \hat{\pmb{y}}^{(j)}\}_{j=1}^{M''}$ and the dictionary $\{\psi_j\}_{j=1}^{N_K''}$, where
$$
\psi_j(\pmb{x})=\begin{bmatrix}\mathcal{S}(\pmb{x},\pmb{x}^{(1)})&\mathcal{S}(\pmb{x},\pmb{x}^{(2)})&\cdots & \mathcal{S}(\pmb{x},\pmb{x}^{(M')})
\end{bmatrix}(U\Sigma^\dagger)Q_j, \qquad 1\leq j\leq N_K''.
$$
\end{algorithmic} \textbf{Output:} Spectral properties of Koopman operator according to \cref{alg:spec_meas_rat,alg:res_EDMD,alg:mod_EDMD}.
\caption{A computational framework for kernelized versions of \cref{alg:spec_meas_rat,alg:res_EDMD,alg:mod_EDMD}.}\label{alg:kern_algs}
\end{algorithm}

\subsection{Example: MD simulation of the Adenylate Kinase enzyme \texorpdfstring{($d=20,\!046$)}{}}
\label{examMDMD}

Molecular dynamics (MD) analyzes the movement of atoms and molecules by numerically solving Newton's equations of motion for a system of interacting particles. Energies and forces between particles are typically computed using potentials. MD is arguably one of the most robust approaches for simulating macromolecular dynamics, primarily due to the availability of full atomistic detail~\cite{dror2012biomolecular}. Recently, DMD-type and Koopman techniques are impacting MD~\cite{nuske2014variational,klus2018data,schwantes2015modeling,schwantes2013improvements}. For example,~\cite{klus2020eigendecompositions} applies kernel EDMD to the positions of the carbon atoms in \textit{n}-butane ($d=12$) and shows that the EDMD eigenfunctions parameterize a dihedral angle that controls crucial dynamics.

Here, we study trajectory data from the dynamics of Adenylate Kinase (ADK), which is an enzyme (see~\cref{fig:ADK1}) that catalyzes important phosphate reactions in cellular biology~\cite{goldberg2004thermodynamics}. ADK is a common benchmark enzyme in MD~\cite{seyler2014sampling} and consists of $3341$ atoms split into 214 residues (specific monomers that can be thought of as parts). The trajectory data comes from an all-atom equilibrium simulation for $1.004\times 10^{-6}$s, with a so-called CHARMM force field, that is produced on PSC Anton~\cite{shaw2009millisecond} and publicly available~\cite{mdanalysis_website}. The data consists of a single trajectory of the positions of all atoms as ADK moves. To make the system Hamiltonian, we append the data with approximations of the velocities computed using centered finite differences. This leads to $d=6\times3341=20046$. We vertically stack the data (as discussed in \cref{sec:RES_DMD}) and sample the trajectory data every $240\times 10^{-12}$s so that $M=4184$. 

To apply the kernelized version of \cref{alg:spec_meas_rat}, we subselect $M'=2000$ initial conditions from the trajectory data and used the Gaussian radial basis function kernel in~\eqref{exp_kern_formula}. We select ${N}_K''=1000$ EDMD eigenfunctions and append the dictionary with the four observables of interest that are discussed below. Accuracy of the corresponding matrices in \eqref{eq:convergenceMatrices} is verified by comparing to smaller $M''$ and also by computing pseudospectra with \cref{alg:res_EDMD}.

ADK has three parts of its molecule called CORE, LID, and NMP (see~\cref{fig:ADK1} (left)). The LID and NMP domains move around the stable CORE. We select the most mobile residue from the LID and NMP domains by computing root-mean-square-fluctuations. These residues have canonical dihedral angles $(\phi,\psi)$ defined on the backbone atoms~\cite{ramachandran1963stereochemistry} that determine the overall shape of the residue. \cref{fig:ADK1} (middle, right) shows the spectral measures with respect to these dihedral angles (where we have subtracted the mean angle value) for both selected residues. These spectral measures are computed using the sixth-order rational kernel  with $\epsilon=0.1$ (see~\cref{rat_unit_kern_table}). The computed spectral measures are verified with higher order kernels and smaller $\epsilon$, and through comparison with a polynomial kernel $\mathcal{S}$ instead of a Gaussian radial basis function kernel. The spectral measures for the angles in the LID residue are much broader than for the NMP residue. This hints at a more complicated dynamical interaction and may have biological consequences. We hope that this example can be a catalyst for the use of spectral measures in MD.

\begin{figure}[!tbp]
  \centering
	\begin{minipage}[b]{0.30\textwidth}
    \begin{overpic}[width=\textwidth,trim={0mm 0mm 0mm 0mm},clip]{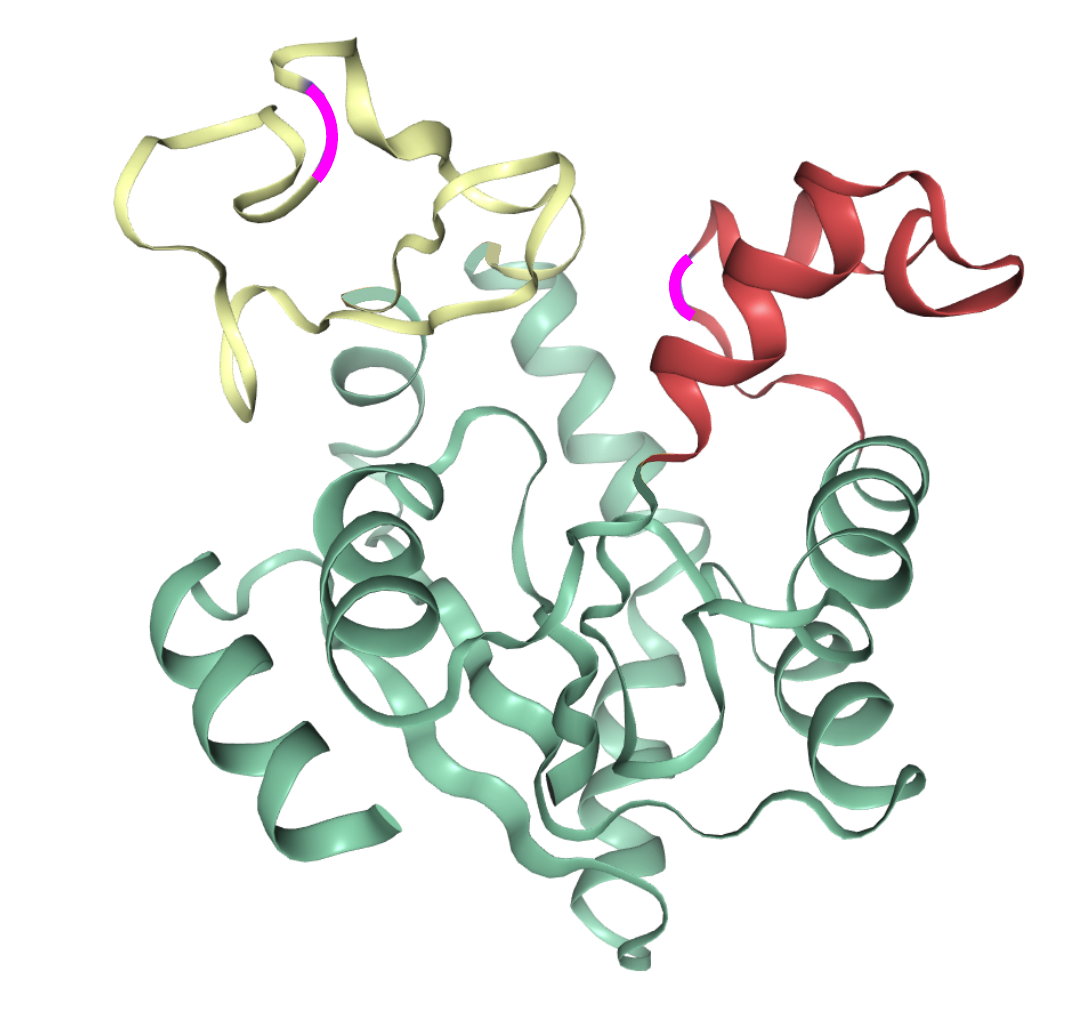}
		\put(33,81){\line(2,1){11}}
		\put(47,87){\line(1,-1){15}}
		\put (47,88) {residues}
		\put(10,60){\line(1,1){8}}
		\put(5,54){LID}
		\put(20,12){\line(0,1){6}}
		\put(15,6){CORE}
		\put(88,60){\line(-1,1){8}}
		\put(85.5,54){NMP}
     \end{overpic}
  \end{minipage}\hspace{1mm}
	\begin{minipage}[b]{0.33\textwidth}
    \begin{overpic}[width=\textwidth,trim={0mm -10mm 0mm 0mm},clip]{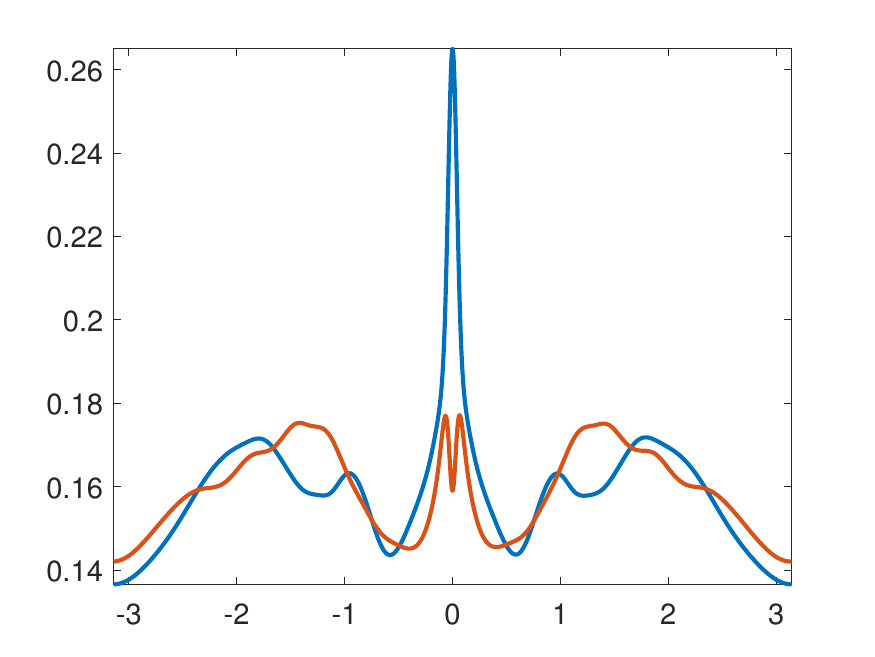}
		\put (50,0) {$\theta$}
		\put (46,80) {LID}
		\put (67,60) { {$\displaystyle \phi$}}
	\put(67,60)  {\vector(-1,-1){14}}
	\put (18,60) { {$\displaystyle \psi$}}
	\put(22,57)  {\vector(1,-2){12}}
     \end{overpic}
  \end{minipage}
	\begin{minipage}[b]{0.33\textwidth}
    \begin{overpic}[width=\textwidth,trim={0mm -10mm 0mm 0mm},clip]{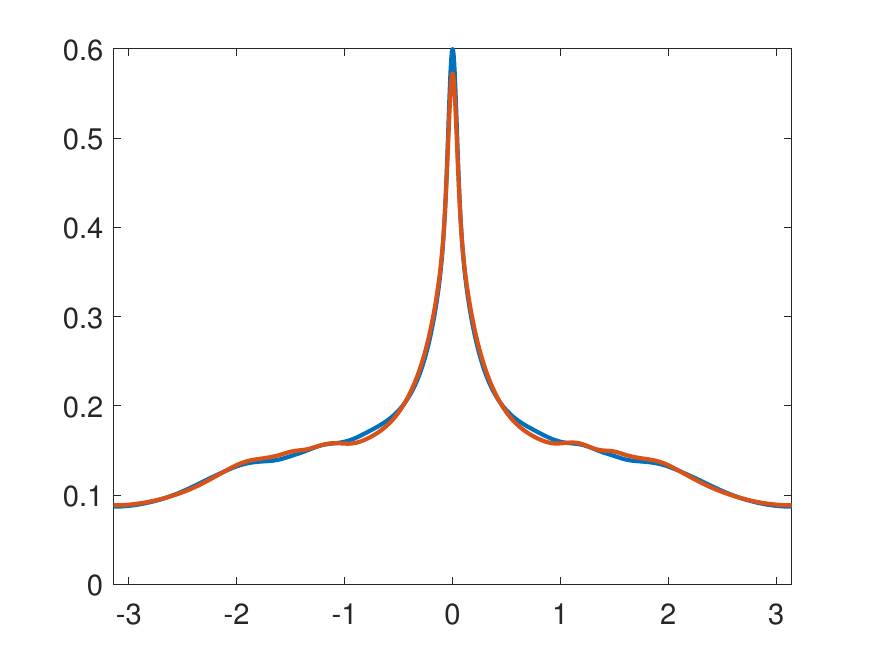}
		\put (50,0) {$\theta$}
		\put (44,80) {NMP}
		\put (67,59) { {$\displaystyle \phi$}}
	\put(67,61)  {\vector(-1,1){14}}
	\put (23,17) { {$\displaystyle \psi$}}
	\put(32,22)  {\vector(1,1){9}}
     \end{overpic}
  \end{minipage}
    \caption{Left: Structure of ADK, which has three domains: CORE (green), LID (yellow) and NMP (red). Middle and right: Spectral measures with respect to the dihedral angles of the selected residues.} 
\label{fig:ADK1}
\end{figure}
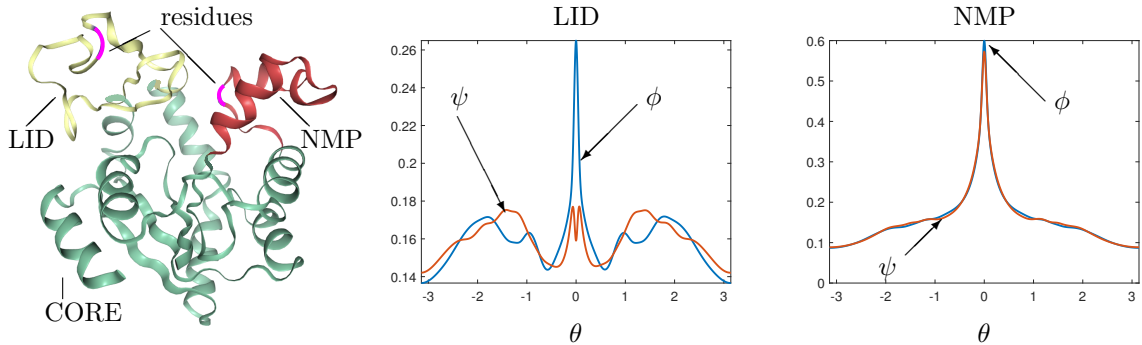

\subsection{Example: Turbulent flow past a cascade of aerofoils \texorpdfstring{($d=295,\!122$)}{}}\label{sec:flow_example}
Some of the most successful applications of DMD are in fluid dynamics~\cite{rowley2009spectral,budivsic2012applied,chen2012variants,mezic2013analysis,taira2017modal}. Here, we consider a large-scale wall-resolved turbulent flow past a periodic cascade of aerofoils with a stagger angle $56.9^{\circ}$ and a one-sided tip gap. The setup is motivated by the need to reduce noise sources from flying objects~\cite{peake2012modern}. We use a high-fidelity simulation that solves the fully nonlinear Navier--Stokes equations~\cite{koch2021large} with Reynolds number $3.88\times10^5$ and Mach number $0.07$. The data consists of a 2D slice of the pressure field, measured at $295,\!122$ points, for trajectories of length 798 that are sampled every $2\times 10^{-5}$s. We stack the data in the form of \eqref{data:snapshots} so that $M=797$.

To apply the kernelized version of \cref{alg:mod_EDMD}, we subselect $M'=350$ initial conditions from the trajectory data and used the Gaussian radial basis function kernel in~\eqref{exp_kern_formula}. We select $N_K''=250$ EDMD eigenfunctions as our dictionary. Accuracy of the corresponding matrices in \eqref{eq:convergenceMatrices} is verified by comparing to smaller $M''$ and also by computing pseudospectra with \cref{alg:res_EDMD}.

\cref{fig:ResDMD_fluids} shows the computed Koopman modes for a range of representative frequencies. We also show the corresponding Koopman modes computed using DMD. In the case of DMD and ResDMD, each mode is only defined up to a global phase. Moreover, the color maps for each $\lambda$ differ because of the pressure modulus variations with frequency. For the first row of \cref{fig:ResDMD_fluids}, ResDMD shows stronger acoustic waves between the cascades. Detecting these vibrations is important as they can damage turbines in engines~\cite{parker1984acoustic}. For the second and third row of \cref{fig:ResDMD_fluids}, ResDMD shows larger scale turbulent fluctuations past the trailing edge corresponding to the figure's right. This can be crucial for understanding acoustic interactions with nearby structures such as subsequent blade rows~\cite{woodley1999resonant}. DMD shows a clear acoustic source for the third row just above the wing. This source is less distinct in the case of ResDMD because of nonlinear interference. The residuals for ResDMD are also shown and are small, particularly given the enormous state-space dimension. This example demonstrates two benefits of the kernelized version of ResDMD (see~\cref{alg:mod_EDMD,alg:res_EDMD}) compared with DMD: (1) ResDMD can capture the nonlinear dynamics, and (2) It computes residuals as well, thus providing an accuracy certificate.

\begin{figure}[!tbp]
  \centering\setlength{\tabcolsep}{1pt}
    \begin{tabular}{@{}>{\centering}m{0.03\textwidth}|>{\centering}m{0.47\textwidth}|>{\centering\arraybackslash}m{0.47\textwidth}@{}}
		&DMD & ResDMD, $\mathrm{res}=0.0054$ \\		
		\rotatebox[origin=c]{90}{$\lambda=e^{0.11i}$} &\begin{overpic}[width=\linewidth,trim={14mm 0mm 10mm 0mm},clip]{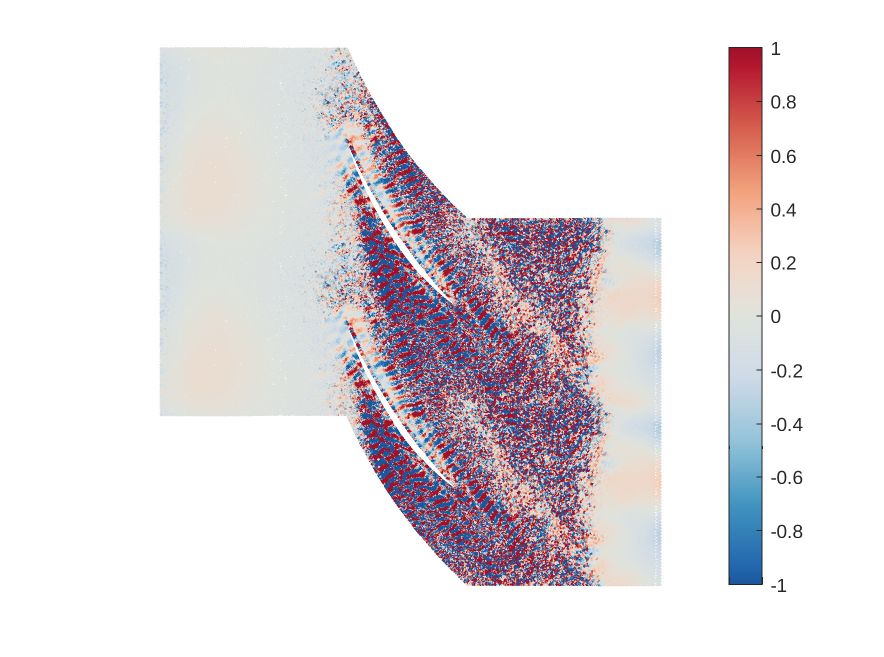}
     \end{overpic}&\begin{overpic}[width=\linewidth,trim={14mm 0mm 10mm 0mm},clip]{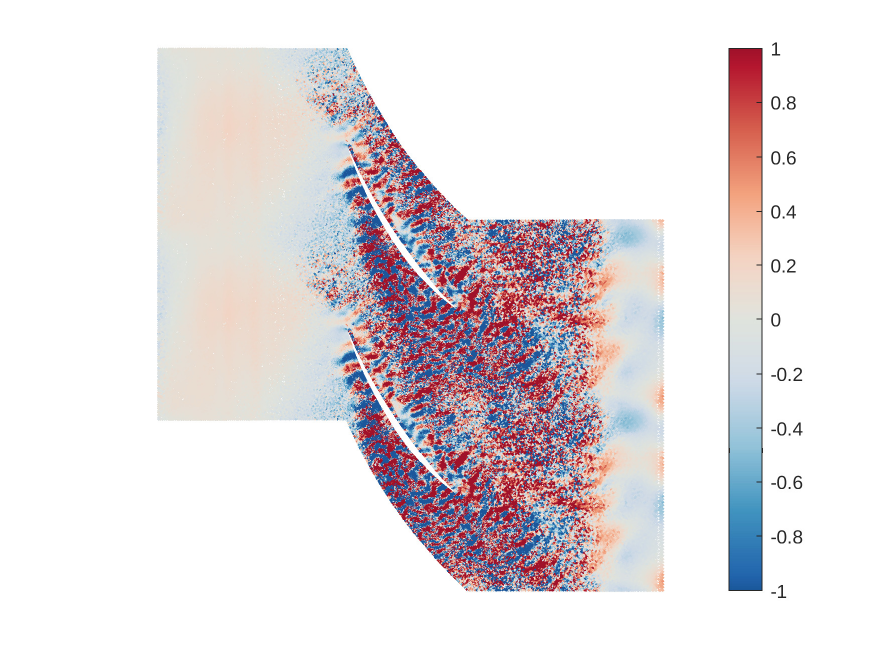}
		\put(10,5){acoustic vibrations}
		\put(30,10){\line(1,1){12}}
     \end{overpic}\\
		\cline{1-3}&&\\
		&DMD & ResDMD, $\mathrm{res}=0.0128$\\		
		\rotatebox[origin=c]{90}{$\lambda=e^{0.51i}$}&\begin{overpic}[width=\linewidth,trim={14mm 0mm 10mm 0mm},clip]{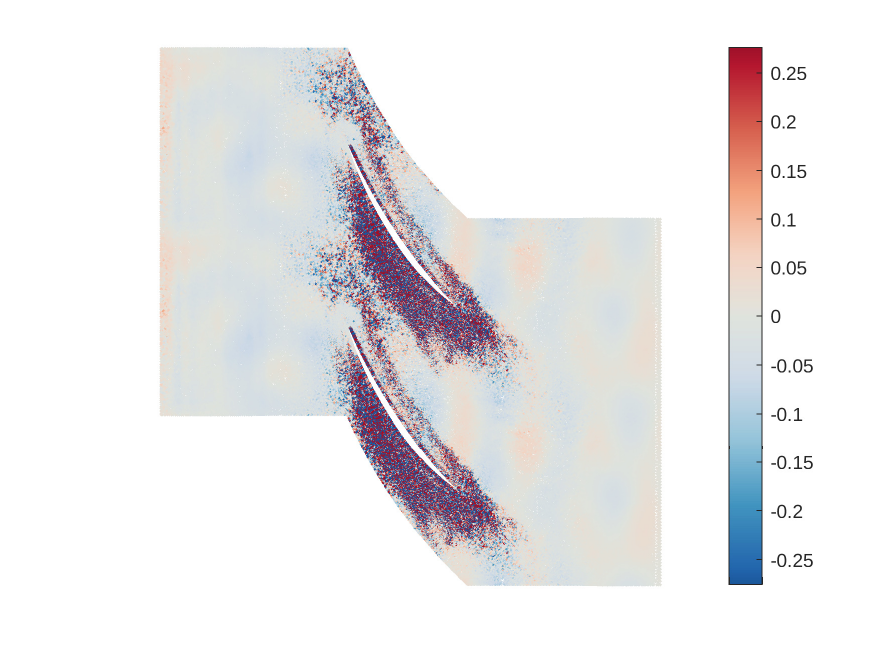}
     \end{overpic}&\begin{overpic}[width=\linewidth,trim={14mm 0mm 10mm 0mm},clip]{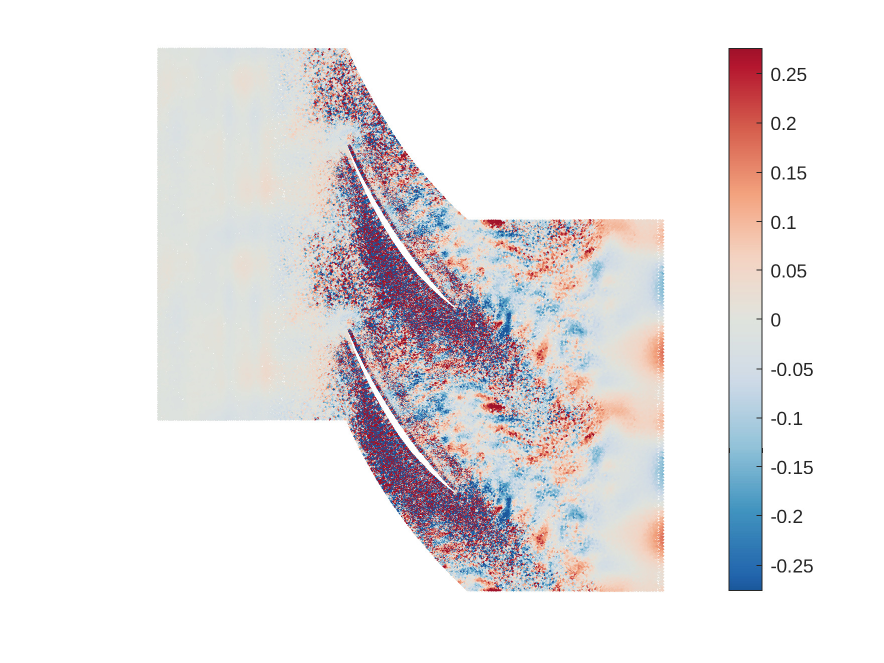}
		\put(50,68){turbulent}
		\put(50,63){fluctuations}
		\put(56,60){\line(0,-1){8}}
		     \end{overpic}\\
		\cline{1-3}&&\\
		&DMD  & ResDMD, $\mathrm{res}=0.0196$\\		
		\rotatebox[origin=c]{90}{$\lambda=e^{0.71i}$}&\begin{overpic}[width=\linewidth,trim={14mm 0mm 10mm 0mm},clip]{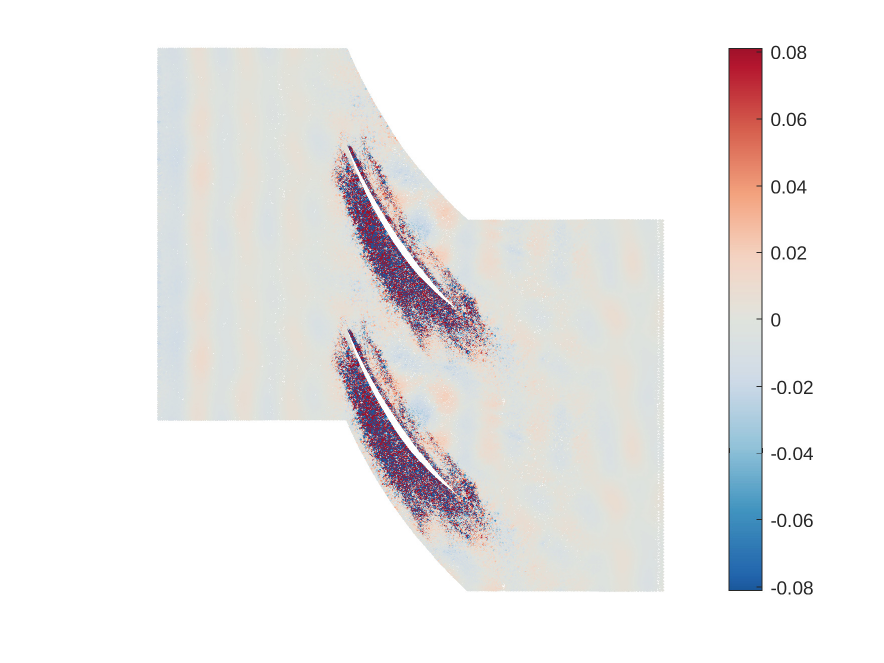}
		\put(43,68){acoustic source}
		\put(45,58){\line(1,1){8}}
     \end{overpic}&\begin{overpic}[width=\linewidth,trim={14mm 0mm 10mm 0mm},clip]{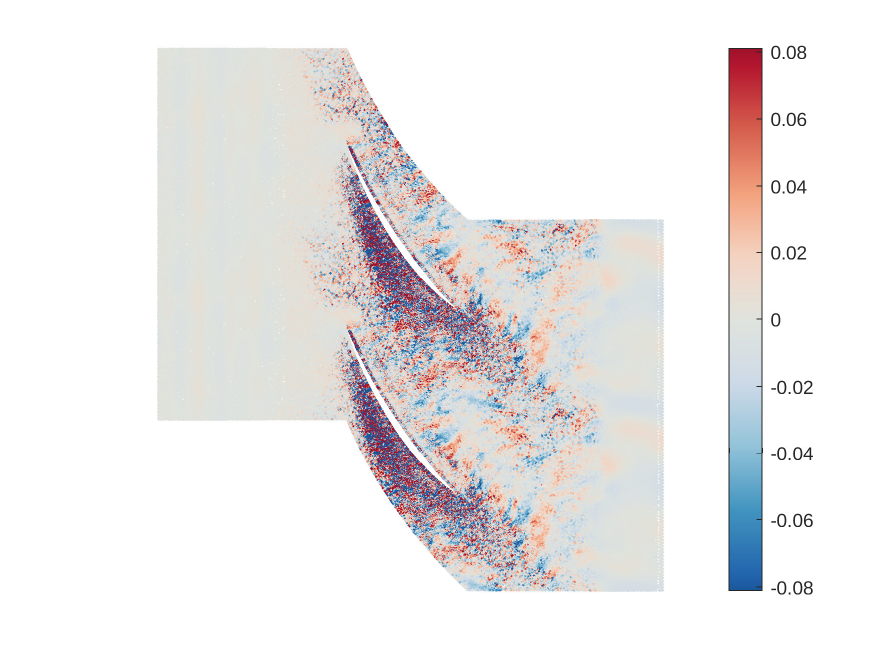}
     \end{overpic}
		\end{tabular}
    \caption{Left column: Modes computed by DMD. Right column: Modes computed by ResDMD with computed residuals. Each row corresponds to the different values of $\lambda$ that correspond to certain physical frequencies of noise pollution.} 
\label{fig:ResDMD_fluids}
\end{figure}

\section*{Acknowledgements}
The work of the first author was supported by a Research Fellowship at Trinity College, Cambridge, a Fondation Sciences Mathématiques de Paris Postdoctoral Fellowship at École Normale Supérieure, and an LMS Cecil King Travel Scholarship. The work of the second author was supported by the National Science Foundation grants DMS-1818757, DMS-1952757 and DMS-2045646. The authors are grateful to John Guckenheimer for stimulating discussions during the completion of this work, Andrew Horning and Jennifer Zvonek for reading a draft of this manuscript, Dimitris Giannakis and Igor Mezić for feedback, Regis Koch, Marlène Sanjosé and Stephane Moreau for sharing their flow data for the example in \cref{sec:flow_example}, and to the anonymous referees whose careful comments helped improve the clarity of the manuscript.

\appendix

\section{A useful criterion for \texorpdfstring{$m$}{m}th order kernels}
The following lemma is used to prove \cref{polynomial_magic} and build $m$th order rational kernels of the form \eqref{rat_kern_unitary}. It provides sufficient conditions for a family of integrable functions to be an $m$th order kernel.

\begin{lemma}\label{unitary_rational_tool}
Let $\{K_{\epsilon}:\epsilon\in(0,1]\}$ be a family of integrable functions on the periodic interval $[-\pi,\pi]_{\mathrm{per}}$ that integrate to $1$. Suppose that there exists a constant $C$ such that for any integer $n$ with $0< n\leq m-1$,
\begin{equation}\label{fourier_cond}
\left|\int_{[-\pi,\pi]_{\mathrm{per}}}K_{\epsilon}(-\theta)e^{in\theta}d\theta-1\right|\leq C \epsilon^m \log(\epsilon^{-1}),
\end{equation}
and such that
\begin{equation}
\label{bound_lemma}
\left|K_{\epsilon}(\theta)\right|\leq \frac{C\epsilon^m}{(\epsilon+|\theta|)^{m+1}},
\end{equation}
for any $\theta\in[-\pi,\pi]_{\mathrm{per}}$ and $\epsilon\in(0,1]$. Then $\{K_{\epsilon}\}$ is an $m$th order kernel for $[-\pi,\pi]_{\mathrm{per}}$.
\end{lemma}
\begin{proof}[Proof of \cref{unitary_rational_tool}]
There are three properties that an $m$th order kernel must satisfy (see~\cref{unit_mth_kern_def}).  The lemma statement already assumes properties (i) and (iii).  To show part (ii) of~\cref{unit_mth_kern_def}, let $n$ be an integer and consider $\theta^n$ with $n<m$. We use the convergent power series for the principal branch of the logarithm around $1$ to write, for $\theta\in[-\pi/4,\pi/4]$,
$$
\theta^n=\sum_{k=1}^\infty a_{k,n} (e^{i\theta}-1)^k=\underbrace{\sum_{k=1}^{m-1}a_{k,n} (e^{i\theta}-1)^k}_{=p_{n,m}(\theta)}+\theta^m\underbrace{\frac{(e^{i\theta}-1)^m}{\theta^m}\sum_{k=m}^{\infty}a_{k,n}(e^{i\theta}-1)^{k-m}}_{=h_{n,m}(\theta)},
$$
where $a_{k,n}\in\mathbb{C}$ and the domain of the trigonometric polynomial $p_{n,m}$ is extended to $[-\pi,\pi]_{\mathrm{per}}$. We note that $h_{n,m}(\theta)$ is a bounded continuous function on $[-\pi/4,\pi/4]$. Since $|\int_{[-\pi,\pi]_{\mathrm{per}}}K_{\epsilon}(-\theta)p_{n,m}(\theta)d\theta-p_{n,m}(0)|\lesssim\epsilon^m \log(\epsilon^{-1})$, we know by~\eqref{bound_lemma} that
\begin{align*}
\left|\int_{-\pi/4}^{\pi/4}\theta^nK_{\epsilon}(-\theta)\,d\theta\right|&\lesssim \int_{0}^{\pi/(4\epsilon)}\!\!\!\!\frac{\epsilon^m\tau^m}{(1+\tau)^{m+1}}\,d\tau+\int_{|\theta|>\pi/4}\!\!\!\!\!\!\!\!\!\! |K_{\epsilon}(-\theta)||p_{n,m}(\theta)|\,d\theta+\epsilon^m \log(\epsilon^{-1})\\
&\lesssim \epsilon^m \log(\epsilon^{-1})+\int_{|\theta|>\pi/4} \frac{\epsilon^m\|p_{n,m}\|_{L^\infty}}{(\epsilon+\pi/4)^{m+1}}\,d\theta\lesssim \epsilon^m \log(\epsilon^{-1}).
\end{align*} 
Furthermore, by~\eqref{bound_lemma} we obtain
$$
\left|\int_{[-\pi,-\pi/4)\cup(\pi/4,\pi]}\theta^n K_{\epsilon}(-\theta)\,d\theta\right|\leq \pi^{n}(2\pi-\pi/2)\frac{C\epsilon^m}{(\epsilon+\pi/4)^{m+1}}\lesssim \epsilon^m.
$$
Property (ii) of~\cref{unit_mth_kern_def} now follows by the symmetry of~\eqref{unit_first_lem2} under the mapping $\theta\mapsto-\theta$.
\end{proof}

\section{ResDMD for computing spectra and pseudospectra}\label{sec:computing_spectra_limits}

In this section, we show how to modify ResDMD in~\cref{sec:RESDMD_spec_pollll1} to rigorously compute the global spectral properties of the Koopman operator $\mathcal{K}$. Recall that we assume that $\mathcal{K}$ is a closed and densely defined operator. Throughout this section, we make the following assumption about the dictionary $\{\psi_j\}_{j=1}^{N_K}$:
\begin{center}
\begin{itemize}
	\item[] \textbf{Assumption:} $\mathrm{span}\{\psi\in{V}_{N_K}:N_K\in\mathbb{N}\}$ forms a core of $\mathcal{K}$.
\end{itemize}
\end{center}
This assumption means that the closure of the restriction of $\mathcal{K}$ to $\mathrm{span}\{\psi\in{V}_{N_K}:N_K\in\mathbb{N}\}$ is $\mathcal{K}$. In other words, the assumption allows us to study the infinite-dimensional operator $\mathcal{K}$, which acts on a possibly strict subdomain $\mathcal{D}(\mathcal{K})\subset L^2(\Omega,\omega)$, by considering its action on the observables in $V_{N_K}$. This is vital for capturing the full operator $\mathcal{K}$, and hence, equivalently, the full dynamics. To simplify our arguments, we also assume that $\psi_j$ is independent of $N_K$ for a fixed $j$, and that $V_{N_K}\subset V_{N_K+1}$. These last two assumptions can be dropped with minor modifications to our arguments.

We begin by re-writing the quantity $\|(\mathcal{K}-\lambda)^{-1}\|^{-1}$ as a minimal residual. Given a linear operator $T$ with domain $\mathcal{D}(T)$, we define the injection modulus $\mathcal{J}(T)=\inf\{\|T\psi\|:\|\psi\|=1,\psi\in\mathcal{D}(T)\}.$ It is well-known that
$$
\|(\mathcal{K}-\lambda)^{-1}\|^{-1}=\min\{\mathcal{J}(\mathcal{K}-\lambda),\mathcal{J}(\mathcal{K}^*-\overline{\lambda})\}\leq\mathcal{J}(\mathcal{K}-\lambda) =\inf_{g\in\mathcal{D}(\mathcal{K}), g\neq 0}\frac{\|(\mathcal{K}-\lambda)g\|}{\|g\|},
$$
with $\mathcal{J}(\mathcal{K}-\lambda)=\mathcal{J}(\mathcal{K}^*-\overline{\lambda})$ if $\lambda\notin\mathrm{\sigma}(\mathcal{K})$. Therefore ${\sigma}(\mathcal{K})=\{\lambda:\min\{\mathcal{J}(\mathcal{K}-\lambda),\mathcal{J}(\mathcal{K}^*-\overline{\lambda})\}=0\}$.

\subsection{The metric space}
Before discussing the convergence of \cref{alg:res_EDMD}, we must be precise about what convergence means. A reliable algorithm for computing a spectral set should converge locally on compact subsets of $\mathbb{C}$. In other words, it should converge to the complete spectral set and have no limiting points not in it. We will see that our algorithms achieve this, and we use the following metric space to quantify this convergence.

In general, $\sigma(\mathcal{K})$ is a closed and possibly unbounded subset of $\mathbb{C}$. We therefore use the Attouch--Wets metric defined by~\cite{beer1993topologies}:
\begin{equation}\label{eq:Attouch-Wets}
d_{\mathrm{AW}}(C_1,C_2)=\sum_{n=1}^{\infty} 2^{-n}\min\big\{{1,\underset{\left|x\right|\leq n}{\sup}\left|\mathrm{dist}(x,C_1)-\mathrm{dist}(x,C_2)\right|}\big\},
\end{equation}
where $C_1,C_2\in\mathrm{Cl}(\mathbb{C})$ and $\mathrm{Cl}(\mathbb{C})$ is the set of closed nonempty\footnote{We can add the empty set as an isolated point of this metric space to cover the case that $\sigma(\mathcal{K})$ may be empty. In what follows, this is implicitly done in our algorithms since the output is eventually the empty set in such circumstances.} subsets of $\mathbb{C}$. This metric generalizes the familiar Hausdorff metric on compact sets to unbounded closed sets and corresponds to local uniform converge on compact subsets of $\mathbb{C}$. If the spectral set is bounded, the two topologies and notions of convergence agree. An intuitive characterization of the Attouch--Wets topology we use in our proofs is as follows. For any closed nonempty sets $C$ and $C_n$, $d_{\mathrm{AW}}(C_n,C)\rightarrow{0}$ if and only if for any $\delta>0$ and $B_m(0)$ (closed ball of radius $m\in\mathbb{N}$ about $0$), there exists $N$ such that if $n>N$ then $C_n\cap B_m(0)\subset{C+B_{\delta}(0)}$ and $C\cap B_m(0)\subset{C_n+B_{\delta}(0)}$. 

\subsection{Computing approximate point pseudospectra}
We first compute the approximate point pseudospectrum, which is given by
\begin{equation}
\label{def:approx_pt_pseudo}
\sigma_{\epsilon,\mathrm{ap}}(\mathcal{K}):=\mathrm{cl}\left(\{\lambda\in\mathbb{C}:\mathcal{J}({\mathcal{K}}-\lambda)< \epsilon\}\right).
\end{equation}
In other words, $\sigma_{\epsilon,\mathrm{ap}}(\mathcal{K})$ is the closure of the set of all $\lambda$ such that there exists $g\in \mathcal{D}(\mathcal{K})$ of norm $1$ with $\|(\mathcal{K}-\lambda)g\|<\epsilon$. Such an observable $g$ is known as an ($\epsilon$-)approximate eigenfunction. Moreover, $\sigma_{\epsilon}(\mathcal{K})=\sigma_{\epsilon,\mathrm{ap}}(\mathcal{K})\cup \left\{\overline{\lambda}:\lambda\in\sigma_{\mathrm{p}}(\mathcal{K}^*)\right\}$. Thus $\sigma_{\epsilon,\mathrm{ap}}(\mathcal{K})$ is equivalent to the usual pseudospectrum $\sigma_{\epsilon}(\mathcal{K})$, up to the eigenvalues of $\mathcal{K}^*$. In most cases of interest, we have $\sigma_{\epsilon,\mathrm{ap}}(\mathcal{K})=\sigma_{\epsilon}(\mathcal{K})$.

To compute $\sigma_{\epsilon,\mathrm{ap}}(\mathcal{K})$, we begin by providing approximations of the function $\lambda\mapsto\mathcal{J}(\mathcal{K}-\lambda)$ using the residual in \eqref{eq:abs_res}. The following lemma shows how to do this.

\begin{lemma}
\label{techno2}
Define a relative residual as 
\begin{equation}
\label{the_key_dist_func}
\tau_{M,N_K}(\lambda,{\mathcal{K}})=\min_{\pmb{g}\in\mathbb{C}^{N_K}}\sqrt{\frac{\sum_{j,k=1}^{N_K}\overline{g_j}g_k\left[(\Psi_Y^*W\Psi_Y)_{jk} - \lambda(\Psi_Y^*W\Psi_X)_{jk} - \overline{\lambda}(\Psi_X^*W\Psi_Y)_{jk} + |\lambda|^2(\Psi_X^*W\Psi_X)_{jk}\right]}{\sum_{j,k=1}^{N_K}\overline{g_j}g_k(\Psi_X^*W\Psi_X)_{jk}}.}
\end{equation}
Then, assuming~\eqref{eq:convergenceMatrices} holds, 
\begin{equation}
\label{triv_limit}
\lim_{M\rightarrow\infty}\tau_{M,N_K}(\lambda,{\mathcal{K}})=\tau_{N_K}(\lambda,{\mathcal{K}}):=\min_{g\in V_{N_K}}\sqrt{\frac{\int_{\Omega}\left|[\mathcal{K}g](\pmb{x})-\lambda g(\pmb{x})\right|^2\, d\omega(\pmb{x})}{\int_{\Omega}\left|g(\pmb{x})\right|^2\, d\omega(\pmb{x})}}.
\end{equation}
Moreover, $\tau_{N_K}(\lambda,{\mathcal{K}})$ is a nonincreasing function of $N_K$ and converges uniformly down to $\mathcal{J}(\mathcal{K}-\lambda)$ on compact subsets of $\mathbb{C}$ as a function of the spectral parameter $\lambda$.
\end{lemma}
\begin{proof}
The limit \eqref{triv_limit} follows trivially from the convergence of matrices in~\eqref{eq:convergenceMatrices}. Since ${V}_{N_K}\subset {V}_{N_K+1}$, $\tau_{N_K}(\lambda,{\mathcal{K}})$ is a nonincreasing function of $N_K$. By definition, we also have that $\tau_{N_K}(\lambda,{\mathcal{K}})\geq \mathcal{J}({\mathcal{K}}-\lambda)$ for every $\lambda$ and $N_K$. Now let $0<\epsilon<1$ and choose $g\in\mathcal{D}({\mathcal{K}})$ of norm $1$ such that $\|({\mathcal{K}}-\lambda)g\|\leq \mathcal{J}({\mathcal{K}}-\lambda)+\epsilon.$ Since $\mathrm{span}\{\psi\in{V}_{N_K}:N_K\in\mathbb{N}\}$ forms a core of ${\mathcal{K}}$, there exists some $n$ and $g_{n}\in{V}_{n}$ such that $\|({\mathcal{K}}-\lambda)g_{n}\|\leq \|({\mathcal{K}}-\lambda)g\|+\epsilon$ and $\|g-g_{n}\|\leq \epsilon$. We find that $\|g_{n}\|\geq 1-\epsilon$ and hence that $\tau_{n}(\lambda,{\mathcal{K}})\leq{(\mathcal{J}({\mathcal{K}}-\lambda)+2\epsilon)}/({1-\epsilon}).$ Since this holds for any $0<\epsilon<1$, $\lim_{N_K\rightarrow\infty}\tau_{N_K}(\lambda,{\mathcal{K}})=\mathcal{J}({\mathcal{K}}-\lambda)$. Since $\mathcal{J}({\mathcal{K}}-\lambda)$ is Lipschitz continuous in $\lambda$ and hence continuous, $\tau_{N_K}(\lambda,{\mathcal{K}})$ converges uniformly down to $\mathcal{J}({\mathcal{K}}-\lambda)$ on compact subsets of $\mathbb{C}$ by Dini's theorem.
\end{proof}

In other words, in the large data limit, we can compute $\tau_{N_K}(\lambda,{\mathcal{K}})$, which is at least as large as $\mathcal{J}(\mathcal{K}-\lambda)$, and converges uniformly as our truncation parameter $N_K$ increases. To compute $\tau_{M,N_K}(\lambda,{\mathcal{K}})$, one can first diagonalize $\Psi_X^*W\Psi_X$ to compute $(\Psi_X^*W\Psi_X)^{-1/2}$ and then the problem reduces to computing a standard singular value decomposition. To turn $\tau_{N_K}(\lambda,{\mathcal{K}})$ into an approximation of $\sigma_{\epsilon,\mathrm{ap}}(\mathcal{K})$, we search for points where this function is less than $\epsilon$. Suppose that $\mathrm{Grid}(N_K)$ is a sequence of finite grids such that for any $\lambda\in\mathbb{C}$, $\lim_{N_K\rightarrow\infty}\mathrm{dist}(\lambda,\mathrm{Grid}(N_K))=0.$ For example, we could take
\begin{equation}
\label{eq:grid_def}
\mathrm{Grid}(N_K)=\frac{1}{N_K}\left[\mathbb{Z}+i\mathbb{Z}\right]\cap \{z\in\mathbb{C}:|z|\leq N_K\}.
\end{equation}
In practice, one considers a grid of points over the region of interest in the complex plane. We then define
$$
\Gamma^\epsilon_{N_K}(\mathcal{K}):=\left\{\lambda\in\mathrm{Grid}(N_K):\tau_{N_K}(\lambda,{\mathcal{K}})<\epsilon\right\}.
$$
Our algorithm is summarized in \cref{alg:res_EDMD}. Note that because of the nature of the nonlinear pencil in \eqref{the_key_dist_func}, one cannot use decompositions such as the generalized Schur decomposition to speed up the computational time. We can now state our convergence theorem, which says that the output of \cref{alg:res_EDMD} lies in $\sigma_{\epsilon,\mathrm{ap}}(\mathcal{K})$ and converges to $\sigma_{\epsilon,\mathrm{ap}}(\mathcal{K})$ as $M\rightarrow\infty$.  
\begin{theorem}
\label{half_pseudospectrum}
Let $\epsilon>0$. Then, $\Gamma^\epsilon_{N_K}(\mathcal{K})\subset\sigma_{\epsilon,\mathrm{ap}}(\mathcal{K})$ and
\begin{equation}
\label{conv_prop}
\lim_{N_K\rightarrow\infty}d_{\mathrm{AW}}\left(\Gamma^\epsilon_{N_K}(\mathcal{K}),\sigma_{\epsilon,\mathrm{ap}}(\mathcal{K})\right)=0.
\end{equation}
\label{thm:pseudospectrum}
In other words, we recover the full approximate point pseudospectrum without spectral pollution.
\end{theorem}

\begin{proof}
If $\lambda\in\Gamma^\epsilon_{N_K}({\mathcal{K}})$, then $\mathcal{J}({\mathcal{K}}-\lambda)\leq \tau_{N_K}(\lambda,{\mathcal{K}})<\epsilon$ by Lemma \ref{techno2}. It follows that $\Gamma^\epsilon_{N_K}({\mathcal{K}})\subset\sigma_{\epsilon,\mathrm{ap}}({\mathcal{K}})$. Now suppose that $\sigma_{\epsilon,\mathrm{ap}}({\mathcal{K}})\neq\emptyset$. It follows that there exists $\lambda\in\mathbb{C}$ with $\mathcal{J}({\mathcal{K}}-\lambda)< \epsilon$. Since $\mathcal{J}({\mathcal{K}}-\lambda)$ is continuous in $\lambda$ and $\lim_{N_K\rightarrow\infty}\mathrm{dist}(\lambda,\mathrm{Grid}(N_K))=0$, it follows that $\Gamma^\epsilon_{N_K}({\mathcal{K}})\neq\emptyset$ for large $N_K$. We use the characterization of the Attouch--Wets topology. Suppose that $m$ is large such that $B_m(0)\cap\sigma_{\epsilon,\mathrm{ap}}({\mathcal{K}})\neq\emptyset$. Since $\Gamma^\epsilon_{N_K}({\mathcal{K}})\subset\sigma_{\epsilon,\mathrm{ap}}({\mathcal{K}})$, we clearly have $\Gamma_{N_K}^{\epsilon}({\mathcal{K}})\cap B_m(0)\subset\sigma_{\epsilon,\mathrm{ap}}({\mathcal{K}})$. Hence, to prove \eqref{conv_prop}, we must show that given $\delta>0$, there exists $n_0$ such that if $N_K>n_0$ then $\sigma_{\epsilon,\mathrm{ap}}({\mathcal{K}})\cap B_m(0)\subset{\Gamma_{N_K}^{\epsilon}({\mathcal{K}})+B_{\delta}(0)}$.   Suppose for a contradiction that this statement is false. Then, there exists $\delta>0$, $\lambda_{n_j}\in\sigma_{\epsilon,\mathrm{ap}}({\mathcal{K}})\cap B_m(0)$, and $n_j\rightarrow\infty$ such that $\mathrm{dist}(\lambda_{n_j},\Gamma_{n_j}^{\epsilon}({\mathcal{K}}))\geq \delta$. Without loss of generality, we can assume that $\lambda_{n_j}\rightarrow \lambda\in\sigma_{\epsilon,\mathrm{ap}}({\mathcal{K}})\cap B_m(0)$. There exists some $z$ with $\mathcal{J}(\mathcal{K}-z)<\epsilon$ and $\left|\lambda-z\right|\leq \delta/2$. Let $z_{n_j}\in\mathrm{Grid}(n_j)$ such that $|z-z_{n_j}|\leq \mathrm{dist}(z,\mathrm{Grid}(n_j))+{n_j}^{-1}.$ In particular, note that $\lim_{n_j\rightarrow\infty}|z-z_{n_j}|=0$. Moreover,
\begin{equation}
\label{name7878769}
\tau_{n_j}(z_{n_j},{\mathcal{K}})\leq \left|\tau_{n_j}(z_{n_j},{\mathcal{K}})-\mathcal{J}(\mathcal{K}-z_{n_j})\right|+\left|\mathcal{J}(\mathcal{K}-z)-\mathcal{J}(\mathcal{K}-z_{n_j})\right|+\underbrace{\mathcal{J}(\mathcal{K}-z)}_{<\epsilon}.
\end{equation}
But, $\mathcal{J}(\mathcal{K}-\lambda)$ is continuous in $\lambda$ and $\tau_{n_j}$ converges uniformly to $\mathcal{J}(\mathcal{K}-\cdot)$ on compact subsets of $\mathbb{C}$ by Lemma \ref{techno2}. It follows that the first two terms on the right-hand side of \eqref{name7878769} must converge to zero. Hence for large $n_j$, $\tau_{n_j}(z_{n_j},{\mathcal{K}})<\epsilon$ so that $z_{n_j}\in\Gamma_{n_j}^{\epsilon}({\mathcal{K}})$. But
$
\left|z_{n_j}-\lambda\right|\leq \left|z-\lambda\right|+\left|z_{n_j}-z\right|\leq \delta/2 + |z-z_{n_j}|,
$
which is smaller than $\delta$ for large $n_j$, and we reach the desired contradiction.
\end{proof}

\cref{half_pseudospectrum} is stated for an algorithm that uses $\tau_{N_K}$. Strictly speaking, $\tau_{N_K}$ is only computed as $M\rightarrow\infty$ (see~\cref{techno2}). In practice, this is not an issue since we can often estimate the accuracy of the quadrature rule for the computed inner products by using known convergence rates as $M\rightarrow\infty$ (see \cref{sec:matrix_conv_galerkin}) or by comparing answers for different values of $M$. We can then combine these estimates and adaptively select $M$ based on $N_K$ so that $\tau_{N_K}$ is sufficiently accurate. For example, if $\tau_{N_K}$ is approximated to accuracy $\mathcal{O}(N_K^{-1})$ by $\tilde \tau_{N_K}$ we can set $\Gamma^\epsilon_{N_K}(\mathcal{K})=\{\lambda\in\mathrm{Grid}(N_K):\tilde \tau_{N_K}(\lambda,\mathcal{K})+N_K^{-1/2}<\epsilon\}$, where the $N_K^{-1/2}$ term is to ensure that it eventually dominates the quadrature error.

\subsection{Recovering the full spectrum and pseudospectrum}
In this subsection, we make the following additional assumption:
\begin{itemize}
	\item[] \textbf{Assumption:} $\mathrm{span}\{\psi\in{V}_{N_K}:N_K\in\mathbb{N}\}$ forms a core of $\mathcal{K}^*$.
\end{itemize}
If we have access to the matrix elements $\langle {\mathcal{K}}^*\psi_j,{\mathcal{K}}^*\psi_i\rangle$, then we can apply the techniques described in~\cref{sec:RES_DMD} to compute $\tau_{N_K}(\overline{\lambda},{\mathcal{K}}^*)$ in the large data limit $M\rightarrow\infty$. Define
$$
\hat\Gamma^\epsilon_{N_K}({\mathcal{K}})=\left\{\lambda\in\mathrm{Grid}(N_K):\min\{\tau_{N_K}(\lambda,{\mathcal{K}}),\tau_{N_K}(\overline{\lambda},{\mathcal{K}}^*)\}<\epsilon\right\}.
$$
The proof of~\cref{half_pseudospectrum} can be adapted to show that $\hat\Gamma^\epsilon_{N_K}({\mathcal{K}})\subset\sigma_{\epsilon}({\mathcal{K}})$, as well as the convergence $\lim_{N_K\rightarrow\infty}d_{\mathrm{AW}}(\hat \Gamma^\epsilon_{N_K}({\mathcal{K}}),\sigma_{\epsilon}({\mathcal{K}}))=0.$ Unfortunately, in general, one does not have access to approximations of $\langle {\mathcal{K}}^*\psi_j,{\mathcal{K}}^*\psi_i\rangle$. Nevertheless, one can overcome this.

To compute $\sigma_\epsilon(\mathcal{K})$, we note that because of \cref{half_pseudospectrum}, it is enough to compute $\sigma_{\epsilon,\mathrm{ap}}(\mathcal{K}^*)$ since $\sigma_{\epsilon}(\mathcal{K})=\sigma_{\epsilon,\mathrm{ap}}(\mathcal{K})\cup \{\overline{\lambda}:\lambda\in\sigma_{\epsilon,\mathrm{ap}}(\mathcal{K}^*)\}$. Let $N_1,N_2\in\mathbb{N}$, with $N_1\geq N_2$, and consider the matrices $\Psi_X$ and $\Psi_Y$ computed using $N_1$ dictionary functions. Assume that \eqref{eq:convergenceMatrices} holds and let
$$
L_{N_1}(\lambda)=\lim_{M\rightarrow\infty}\Psi_X^*W\Psi_Y(\Psi_X^*W\Psi_X)^{-1}\Psi_Y^*W\Psi_X-\lambda \Psi_Y^*W\Psi_X -\overline{\lambda}\Psi_X^*W\Psi_Y+|\lambda|^2 \Psi_X^*W\Psi_X\in\mathbb{C}^{N_1\times N_1}.
$$
Note that this matrix can be obtained from the given data in the large data limit $M\rightarrow\infty$. Let $\mathcal{P}_S$ denote the orthogonal projection onto a subspace $S\subset L^2(\Omega,\omega)$. If $g=\sum_{j=1}^{N_2}\psi_jg_j\in V_{N_2}$, then
\begin{equation}
\label{inner_prod_est}
\left\langle \mathcal{P}_{{V}_{N_1}}({\mathcal{K}}^*-\overline{\lambda})g,\mathcal{P}_{{V}_{N_1}}({\mathcal{K}}^*-\overline{\lambda})g\right\rangle=\sum_{j,k=1}^{N_2}\overline{g_j}g_k [L_{N_1}(\lambda)]_{jk}.
\end{equation}
In other words, this inner product is given by a square $N_2\times N_2$ truncation of the matrix $L_{N_1}(\lambda)\in\mathbb{C}^{N_1\times N_1}$. The following lemma uses the fact that \eqref{inner_prod_est} provides an increasingly more accurate approximation of the inner product $\left\langle ({\mathcal{K}}^*-\overline{\lambda})g,({\mathcal{K}}^*-\overline{\lambda})g\right\rangle$, as $N_1\rightarrow\infty$.

\begin{lemma}
\label{techno3}
Define $\smash{\upsilon_{N_1,N_2}(\lambda,{\mathcal{K}})=\min_{g\in V_{N_2},\|g\|=1}\sqrt{{\sum_{j,k=1}^{N_2}\overline{g_j}g_k [L_{N_1}(\lambda)]_{jk}}}}.$ Then, $\upsilon_{N_1,N_2}(\lambda,{\mathcal{K}})$ is a nondecreasing function of $N_1$ and, for fixed $N_2$, converges uniformly up to $\tau_{N_2}(\overline{\lambda},{\mathcal{K}}^*)$ on compact subsets of $\mathbb{C}$ as a function of $\lambda$ as $N_1\rightarrow\infty$.
\end{lemma}
\begin{proof}
It is clear that $\upsilon_{N_1,N_2}(\lambda,{\mathcal{K}})$ is a nondecreasing function of $N_1$. Since $\mathcal{P}_{{V}_{N_1}}$ converges strongly to the identity and $({\mathcal{K}}^*-\overline{\lambda})\mathcal{P}_{{V}_{N_2}}^*$ acts on a finite-dimensional subspace, we also have that $\lim_{N_1\rightarrow\infty}\|(\mathcal{P}_{{V}_{N_1}}-I)({\mathcal{K}}^*-\overline{\lambda})\mathcal{P}_{{V}_{N_2}}^*\|=0.$ The uniform convergence of $\upsilon_{N_1,N_2}(\lambda,{\mathcal{K}})$ up to $\tau_{N_2}(\overline{\lambda},{\mathcal{K}}^*)$ on compact subsets of $\mathbb{C}$ now follows from Dini's theorem.
\end{proof}

The following theorem shows that we have inclusion and convergence. 
\begin{theorem}
\label{Other_half_pseudospectrum}
Let $\epsilon> 0$ and define $\smash{\hat\Gamma^\epsilon_{N_2,N_1}({\mathcal{K}})=\left\{\lambda\in\mathrm{Grid}(N_2):\upsilon_{N_1,N_2}(\lambda,{\mathcal{K}})+1/{N_2}\leq\epsilon\right\}}.$ Then $\lim_{N_1\rightarrow\infty}\hat\Gamma^\epsilon_{N_2,N_1}({\mathcal{K}})=\hat\Gamma^\epsilon_{N_2}({\mathcal{K}})=\left\{\lambda\in\mathrm{Grid}(N_2):\tau_{N_2}(\overline{\lambda},{\mathcal{K}}^*)+1/{N_2}\leq\epsilon\right\}=:\hat\Gamma^\epsilon_{N_2}({\mathcal{K}}).$ Moreover, we have $\hat\Gamma^\epsilon_{N_2}({\mathcal{K}})\subset\sigma_{\epsilon,\mathrm{ap}}({\mathcal{K}}^*)$ and
$
\smash{\lim_{N_2\rightarrow\infty}d_{\mathrm{AW}}(\hat\Gamma^\epsilon_{N_2}({\mathcal{K}}),\left\{\overline{\lambda}:\lambda\in\sigma_{\epsilon,\mathrm{ap}}({\mathcal{K}}^*)\right\})=0.}
$
\end{theorem}

\begin{proof}
We prove that $\hat\Gamma^\epsilon_{N_2}({\mathcal{K}})=\lim_{N_1\rightarrow\infty}\hat\Gamma^\epsilon_{N_2,N_1}({\mathcal{K}})$ exists and is given by $\{\lambda\in\mathrm{Grid}(N_2):\tau_{N_2}(\overline{\lambda},{\mathcal{K}}^*)+1/{N_2}\leq\epsilon\}$. The rest of the proof carries over from~\cref{half_pseudospectrum} and its proof with minor modifications. (The additional term $N_2^{-1}$ is needed to ensure convergence to the pseudospectrum in the second limit - it is added so that we can use ``$\leq$'' instead of ``$<$'' in the definition of $\hat\Gamma^\epsilon_{N_2,N_1}$ to ensure the existence of the first limit via the following argument.) Fix $\lambda\in \mathrm{Grid}(N_2)$. If $\tau_{N_2}(\overline{\lambda},{\mathcal{K}}^*)+N_2^{-1}\leq\epsilon$, then by Lemma \ref{techno3}, $\upsilon_{N_1,N_2}(\lambda,{\mathcal{K}})+N_2^{-1}\leq\epsilon$. It follows that
$\hat\Gamma^\epsilon_{N_2}({\mathcal{K}})\subset \hat\Gamma^\epsilon_{N_2,N_1}({\mathcal{K}}).$ If $\tau_{N_2}(\overline{\lambda},{\mathcal{K}}^*)+N_2^{-1}>\epsilon$, then since $\lim_{N_1\rightarrow\infty}\upsilon_{N_1,N_2}(\lambda,{\mathcal{K}})=\tau_{N_2}(\overline{\lambda},{\mathcal{K}}^*)$, $\upsilon_{N_1,N_2}(\lambda,{\mathcal{K}})+N_2^{-1}>\epsilon$ for large $N_1$. The theorem now follows.
\end{proof}

Combining \cref{half_pseudospectrum} and \cref{Other_half_pseudospectrum}, we have the following corollary that shows how the full spectrum and pseudospectrum can be recovered.

\begin{corollary}
For any $\epsilon>0$,
$$
\lim_{N_2\rightarrow\infty} \lim_{N_1\rightarrow\infty}\left[\Gamma^\epsilon_{N_2}({\mathcal{K}})\cup\hat\Gamma^\epsilon_{N_2,N_1}({\mathcal{K}})\right]=\sigma_{\epsilon}({\mathcal{K}}).
$$
Additionally,
$$
\lim_{\epsilon\downarrow 0}\lim_{N_2\rightarrow\infty} \lim_{N_1\rightarrow\infty}\left[\Gamma^\epsilon_{N_2}({\mathcal{K}})\cup\hat\Gamma^\epsilon_{N_2,N_1}({\mathcal{K}})\right]=\sigma({\mathcal{K}}),
$$
where the convergence holds in the Attouch--Wets topology.
\end{corollary}

As a final remark, the reader may notice that a few of our algorithms require us to take several parameters successively to infinity. These limits do not generally commute, and it may be impossible to rewrite them in fewer limits. This is a generic feature of infinite-dimensional spectral problems~\cite{colbrookthesis} and has given rise to the Solvability Complexity Index~\cite{hansen2011solvability,colbrook2022foundations}. We do not go into the details but there are many open questions on the foundations of computing spectral properties of Koopman operators.

\section{Computing rational kernel convolutions with error control}
\label{ap_sec:inner_prod_err_control}
This section shows that the required convolutions with rational kernels can be computed rigorously with error control. Recall from~\eqref{unit_res_comp} that there are two types of inner products that we need to compute: (i) $\langle g,(\mathcal{K}-\lambda)^{-1}g \rangle$, and (ii) $\langle (\mathcal{K}-\lambda)^{-1}g,\mathcal{K}^*g \rangle$ for some observable $g$. Here, we show that the inner products can be computed with error control via the right-hand sides of \eqref{eq:ra_kern_inner_prod_type1} and \eqref{eq:ra_kern_inner_prod_type2} by adaptively selecting the truncation size $N_K$ for a fixed smoothing parameter $\epsilon$. Without loss of generality, we assume that $g$ is normalized so that $\|g\|=1$ and we denote $\tilde g_{N_K}=\Psi \pmb{a}$. To compute the error bounds in the large data limit, we also use the fact that
$$
h_{N_K}=\sum_{j=1}^{N_K}\left[(\Psi_X^*W\Psi_Y-\lambda \Psi_X^*W\Psi_X)^{-1}\Psi_X^*W\Psi_X \pmb{a}\right]_j\psi_j=\Psi\pmb{h}\approx(\mathcal{P}_{{V}_N}{\mathcal{K}}\mathcal{P}_{{V}_N}^*-\lambda I_N)^{-1} g_{N_K}.
$$
The following theorem shows that we can bound the error of using the right-hand sides of \eqref{eq:ra_kern_inner_prod_type1} and \eqref{eq:ra_kern_inner_prod_type2} to approximate the desired inner products.
\begin{theorem}
\label{inner_prod_error_cont}
Let $\lambda\in\mathbb{C}$ with $|\lambda|>1$. Consider the setup of \cref{sec:comput_res} with $\|g\|=1$ and let
\begin{align*}
\delta_1({N_K})&=\|g-g_{N_K}\|,\\
\delta_2({N_K})&=\|(\mathcal{K}-\lambda)(\mathcal{P}_{{V}_{N_K}}{\mathcal{K}}\mathcal{P}_{{V}_{N_K}}^*-\lambda I_{N_K})^{-1}g_{N_K}- g_{N_K}\|,\\
\delta_3({N_K})&=\|(\mathcal{P}_{{V}_{N_K}}{\mathcal{K}}\mathcal{P}_{{V}_{N_K}}^*-\lambda I_{N_K})^{-1} g_{N_K}\|.
\end{align*}
Then, $\lim_{{N_K}\rightarrow\infty}\delta_1({N_K})=\lim_{{N_K}\rightarrow\infty}\delta_2({N_K})=0$ and $\delta_3({N_K})$ remains bounded as ${N_K}\rightarrow\infty$. Moreover,
\begin{align*}
\left|\left\langle g_{N_K},(\mathcal{P}_{{V}_{N_K}}{\mathcal{K}}\mathcal{P}_{{V}_{N_K}}^*-\lambda I_{N_K})^{-1} g_{N_K}\right\rangle-\left\langle g,(\mathcal{K}-\lambda)^{-1}g\right\rangle\right|&\leq\frac{\delta_1({N_K})+\delta_2({N_K})}{|\lambda|-1}+\delta_1({N_K})\delta_3({N_K}),\\
\left|\left\langle (\mathcal{P}_{{V}_{N_K}}{\mathcal{K}}\mathcal{P}_{{V}_{N_K}}^*-\lambda I_{N_K})^{-1} g_{N_K},\mathcal{K}^* g_{N_K}\right\rangle-\left\langle (\mathcal{K}-\lambda)^{-1}g,\mathcal{K}^*g\right\rangle\right|&\leq\frac{\delta_1({N_K})+\delta_2({N_K})}{|\lambda|-1}+\delta_1({N_K})\delta_3({N_K}).
\end{align*}
Assuming~\eqref{eq:convergenceMatrices}, $\delta_2({N_K})$ and $\delta_3({N_K})$ can be computed in the large data limit via
\begin{align}
\delta_2({N_K})^2&=\lim_{M\rightarrow\infty}\pmb{h}^*\Psi_X^*W\Psi_X\pmb{h}\cdot\mathrm{res}(\lambda,h_{N_K})^2-2\mathrm{Re}\big(\pmb{a}^*(\Psi_X^*W\Psi_Y-\lambda \Psi_X^*W\Psi_X)\pmb{h}\big)+\pmb{a}^*\Psi_X^*W\Psi_X\pmb{a},\label{compute11}\\
\delta_3({N_K})^2&=\lim_{M\rightarrow\infty}\pmb{h}^*\Psi_X^*W\Psi_X\pmb{h}.\label{compute12}
\end{align}
\end{theorem}
\begin{proof}
First, by assumption, $\lim_{{N_K}\rightarrow\infty}\delta_1({N_K})=0$. Using this, together with \cref{res_square_truncate} and the fact that $\mathcal{K}$ is bounded, it follows that $\lim_{{N_K}\rightarrow\infty}\delta_2({N_K})=0$. Similarly, $\delta_3({N_K})$ remains bounded as ${N_K}\rightarrow\infty$. Since $|\lambda|>1$ and $\|\mathcal{K}\|=1$, the arguments in \cref{sec:stability_argument} show that $\|(\mathcal{K}-\lambda)^{-1}\|\leq {1}/({|\lambda|-1}).$ It follows that
\begin{align*}
\|(\mathcal{P}_{{V}_{N_K}}{\mathcal{K}}\mathcal{P}_{{V}_{N_K}}^*-\lambda I_{N_K})^{-1} g_{N_K}-(\mathcal{K}-\lambda)^{-1}g\|&\leq \frac{\|(\mathcal{K}-\lambda)(\mathcal{P}_{{V}_{N_K}}{\mathcal{K}}\mathcal{P}_{{V}_{N_K}}^*-\lambda I_{N_K})^{-1} g_{N_K}-g\|}{|\lambda|-1}\\
&\hspace{-10mm}\leq \frac{\|(\mathcal{K}-\lambda)(\mathcal{P}_{{V}_{N_K}}{\mathcal{K}}\mathcal{P}_{{V}_{N_K}}^*-\lambda I_{N_K})^{-1} g_{N_K}- g_{N_K}\|+\|g- g_{N_K}\|}{|\lambda|-1}\\
&\hspace{-10mm}\leq \frac{\delta_1({N_K})+\delta_2({N_K})}{|\lambda|-1}.
\end{align*}
We now re-write the inner product
\begin{align*}
\left\langle g_{N_K},(\mathcal{P}_{{V}_{N_K}}{\mathcal{K}}\mathcal{P}_{{V}_{N_K}}^*-\lambda I_{N_K})^{-1} g_{N_K}\right\rangle=&\left\langle g,(\mathcal{K}-\lambda)^{-1}g\right\rangle+\left\langle g_{N_K}-g,(\mathcal{P}_{{V}_{N_K}}{\mathcal{K}}\mathcal{P}_{{V}_{N_K}}^*-\lambda I_{N_K})^{-1} g_{N_K}\right\rangle\\
&+\left\langle g,(\mathcal{P}_{{V}_{N_K}}{\mathcal{K}}\mathcal{P}_{{V}_{N_K}}^*-\lambda I_{N_K})^{-1} g_{N_K}-(\mathcal{K}-\lambda)^{-1}g\right\rangle.
\end{align*}
Since $\|g\|=1$, it follows by the Cauchy--Schwarz inequality that
\begin{equation*}
\begin{split}
\left|\left\langle g_{N_K},(\mathcal{P}_{{V}_{N_K}}{\mathcal{K}}\mathcal{P}_{{V}_{N_K}}^*-\lambda I_{N_K})^{-1} g_{N_K}\right\rangle-\left\langle g,(\mathcal{K}-\lambda)^{-1}g\right\rangle\right|&\leq \frac{\delta_1({N_K})+\delta_2({N_K})}{|\lambda|-1}\\
&\quad+\delta_1({N_K})\left\|(\mathcal{P}_{{V}_{N_K}}{\mathcal{K}}\mathcal{P}_{{V}_{N_K}}^*-\lambda I_{N_K})^{-1}g_{N_K}\right\|.
\end{split}
\end{equation*}
Similarly, since $\|\mathcal{K}^*\|\leq 1$, we obtain
\begin{equation*}
\begin{split}
\left|\left\langle (\mathcal{P}_{{V}_{N_K}}{\mathcal{K}}\mathcal{P}_{{V}_{N_K}}^*-\lambda I_{N_K})^{-1} g_{N_K},\mathcal{K}^* g_{N_K}\right\rangle-\left\langle (\mathcal{K}-\lambda)^{-1}g,\mathcal{K}^*g\right\rangle\right|&\leq \frac{\delta_1({N_K})+\delta_2({N_K})}{|\lambda|-1}\\
&\hspace{-5mm}+\delta_1({N_K})\left\|(\mathcal{P}_{{V}_{N_K}}{\mathcal{K}}\mathcal{P}_{{V}_{N_K}}^*-\lambda I_{N_K})^{-1} g_{N_K}\right\|.
\end{split}
\end{equation*}
This proves the upper bounds of the theorem. The equalities in \eqref{compute11} and \eqref{compute12} follow from expanding the inner products and from the definition of $\pmb{h}$.
\end{proof}

From~\cref{inner_prod_error_cont} we see that if $\delta_1({N_K})$ is known, then we can compute the inner products $\langle g,(\mathcal{K}-\lambda)^{-1}g \rangle$ and $\langle (\mathcal{K}-\lambda)^{-1}g,\mathcal{K}^*g \rangle$ to any desired accuracy by adaptively selecting ${N_K}$. Moreover, the error bounds in \cref{inner_prod_error_cont} can be computed in $\mathcal{O}(N_K^2)$ operations without the generalized Schur decomposition. In practice, we increase $N_K$, keeping track of the computed inner product until convergence is apparent.

\linespread{0.9}\selectfont{}
\small
\bibliographystyle{abbrv}
\bibliography{bib_file}
\end{document}